\documentclass[11pt]{article}
\title{The Heavy-tailed Frog Model}
\author{
    Omer Angel
    \thanks{University of British Columbia, Vancouver, Canada.  Research was supported in part by NSERC, SLMath Clay senior scholarship, and Magdalen college, Oxford. Email: angel@math.ubc.ca } \and 
    Jonathan Hermon 
    \thanks{University of British Columbia, Vancouver, Canada. Research was supported in part by NSERC. Email: jhermon@math.ubc.ca } \and 
    Yuliang Shi
    \thanks{University of British Columbia, Vancouver, Canada. Research was supported in part by NSERC. Email: yuliang@math.ubc.ca }}

\usepackage{amsmath,amssymb,amsfonts,amsthm,bbm,mathabx,mathtools,aliascnt}
\usepackage{tikz}
\usetikzlibrary{arrows.meta,decorations.pathmorphing,calc}
\usepackage{enumitem}
\usepackage{graphicx}
\usepackage{color}
\usepackage[margin=3cm]{geometry}
\usepackage[protrusion=true,expansion=true]{microtype}
\usepackage[font=sf, labelfont={sf,bf}, margin=0.5cm]{caption}
\usepackage[
  colorlinks=true,
  linkcolor=black,
  citecolor=black,
  urlcolor=blue,
  pdfborder={0 0 0}
]{hyperref}
\usepackage[nameinlink]{cleveref}
\usepackage{hyperref}

\newtheorem{thm}{Theorem}[section]

\newaliascnt{proposition}{thm}
\newaliascnt{lemma}{thm}
\newaliascnt{corollary}{thm}
\newaliascnt{claim}{thm}
\newaliascnt{conjecture}{thm}
\newaliascnt{definition}{thm}
\newaliascnt{example}{thm}
\newaliascnt{remark}{thm}

\newtheorem{theorem}[thm]{Theorem}
\newtheorem{proposition}[proposition]{Proposition}
\newtheorem{lemma}[lemma]{Lemma}
\newtheorem{corollary}[corollary]{Corollary}
\newtheorem{claim}[claim]{Claim}
\newtheorem{conjecture}[conjecture]{Conjecture}

\theoremstyle{definition}
\newtheorem{definition}[definition]{Definition}

\theoremstyle{remark}
\newtheorem{remark}[remark]{Remark}

\aliascntresetthe{proposition}
\aliascntresetthe{lemma}
\aliascntresetthe{corollary}
\aliascntresetthe{claim}
\aliascntresetthe{conjecture}
\aliascntresetthe{definition}
\aliascntresetthe{example}
\aliascntresetthe{remark}

\crefname{thm}{Theorem}{Theorems}
\crefname{theorem}{Theorem}{Theorems}
\crefname{proposition}{Proposition}{Propositions}
\crefname{lemma}{Lemma}{Lemmas}
\crefname{corollary}{Corollary}{Corollaries}
\crefname{claim}{Claim}{Claims}
\crefname{definition}{Definition}{Definitions}
\crefname{example}{Example}{Examples}
\crefname{remark}{Remark}{Remarks}
\crefname{conjecture}{Conjecture}{Conjectures}

\Crefname{thm}{Theorem}{Theorems}
\Crefname{theorem}{Theorem}{Theorems}
\Crefname{proposition}{Proposition}{Propositions}
\Crefname{lemma}{Lemma}{Lemmas}
\Crefname{corollary}{Corollary}{Corollaries}
\Crefname{claim}{Claim}{Claims}
\Crefname{definition}{Definition}{Definitions}
\Crefname{example}{Example}{Examples}
\Crefname{remark}{Remark}{Remarks}
\Crefname{conjecture}{Conjecture}{Conjectures}

\crefname{section}{Section}{Sections}
\crefname{equation}{equation}{equations}
\crefname{figure}{Figure}{Figures}
\crefname{table}{Table}{Tables}

\renewcommand{\P}{\mathbb P}
\newcommand{\Z}{\mathbb Z}
\newcommand{\E}{\mathbb E}
\newcommand{\R}{\mathbb R}
\newcommand{\N}{\mathbb N}

\DeclareMathOperator{\Bern}{Bernoulli}

\DeclareMathOperator{\Poi}{Poi}

\newtheorem*{lemma*}{Lemma}

\newcommand{\bfP}{\mathbf P}

\newcommand{\bP}{\mathbb P}
\newcommand{\bQ}{\mathbb Q}
\newcommand{\bE}{\mathbb E}
\newcommand{\bfE}{\mathbf E}

\newcommand{\T}{\mathbb{T}}

\newcommand{\cW}{\mathcal{W}}

\newcommand{\cU}{\mathcal{U}}
\newcommand{\cV}{\mathcal{V}}
\newcommand{\cR}{\mathcal{R}}

\newcommand{\cF}{\mathcal{F}}
\newcommand{\1}{\mathbbm{1}}
\newcommand{\Pois}{\text{Pois}}

\renewcommand{\epsilon}{\varepsilon}

\begin{document}
\maketitle

\begin{abstract}
We study the frog model on $\Z^d$ and on the discrete tori $\T_L^d$, $d\ge 2$, with a symmetric, translation-invariant, and heavy-tailed transition kernel satisfying
\[
Q(x,y)\asymp |x-y|^{-(d+\alpha)},
\qquad \alpha>0.
\]
Starting from an i.i.d.\ Poisson$(\lambda)$ number of sleeping particles per site and one active particle at the origin. Active particles perform independent $Q$-random walks and activate the particles they encounter.

We first determine the timescale for activating distant vertices. 
When $\alpha\in(0,d)$, the time required to activate all vertices within distance $L$ of the origin is, with high probability,
\[
(\log L)^{\Delta+o(1)},
\qquad
\Delta^{-1}:=\log_2\left(\frac{2d}{d+\alpha}\right),
\]
as $L\to\infty$. This polylogarithmic spreading contrasts sharply with the linear spreading of the classical frog model driven by simple random walks; see \cite{MR1910638, MR2060478}. When $\alpha>d$, we recover this classical linear behavior by proving matching linear upper and lower bounds; at $\alpha=d$, we prove a linear upper bound.

Finally, we consider the finite-lifespan model on $\T_L^d$, in which each particle is removed after taking $\ell$ steps. We show that the cover lifespan, defined as the smallest $\ell$ for which the torus is entirely activated, is asymptotic to the cover time of a Poisson$(\lambda L^d)$ cloud of independent stationary random walkers.
\end{abstract}

\newpage

\setcounter{tocdepth}{2}
\tableofcontents

\newpage

%%%%%%%%%%%%%%%%%%%%%%%%%%%%%%%%%%%%%%%%%%%%%%%%%%%%%%%%%%%%%%%%%%
\section{Introduction}

This paper focuses on a specific type of reaction-diffusion system, known as the \textbf{frog model}, which can be categorized as a special case of the ``$A + B \rightarrow 2B$'' model (see \cite{MR2184100, MR2247840, MR2415386}). 
Such models provide a framework for analyzing the spread of infectious diseases and the diffusion of information across populations.
The frog model on $\Z^d$ consists of particles sitting on the vertices of $\Z^d$, which can be either active or inactive. 
Active particles perform random walks on $\Z^d$, while inactive particles remain stationary. 
When an active particle encounters an inactive particle, the inactive particle becomes active. 
The frog model starts with Poisson($\lambda$) inactive particles independently at each vertex, where $\lambda$ is called the \textbf{particle density}. 
At time zero, only particles at the origin are activated. 
Sometimes, we include an active \textbf{planted particle} at the origin $o = (0, \ldots, 0)$ to avoid the possibility of immediate extinction. 
This paper focuses on two versions of the frog model: the \textbf{SIR} (Susceptible, Infected, and Removed) frog model with finite \textbf{particle lifespan} and the \textbf{SI} frog model with infinite lifespan. 
The lifespan $\ell > 0$ is the number of steps an active particle takes before being \textbf{removed} from the lattice. 
The lifespan $\ell = \infty$ means that the active particles are never removed, which is the case for the SI frog model.

Classical studies of the frog model have focused primarily on frogs performing simple random walks (see \cite{MR1943889,MR1910638,MR2042394,MR2060478,MR1742145}), with a few papers considering asymmetric walks (see \cite{MR3283610,MR3668379}).
The novelty of this paper lies in the introduction of heavy-tailed walks into the frog model, characterized by a power-law transition kernel. This enables particles to traverse long distances in just a few steps. One of the main contributions of this paper is to demonstrate that the introduction of heavy-tailed walks significantly accelerates the spread of active particles on the hypercubic lattice. In the senario of infection spreading, this implies that, under the influence of long distance traveller, the infection can spread much faster than the simple random walk case.

To quantify the result, we introduce the \textbf{activation time} of a vertex $x \in \Z^d$, denoted by $\mathrm{AT}(x)$, which represents the minimal time needed for an active particle to visit site $x$. If site $x$ remains unvisited before the process extinguishes, then $\mathrm{AT}(x)$ is deemed infinite. 
The main question is how $\mathrm{AT}(x)$ scales with the graph distance from the origin. 
The concept of activation time, also known as passage time in some literature, has been extensively studied in works such as \cite{MR1893139}, \cite{10.1214/18-EJP144}, and \cite{MR4031108}. In particular, results from \cite{MR1910638,MR1893139} established a shape theorem for the simple random walk frog model on $\Z^d$, demonstrating that the activation time grows linearly with the graph distance from the origin.
Let $|\cdot|$ denote a norm on $\Z^d$. Itai Benjamini conjectured in private communications that the heavy-tailed frog model on $\Z^d$ with random walk one-step transition kernel $Q(x, y) \asymp |x - y|^{-(d + \alpha)}$ for $\alpha \in (0, d)$ would exhibit a wildly different behavior. Specifically, he proposed that the activation time of a vertex $x$ would grow as a polylogarithmic function of the norm of $x$. This paper confirms Benjamini's conjecture.
This conjecture was motivated by established results in long-range percolation literature on $\Z^d$ (see \cite{MR2850269, MR2663638}). In long-range percolation with power-law exponent $s \in (d, 2d)$, the chemical distance between vertices $x$ and $y$ is proportional to $(\log |x - y|)^{\Delta}$, where $\Delta > 0$ is a constant determined by the dimension $d$ and the power-law exponent $s$.
The frog model can be viewed as a directed long-range percolation model with dependencies, where a directed edge exists from $x$ to $y$ if a particle starting at $x$ visits $y$. 
While this coupling with dependent long-range percolation is helpful conceptually, the frog model presents additional complexity due to the unbounded range of dependency among directed edges. The dependency becomes more substantial and complex when the particle lifespan is large or infinite. 

A parallel line of research in this paper focuses on the study of the \textbf{cover lifespan} on finite graphs. The cover lifespan represents the minimum particle lifespan needed for active particles to visit every vertex. On finite graphs with irreducible transition kernels, the cover lifespan is an almost surely finite random variable. A rigorous definition is provided in Section \ref{sec:preliminary}. In epidemiological contexts, the cover lifespan represents the minimum duration an infected individual must remain infectious to ensure the disease spreads throughout the entire population. In the frog model literature, this quantity is also known as ``Susceptibility". For instance, \cite{MR4068310} studies the susceptibility of the frog model on tori with simple random walks and demonstrates that the cover lifespan is asymptotically equivalent to the minimum time needed for all vertices to be visited when all frogs are initially active. This paper extends these findings by showing that this equivalence holds even when simple random walks are replaced by heavy-tailed walks.

\vspace{2mm}

\textbf{Notation. \,}
The following notation is needed to discuss the main results. We also use them throughout the paper.
\begin{itemize}[label={}, left=0.2cm, itemsep=3pt]
\item[-] $\mathrm{AT}(o, x)$ is an integer-valued random variable, which denotes the smallest time at which vertex $x$ is visited by an active particle. Let
\begin{align*}
	\mathcal{A}(k) := \{ x : \mathrm{AT}(o, x) < k \},
\end{align*}
denote the set of vertices activated by time $k \in [1, +\infty]$.
% \item[-] Let $\Z^d$ be the $d$-dimensional integer lattice.
% \item[-] Let $B_x(L) := \{ y \in \Z^d : x_i - L/2 \leq y_i < x_i + L/2, i = 1, ..., d \}$. Let $B(L) := B_o(L)$.
\item[-] Let $\T^d_L := \Z^d / L\Z^d$ be the $d$-dimensional discrete torus with side length $L$. 
\item[-] Let $\bP$ and $\bP_L$ be the probability measures for the frog model on $\Z^d$ and $\T^d_L$. We suppress the dependency on $\lambda$, $\ell$, and $Q$ since they will be clear from the context. 
\item[-] Let $\bP^+$ and $\bP_L^+$ be the same measures but with the modification that, at time 0, we add one extra particle, i.e. the \textbf{planted particle}, at the origin. 
\item[-] Let $\bfP_x$ be the measure for a single discrete-time random walk starting at $x$ on $\Z^d$ whose transition kernel will be clear from the context. 
\item[-] For $x, y \in \Z^d$ or $\T^d_L$, $|x-y|$ denotes their $\ell_\infty$-distance.\footnote{On $\T^d_L$, we define
\[
|u-v| := \max_{1\le i\le d} \min\{ |u_i-v_i|,\; L-|u_i-v_i| \}.
\]
} We use the $\ell_\infty$-metric only for technical convenience.
\item[-] For functions $f, g$ on $\R$:
\begin{itemize}[label={}, left=0.2cm, itemsep=3pt]
    \item[-] we write $f(x) = O(g(x))$ if there exist $C > 0$ and $x_0$ such that for any $x \geq x_0$: 
    $$0 \leq f(x) \leq Cg(x).$$
    \item[-] we write $f(x) = \Omega(g(x))$ if there exist $c > 0$ and $x_0$ such that for any $x \geq x_0$: 
    $$f(x) \geq cg(x) \geq 0.$$
    \item[-] we write $f(x) = \Theta(g(x))$ if $f(x) = O(g(x))$ and $f(x) = \Omega(g(x))$.
\end{itemize}
\item[-] For functions $f, g$ on $\Z^d \times \Z^d$: we write $f(x, y) = O(g(x, y))$ if there exist $C > 0$ and $r_0$ such that $0 \leq f(x, y) \leq Cg(x, y)$ for any $x, y$ with $|x - y| \geq r_0.$
\item[-] For functions $f, g$ on $\R$ or $\Z^d \times \Z^d$, we write $f(x) \asymp g(x)$ if we have both $f(x) = O(g(x))$ and $g(x) = O(f(x))$ as $x \rightarrow \infty$. Moreover, $\asymp_{\beta}$ means that the constants $C, c$ depend on $\beta$.  
\item[-] Throughout this paper, $C, c$, $C_0, c_0$, $C_1, c_1, \dots$ denote positive constants that generally depend on $\lambda$ and $Q$ unless specified otherwise. The exact values of these constants may change from one expression to another, even within the same proof.
\end{itemize}

%%%%%%%%%%%%%%%%%%%%%%%%%%%%%%%%%%%%%%%%%%%%%%%%%%%%%%%%%%%%%%%%%%
\section{Main Results}

In this section, we present our main results concerning the \textbf{activation time} and the \textbf{cover lifespan} for a frog model where particles execute heavy-tailed random walks on a $d$-dimensional hypercubic lattice. 

Throughout this paper, the particles are assumed to perform discrete-time random walk. And, the particle lifespan is always taken as a positive integer.
Let \( Q(\cdot, \cdot) \) denote a symmetric translation-invariant random walk transition kernel on \( \mathbb{Z}^d \times \mathbb{Z}^d \), i.e. for any $x, y \in \mathbb{Z}^d$, we have \( Q(x, y) = Q(y, x) = Q(0, x{-}y) \). We say $Q$ is heavy-tailed with \textbf{exponent} \( \alpha \) if there exists a constant \( K > 0 \) such that
\begin{align}
    \frac{K^{-1}}{|x - y|^{d + \alpha}} \le Q(x, y) \le \frac{K}{|x - y|^{d + \alpha}}
\end{align}
holds for all \( x, y \in \mathbb{Z}^d \).
Throughout this paper, our primary assumption is that $Q$ is a heavy-tailed transition matrix with positive exponent $\alpha$.

Let $L$ be a positive integer. Aside from $\Z^d$, we also consider the frog model on the $d$-dimensional discrete torus $\T^d_L$. Let $Q_{L}$ denote the standard projection of $Q$ onto the discrete torus $\T^d_L$. Since $Q$ is translation-invariant, $Q_L$ is well-defined. 

\subsection{Bounds on the Activation Time}

Our study of the activation time begins with the case where the particle lifespan is a fixed constant $\ell \in \mathbb{N}$. For any $\rho \in (0, 1)$, we investigate how long it takes for the frog model to activate a $\rho$ fraction of vertices on $\T^d_L$ as $L \to \infty$. We discover two distinct regimes based on the exponent $\alpha$. When $\alpha \in (0, d)$ and $\ell$ is sufficiently large, we establish that the activation time is of order $O((\log L)^{\Delta + o(1)})$ for some $\Delta > 0$ that depends only on $\alpha$ and $d$. When $\alpha > d$, the activation time is asymptotically linear in $L$. In the critical case where $\alpha = d$, we prove a linear upper bound and conjecture sub-linear polynomial growth. 

We now proceed to state our results and conjectures precisely. Define $\Delta > 0$ as
\begin{align*}
\Delta^{-1} := \log_2(2d/(\alpha + d)).
\end{align*}

\begin{theorem}
  \label{Mainthm: activate_constant_fraction}
Let $d \ge 2$. Assume $Q$ is a symmetric, translation-invariant, and heavy-tailed transition kernel on $\Z^d$ with exponent $\alpha > 0$, and let $Q_L$ be its standard projection onto $\T^d_L$. Consider the frog model on $\T^d_L$ equipped with $Q_L$. Assume the lifespan $\ell_o$ of the planted particle diverges as $L$ approaches infinity.
\begin{enumerate}[itemsep=5pt, left=0.3cm]
\item If $\alpha \in (0, d)$, then for any particle density $\lambda > 0$, $\rho \in (0, 1)$, and $\epsilon > 0$, there exists a positive constant $C_0$ such that
  \begin{align*}
    \lim_{L \to \infty} \bP_L^+\left(\frac{|\mathcal{A}\left((\log L)^{\Delta + \epsilon}\right)|}{L^d} \geq \rho\right) = 1,
  \end{align*}
  for any particle lifespan $\ell > C_0$.

\item If $\alpha \ge d$, then for any $\rho \in (0, 1)$, there exist positive constants $C$ and $C_0$ such that
  \begin{align*}
    \lim_{L \to \infty} \bP_L^+\left(\frac{|\mathcal{A}(CL)|}{L^d} \geq \rho\right) = 1
  \end{align*}
  for any particle lifespan $\ell > C_0$.
\end{enumerate}
\end{theorem}

The proof strategy consists of two steps while each step uses one half of the particles. The first step constructs a supercritical Bernoulli site percolation on a renormalized torus, which is used to obtain (2) in Theorem \ref{Mainthm: activate_constant_fraction} for all $\alpha > 0$. The second step constructs a directed long-range percolation on $\T^d_L$. The long-range percolation is used to reduce the activation time to a poly-logarithmic function of the side length when $\alpha \in (0, d)$, which is the content of (1) in Theorem \ref{Mainthm: activate_constant_fraction}. Those two parts are presented in detail in Section \ref{Sec: BernPerc} and Section \ref{Sec: Backbone} respectively.

We conclude with two important remarks regarding Theorem \ref{Mainthm: activate_constant_fraction}. First, our strategy only works when $d \ge 2$: while the second part of the proof extends to the case $d=1$, the first part relies on the existence of an infinite cluster from a Bernoulli site percolation. This fails when $d=1$. Second, our construction of a supercritical Bernoulli site percolation on $\Z^d$ necessitates choosing a sufficiently large particle lifespan $\ell$. In essence, our result lies in the ``supercritical" regime of the frog model. For a comprehensive analysis of phase transitions in the frog model on $\Z^d$ and general vertex-transitive graphs, we refer readers to \cite{MR1943889, angel2025existencesharpnessphasetransition}. We would like to point out that the renormalization argument is somewhat robust and in \cite{angel2025existencesharpnessphasetransition}, we managed to extend it to any vertex-transitive graphs with polynomial growth and establish the existence of a non-trivial phase transition for the frog model on such graphs. 

Our results on the activation time are inspired by existing literature on the metric property of the long-range percolation clusters on $\Z^d$. Assume the connection probability between two vertices of distance $r$ is proportional to $r^{-s}$. Then the graph-theoretical (a.k.a. chemical) distance between $o$ and $x$ exhibits five regimes: 
\begin{itemize}[left=0.3cm, itemsep=3pt]
\item[1.] $s < d$: the chemical distance approaches to a deterministic number $\lceil \frac{d}{d - s} \rceil$ as the Euclidean distance between $o$ and $x$, denoted by $|x|$, goes to infinity (see \cite{MR2123930});
\item[2.] $s = d$: the chemical distance grows as $\log |x| / \log \log |x|$ (see \cite{MR1913075});
\item[3.] $d < s < 2d$: the chemical distance is asymptotic to $\phi(|x|) (\log |x|)^{\Delta}$, where the function $\phi$ is not a constant and it describes the oscillation of the ratio between the chemical distance and $(\log |x|)^{\Delta}$ (see \cite{MR4757496});
\item[4.] $s = 2d$: the chemical distance grows like $\Theta(|x|^{\theta})$ as $|x| \rightarrow \infty$ for some $\theta \in (0, 1)$ (see \cite{MR4677584, MR4975504}); 
\item[5.] $s > 2d$: the chemical distance divided by $|x|$ converges to a constant (see \cite{berger2004lower}). 
\end{itemize}

In the frog model, the parameter comparable to $s$ is $\alpha + d$. The intuition comes from the following observation: we can think of the frog model as a dependent directed long-range percolation model, where we say there is a directed edge going from $x$ to $y$ if a particle initially sitting on $x$ visits $y$ in its lifespan $\ell$. By Poisson thinning, the number of particles starting from $x$ and visiting $y$ in their lifespan $\ell$ is a Poisson random variable with mean $\lambda \bfP_{x}(\tau_y \le \ell)$, where $\tau_y$ is the hitting time of $y$. Let $f(x, y) := \bfP_{x}(\tau_y \le \ell)$, then for fixed $\ell$, it is straightforward to verify that $f(x, y) \asymp_{\ell, K} |x - y|^{-(\alpha + d)}$. Since $\alpha$ can only be positive, the relevant cases for us would be the cases 3., 4., and 5. 

Note that in Theorem~\ref{Mainthm: activate_constant_fraction} we require the particle lifespan $\ell$ to be sufficiently large. We further conjecture that the same conclusions hold for any lifespan $\ell$ for which the frog model is supercritical; that is, without the planted particle, the process survives with positive probability (equivalently, with positive probability there is at least one active particle for all $t \ge 0$). More precisely, consider the SIR frog model on $\Z^d$ in which each frog performs a continuous-time random walk with jump rate $1$ and has lifespan $\ell \in (0, \infty)$. It follows from part (2) of Theorem 1.5 and Remark 1.6 in \cite{angel2025existencesharpnessphasetransition} that there exists a critical lifespan $\ell_c = \ell_c(\lambda, Q) \in (0, \infty)$ such that, for any $\ell > \ell_c$, the frog model on $\Z^d$ (without the planted particle) is supercritical, whereas for any $\ell < \ell_c$, the active particles die out almost surely in finite time. Motivated by this, we propose the following conjecture.

\begin{conjecture} \label{Conj: activate_constant_fraction}
Let $d \ge 2$, $\lambda > 0$, and let $Q$ be a symmetric, translation-invariant, heavy-tailed transition kernel on $\Z^d$ with exponent $\alpha > 0$. For any particle lifespan $\ell > \ell_c(\lambda, Q)$, the conclusions of Theorem~\ref{Mainthm: activate_constant_fraction} hold for some $\rho \in (0, 1)$.
\end{conjecture}

We further consider the case where particles have infinite lifespan, i.e. SI frog model,
a scenario particularly relevant when modeling the spread of information or rumors where active particles remain perpetually active. In this case, we can strengthen the result of Theorem \ref{Mainthm: activate_constant_fraction}. Specifically, rather than merely activating a constant fraction of vertices, we prove that the maximal activation time over all vertices on a torus is bounded from above by the same functions as established in Theorem \ref{Mainthm: activate_constant_fraction} for both regimes $\alpha \in (0, d)$ and $\alpha \ge d$.

\begin{theorem} \label{MainThm: infinite_l_activation_time}
Let $d \ge 2$. Assume $Q$ is a symmetric, translation-invariant, and heavy-tailed transition kernel on $\Z^d$ with exponent $\alpha > 0$, and let $Q_L$ be its standard projection onto $\T^d_L$. Consider the frog model on $\T^d_L$ equipped with $Q_L$. 
Assume that all particles have infinite lifespan. 
\begin{enumerate}[left=0.3cm]
\item If $\alpha \in (0, d)$, then for any $\lambda > 0$ and $\epsilon > 0$,
    \begin{align*}
        \lim_{L \rightarrow \infty} \bP_L^+\left(\mathcal{A}((\log L)^{\Delta + \epsilon}) = \T^d_L\right) = 1.
    \end{align*}
\item If $\alpha \ge d$, then for any $\lambda > 0$, there exists a positive constant $C$ such that
\begin{align*}
\lim_{L \rightarrow \infty} \bP_L^+\left(\mathcal{A}(C L) = \T^d_L\right) = 1.
\end{align*}
\end{enumerate}
\end{theorem}

We now turn to establishing matching lower bounds for the activation time. The following theorem addresses the case where $\alpha \in (0, d)$ and should be compared with part (1) of Theorem \ref{Mainthm: activate_constant_fraction}. Together, these results provide matching upper and lower bounds for the activation time in this regime.

\begin{theorem} \label{MainThm: LowerBound_Smaller_d}
Let $d \ge 1$. Assume $Q$ is a symmetric, translation-invariant, and heavy-tailed transition matrix on $\Z^d$ with exponent $\alpha \in (0, d)$. Consider the frog model on \( \Z^d \) equipped with $Q$. Fix any particle lifespan \( \ell \in \mathbb{N} \) and particle density \( \lambda > 0 \). Then, there exists a positive constant $c = c(Q, \lambda, \ell)$ such that $\bP$-almost surely $\mathcal{A}(c(\log L)^\Delta) \subseteq B(L)$ for all large $L$ and 
\begin{align*}
	\lim_{L \rightarrow \infty} \frac{|\mathcal{A}(c(\log L)^\Delta)|}{L^d} = 0,
\end{align*}
where $B(L) := [-L, L]^d \cap \Z^d$.
\end{theorem}

Similar almost sure statement is also valid for the frog model on $\T^d_L$ equipped with $Q_L$ when sending $L$ to infinity: since $Q_L$ is the standard projection of $Q$ on to $\T^d_L$, we can define the frog model on $\T^d_L$ from the frog model on $\Z^d$ and thus couple frog model on $\T^d_L$ for all $L \in \N$ on the same probability space. Therefore, the almost sure statement makes sense in the coupled probability space as that for $\Z^d$. 
Additionally, the proof for Theorem \ref{MainThm: LowerBound_Smaller_d} extends to the torus setting with minor modifications.
We state the result on $\Z^d$ to avoiding devoting too much effort on defining the coupling in detail. The next result concerns with the case when $\alpha > d$.

\begin{theorem} \label{MainThm: lower_bound_alpha_ge_d}
Let $d \ge 1$. Assume $Q$ is a symmetric, translation-invariant, and heavy-tailed transition matrix on $\Z^d$ with exponent $\alpha > d$. Consider the frog model on \( \Z^d \) equipped with $Q$. For any $\epsilon \in (0, 1)$, there exists a positive constant $c = c(Q, \lambda, \ell, \epsilon)$ such that 
\begin{align*}
    \bP\left( \text{for } \forall L > 0, \mathcal{A}(cL) \subseteq B(L), \frac{|\mathcal{A}(cL)|}{L^d} < \epsilon \right) \ge 1 - \epsilon,
\end{align*}
where $B(L) := [-L, L]^d \cap \Z^d$.
\end{theorem}

Inspired by \cite{MR4677584}, we put forward the following conjecture regarding the case when $\alpha = d$.

\begin{conjecture} \label{Conj: alpha_equals_d}
    Let $d \ge 1$. Assume that $Q$ is a symmetric, translation-invariant, 
    heavy-tailed transition matrix on $\Z^d$ with exponent $\alpha = d$. 
    Consider the frog model on $\Z^d$ equipped with $Q$.
    For every particle density $\lambda > 0$ and particle lifespan 
    $\ell \in \N$, there exists $\theta = \theta(Q,\lambda,\ell) \in (0,1)$
    such that the following statement holds: for every $\epsilon \in (0,1)$, 
    there exists $c = c(d,Q,\lambda,\ell,\epsilon)>0$ such that
    \begin{align*}
        \bP\left( \mathrm{AT}(o,x) \ge c |x|^\theta \right) 
        \ge 1-\epsilon,
        \qquad \text{for all } x \in \Z^d .
    \end{align*}
    Moreover, if $\ell > \ell_c(\lambda,Q)$, then the same exponent $\theta$ 
    also satisfies the following statement: for every 
    $\epsilon,\rho \in (0,1)$, there exists 
    $C = C(d,Q,\lambda,\ell,\epsilon,\rho)>0$ such that, for all sufficiently 
    large $L$,
    \begin{align*}
        \bP^+\left(
            \left|\mathcal{A}\left(C L^\theta\right) \cap B(L)\right|
            \ge \rho |B(L)|
        \right) 
        \ge 1-\epsilon .
    \end{align*}
\end{conjecture}

\subsection{Concentration Inequalities for Cover Lifespan}

Aside from the activation time, we also investigate another random quantity called the \textbf{cover lifespan} on $\T^d_L$ with $d \ge 2$. Given a particle density $\lambda$ and a random walk transition kernel $Q$, we independently sample $\mathrm{Pois}(\lambda)$ (and $\mathrm{Pois}(\lambda) + 1$ for $o$ if we include a planted particle) infinite discrete-time trajectories for each vertex on the graph. Let the cover lifespan be the smallest (random) positive integer $\mathcal{L}$ such that every vertex is visited by at least one active particle if we start the frog model dynamics from $o$ and each active particle has lifespan $\mathcal{L}$. That is, $\mathcal{L}$ is the minimal lifespan so that the frog model activates all vertices on the graph. We refer the reader to Section \ref{Sec: formal_construction} for a formal definition of the cover lifespan.

Let $Q$ be a symmetric, translation-invariant, and heavy-tailed transition kernel on $\Z^d$ with exponent $\alpha$. Let $Q_L$ be the standard projection of $Q$ on to $\T^d_L$. Let $\lambda > 0$ and define $\ell^*$ as the smallest $\ell \in \N$ such that
\begin{align} \label{Maindef: cover_time}
    \frac{\lambda \ell}{G_{\ell}(o, o)} \ge d \log L,
\end{align} 
where $G_\ell(o, o) := \sum_{n = 0}^{\ell} Q^n(o, o)$. We prove that $\mathcal{L}$ is concentrated around $\ell^*$, i.e. 
$$\frac{\mathcal{L}}{\ell^*} \xrightarrow[]{d} 1 $$ 
as $L$, the side length of the tori, tends to infinity. Moreover, we also prove that, as presented in Theorem~\ref{MainThm: Susceptibility}, $\mathcal{L}$ is asymptotically equivalent to a type of ``cover time". Let $\tau_{\mathrm{cov}}$ be the first time all vertices are visited if we start from all particles being active and without the presence of the planted particle. An equivalent way to think of this quantity is that $\tau_{\mathrm{cov}}$ is the first time all vertices are visited if we start with Pois$(\lambda L^d)$ active independent walkers evolving according to $Q$ whose initial distribution is the uniform distribution on $\T^d_L$. We call $\tau_{\mathrm{cov}}$ the \textbf{cover time}. 
We would like to point out that, in \cite{MR4987417}, the cover time of $k$ independent stationary random walks on graphs is investigated.
Note that $\tau_{\mathrm{cov}}$ can be coupled with the frog model on the same probability space by not using the planted particle at the origin.
Through a coupling argument, it follows that $\tau_{\mathrm{cov}}$ serves as a lower bound for $\mathcal{L}$. We show that, for any $\epsilon > 0$, with overwhelming probability, 
$$(1-\epsilon)\ell^* \le \tau_{\mathrm{cov}} \le (1+\epsilon)\ell^*,$$
which naturally proves the lower bound for $\mathcal{L}$. The asymptotic equivalence between $\mathcal{L}$ and $\tau_{\mathrm{cov}}$ has been proved for frog model on tori with simple random walk transition kernel in \cite{MR4068310}. Our result extend it to heavy-tailed random walk transition kernels, which is summarized as follows.

\begin{theorem} \label{MainThm: Susceptibility}
Let $d \ge 2$, $\alpha > 0$, and $\lambda > 0$. Consider the corresponding frog model on $\T^d_L$. Then, for any $\epsilon>0$, we have
\begin{align*}
	&\lim_{L \rightarrow \infty} \bP_L^+((1-\epsilon)\ell^* \le \tau_{\mathrm{cov}} \le \mathcal{L} \le (1+\epsilon)\ell^*) = 1.
\end{align*}
\end{theorem}

\section{Preliminaries}
\label{sec:preliminary}

\subsection{The Frog Model, Activation Time, and Cover Lifespan} \label{subsection: FrogModelFormal} \label{Sec: formal_construction}

We first describe the frog model dynamics and then give a formal construction of the probability space with which we will be working.
For a graph $G = (V, E)$ with either a finite or countable vertex set, we fix a one-step random walk transition kernel $Q$ on $V \times V$.

Fix $\lambda > 0$ and $\ell \in \N \cup \{+\infty\}$.
Choose a vertex $o \in V$ as the \textbf{origin}.
Initially, each site contains independently $\mathrm{Pois}(\lambda)$ particles.
To avoid immediate extinction, we sometimes place an additional particle at the origin, called the \textbf{planted particle}. 

The process evolves in discrete time.
Each particle exists in one of three states: \textbf{active}, \textbf{inactive}, or \textbf{removed}.
A particle can only be in one state at any given time.
Their movements and interactions are governed by the following rules:
\begin{itemize}[left=10pt]
\item Active particles:
    \begin{itemize}[left=4pt, rightmargin=65pt]
    \item perform independent discrete-time random walks governed by transition kernel $Q$;
    \item change to the removed state after completing exactly $\ell$ steps.
    \end{itemize}

\item Inactive particles are immobile until activated.

\item When an active particle visits a vertex containing inactive particles:
    \begin{itemize}[left=4pt, rightmargin=65pt]
    \item all inactive particles at that vertex become active instantly and start performing independent random walks;
    \item the active particle continues its random walk unaffected if it made less than $\ell$ steps so far.
    \end{itemize}

\item Removed particles no longer participate in the dynamics of the model.
\end{itemize}

At time 0, only particles at the origin are active while all other particles are inactive.
The process continues until all particles are either removed or inactive, at which point we say the process has \textbf{died out}.
We refer to the dynamics described above as the \textbf{frog model} starting from $o$ with particle density $\lambda$ and lifespan $\ell$. 

We now provide a formal construction of a probability space that couples frog models with varying particle densities $\lambda$ and different origins.

For each $v \in V$, let $\cW_v := \{ \omega_i^v \}_{i = 0}^{\infty}$ be a countably infinite collection of particles. We say $v$ is the \textbf{initial position} for particles in $\cW_v$. 
Each particle independently samples an infinite random walk trajectory by performing a discrete-time random walk starting from its initial position, governed by $Q$. For a particle $\omega$, we use $\omega(k)$ to denote the position of $\omega$ after $k$ steps. 
For $\ell \in \mathbb{N}$, let $\mathrm{R}_\omega(\ell) := \{ \omega(i) : i = 0, 1, \ldots, \ell \}$ denote the set of vertices visited by $\omega$ within the first $\ell$ steps. 

Let $\{ \mathcal{N}_v^{\lambda} \}_{v \in V}$ be a collection of i.i.d. Poisson random variables with mean $\lambda$, which are independent of the random walk trajectories. Let $\cW_v^\lambda := \{ \omega_i^v \}_{i = 1}^{\mathcal{N}_v^{\lambda}}$ be the collection of particles initially sitting at $v$. Note that the particle with index $0$ is not included in $\cW_v^\lambda$, as this index is reserved for the planted particle. Moreover, for $\ell \in \N$, we define
\begin{align} \label{def: union_of_ranges_single_vertex}
    \mathcal{R}_{v, \lambda}(\ell) := \bigcup_{\omega \in \cW_v^\lambda} \mathrm{R}_\omega(\ell).
\end{align}
For $\ell \in \mathbb{N}$ and $v, u \in V$, we write $v \xrightarrow{\lambda, \ell} u$ if $u \in \mathcal{R}_{v, \lambda}(\ell)$. Whenever it is clear from the context, we simply write $v \rightarrow u$. Let
\begin{align} \label{def: T_lambda}
    \mathrm{T}_\lambda(v, u) &:= \inf \left\{ \ell \ge 0 : v \xrightarrow{\lambda, \ell} u \right\}.
\end{align}
We say a finite sequence of vertices $\left( v_i \right)_{i=0}^m$ is a \textbf{$(\lambda, \ell)$-chain} from $v$ to $u$ if $v = v_0$, $u = v_m$, and $\mathrm{T}_\lambda(v_{i-1}, v_i) \le \ell$ for all $1 \le i \le m$. In this case, we say the $(\lambda, \ell)$-chain has length $m$ ($m \ge 0$).

We define and are interested in the following related quantities: 
\begin{align} 
    \mathrm{D}_{\lambda, \ell}(v, u) &:= \inf \Bigg\{ m \ge 0 : ( v_i )_{i=0}^m \text{ is a $(\lambda, \ell)$-chain from } v \text{ to } u \Bigg\}, \label{def: infection_distance} \\
    \mathrm{AT}_{\lambda, \ell}(v, u) &:= \inf \left\{ \sum_{i = 1}^{m} \mathrm{T}_\lambda(v_{i-1}, v_{i}) : ( v_i )_{i=0}^m \text{ is a $(\lambda, \ell)$-chain from } v \text{ to } u \right\}, \label{def: activation_time}
\end{align}
where we follow the convention that $\inf \emptyset = +\infty$.
The quantity $\mathrm{AT}_{\lambda, \ell}(v, u)$ is called the \textbf{activation time} from $v$ to $u$.
It represents the minimal time required for the frog model starting from $v$ with particle density $\lambda$ and lifespan $\ell$ to activate vertex $u$.
The quantity $\mathrm{D}_{\lambda, \ell}(v, u)$ is called the \textbf{infection distance} from $v$ to $u$, representing the length of the shortest $(\lambda, \ell)$-chain connecting $v$ to $u$.
There is a straightforward and useful relation between these two quantities:
\begin{align} \label{eq: relation_between_AT_and_D}
    \mathrm{D}_{\lambda, \ell}(v, u) \le \mathrm{AT}_{\lambda, \ell}(v, u) \le \ell \cdot \mathrm{D}_{\lambda, \ell}(v, u),
\end{align} 
which holds for all $v, u \in V$ and $\ell \in \mathbb{N}$.
Note though, that the chains attaining $\mathrm{D}_{\lambda, \ell}(v,u)$ and $\mathrm{AT}_{\lambda, \ell}(v,u)$ may differ.

Let $n \in \N$, we define the \textbf{activated set} by time $n$ as
\begin{align} \label{def: activated_set}
  \mathcal{A}^{\lambda, \ell}_v(n) := \{ u \in V: \mathrm{AT}_{\lambda, \ell}(v, u) \leq n \} 
\end{align}
which is the set of vertices visited by an active particle within time $n$ when the frog model starts from $v$ with particle density $\lambda$ and lifespan $\ell$.
We shall use $\mathcal{A}^{\lambda, \ell}_v(\infty)$ to denote the set of vertices that are activated at some finite time by the frog model.

For a finite graph $G = (V, E)$, we define an integer-valued random variable $\mathcal{L}_{v, \lambda}$, called the \textbf{cover lifespan}, as 
\begin{align*}
    \mathcal{L}_{v, \lambda} := \inf \left\{ \ell \in \mathbb{N} : \mathcal{A}^{\lambda, \ell}_v(\infty) = V \right\}.
\end{align*} 
In words, $\mathcal{L}_{v, \lambda}$ is the minimum lifespan required for the frog model, starting from vertex $v$ with particle density $\lambda$, to eventually activate all vertices in $G$. Since $G$ is finite, the cover lifespan is almost surely finite whenever the transition kernel $Q$ is irreducible.

Let us fix a vertex $o \in V$ as the origin.
Sometimes, we want to include one extra particle at the origin, called a \textbf{planted particle}.
It is sometimes useful for the planted particle to have a different lifespan  $\ell_o \in \mathbb{N}$. To accommodate this case, we define
\begin{align*}
  \mathcal{R}_{o, \lambda}^{+}(k, \ell_o) := \left( \bigcup_{w \in \mathcal{W}_o^\lambda} \mathrm{R}_w(k) \right) \cup \mathrm{R}_{w_0^o}(\ell_o),
\end{align*}
where $w_0^o$ denotes the planted particle. When a planted particle is present, we write $o \xrightarrow{\lambda, k, \ell_o} v$ (abbreviated as $o \xrightarrow{\lambda, k} v$ or simply $o \rightarrow v$) if $v \in \mathcal{R}_{o, \lambda}^{+}(k, \ell_o)$.
This naturally introduces additional directed connections originating from $o$.
For $v \neq o$, we continue to use $\mathcal{R}_{v, \lambda}(k)$ to define the directed connections. Consequently, we define $\mathrm{T}_\lambda(v, u)$, $\mathrm{AT}_{\lambda, \ell}(v, u)$, $\mathrm{D}_{\lambda, \ell}(v, u)$, $\mathcal{A}_v^{\lambda, \ell}(n)$, and $\mathcal{L}_{v, \lambda}$ in the same way as before, with the dependency on $\ell_o$ being implicit. We will explicitly indicate the inclusion of a planted particle whenever it is used in our analysis.

When the particle density $\lambda$ and lifespan $\ell$ are clear from context, we may drop them from notations.
For the special case where $v = o$, we further simplify notations by writing $\mathrm{AT}(o, u)$ as $\mathrm{AT}(u)$ and $\mathcal{A}_o(n)$ as $\mathcal{A}(n)$.

Let $(\Omega, \mathcal{F}, \P_G)$ denote the product probability space generated by the random walk trajectories $\cup_{v \in V} \cW_v$ and $\{ \mathcal{N}_v(\cdot) \}_{v \in V}$, where each $\omega \in \cup_{v \in V} \cW_v$ represents an independent infinite random walk trajectory starting from its initial position.

In this paper we focus on two graphs:
the $d$-dimensional integer lattice $\Z^d$ and the $d$-dimensional discrete torus with side length $L \in \N^+$, denoted by $\T^d_L := \Z^d / L\Z^d$.
We use $\P := \P_{\Z^d}$ and $\P_L := \P_{\T^d_L}$ to represent the corresponding probability measures for the frog model on these graphs.
Unless specified otherwise, we designate $o = (0,\dots,0)$ as the origin.
When a planted particle is present at the origin, we denote the respective probability measures as $\P^+$ for the integer lattice and $\P_L^+$ for the torus.

\subsection{Heavy-Tailed Random Walks} \label{subsection: HTRW}

In this section, we specify the type of random walk transition kernel $Q$ that we will be working with.
Let $Q$ be a transition kernel on $\Z^d$.
We say $Q$ is \textbf{symmetric} if $Q(x,y) = Q(y,x)$ for all $x,y$.
We say $Q$ is \textbf{translation-invariant} if for any $x,y$ we have $Q(x,y) = Q(0,y-x)$.
Moreover, we say $Q$ is \textbf{heavy-tailed with parameters} $\alpha$ and $K$ if for any $x\neq y \in \Z^d$, the following inequality holds:
\begin{align}
  \frac{K^{-1}}{|x - y|^{d + \alpha}} \le Q(x, y) \le \frac{K}{|x - y|^{d + \alpha}},
\end{align}
where $|\cdot|$ is a norm on $\Z^d$.
For convenience, we usually take $|\cdot|$ to be the $\ell^\infty$-norm throughout this paper.
For $\alpha, K > 0$, we define the set of kernels of interest to us:
\begin{align} \label{def: H_alpha}
\mathcal{H}_{\alpha, K} := \left\{ Q : Q \begin{array}{l} \text{ is symmetric, translation-invariant, and} \\ \text{ heavy-tailed with parameters } \alpha \text{ and } K. \end{array} \right\}.
\end{align}

For the torus $\T^d_L$, we define a random walk by projecting the random walk from $\Z^d$ through the canonical projection to the torus.
For a kernel $Q$ on $\Z^d$, we get the resulting kernel on $\T^d_L$ given by
\begin{align}
  Q_{L}(u,v) := \sum_{x \in \pi^{-1}(v-u)} Q(o,x). \label{def: Q_L}
\end{align}

We use $\bfP_x$ and $\bfP^L_x$ to denote the measures associated with a single random walker starting from vertex $x$ in $\Z^d$ and the torus, respectively, omitting the dependency on $Q$ and $Q_{L}$ for brevity.
For $A \subset \Z^d$, the \textbf{hitting time} of $A$, denoted by $\tau_A$, is defined as
\begin{align*}
  \tau_A := \inf \{n \ge 0: X_n \in A\}.
\end{align*}
Furthermore, we define:
\begin{align*}
    \mathrm{R}(k) := \{ X_i: i = 0, 1, \ldots, k\},
\end{align*}
where we omit the subscript $w$ that appeared in the related quantity in the frog model, which specifies a particle.

%%%%%%%%%%%%%%%%%%%%%%%%%%%%%%%%%%%%%%%%%%%%%%%%%%%%%%%%%%%%%%%%%%
\section{Upper Bounds on Activation Time}
\label{sec:activation_UB}

This section contains the proofs of Theorem~\ref{Mainthm: activate_constant_fraction} and \ref{MainThm: infinite_l_activation_time}.
The proofs is based on the renormalization method, relating the frog model to long-range percolation on a rescaled lattice.
The two main ingredients are:
(1) the frog model with sufficiently large particle lifespan stochastically dominates an auxiliary supercritical Bernoulli site percolation on a renormalized torus, and
(2) a coupling between the frog model and a long-range percolation model with the power-law decay exponent given by $s = \alpha + d$.
Section~\ref{Sec: BernPerc} provides the tools to construct this Bernoulli percolation.
Section~\ref{Sec: Backbone} presents the coupling.
Finally, Theorem~\ref{Mainthm: activate_constant_fraction} and \ref{MainThm: infinite_l_activation_time} are proved in Section~\ref{Sec: upperbound_finite_lifespan} and \ref{Sec: punchline} respectively.

\subsection{Bernoulli Site Percolation on a Renormalized Torus}
\label{Sec: BernPerc}

The main goal of this section is to present two intermediate results: Proposition~\ref{prop: BernPerc_internal_1} and \ref{prop: BernPerc_internal_2}.
These propositions will enable us to construct an auxiliary supercritical Bernoulli site percolation on a renormalized torus. 

Let $\alpha > 0$ and let $Q$ be a symmetric, translation-invariant, and heavy-tailed random walk transition kernel on $\Z^d$ with exponent $\alpha$, i.e. $Q \in \mathcal{H}_{\alpha, K}$ (see \eqref{def: H_alpha}).
For every $L \in \N^{+}$, let $Q_{L}$ be the standard projection of $Q$ onto $\T^d_L$.
We first define a renormalization structure.
Let $r$ be a positive integer, which will be our scale-0.
We will partition $\T^d_L$ into boxes of sidelength $r$, except that the last box in each coordinate may have sidelenth up to $2r$ depending on divisibility.
Formally, let
\begin{align}
  m &= m(L,r) := \lfloor L/r \rfloor. \label{def: m}
\end{align}
and let $I_k$ be the interval
\[
  I_k = \begin{cases}
    [kr,(k+1)r) & k<m-1, \\
    [kr,L) & k=m-1.
  \end{cases}
\]
We partition $\T^d_L$ into disjoint boxes:
\begin{align}
  \mathcal{P}_{L, r} := &\{ B_v : v \in \T^d_m \}, \hspace{3mm} \text{where} \label{def: partition_const_level} \\
  B_v := &\prod_{i=1}^d I_{v_i}.  \label{def: box_const_level}
\end{align}
We refer to $\T^d_m$ as the renormalized torus.
Here, $v_i$ denotes the $i$-th coordinate of $v$.
This construction yields $\Theta(m^d)$ boxes of uniform side length $r$ and $O(m^{d-1})$ boxes with potentially uneven side lengths. For the latter, their side lengths range between $r$ and $2r$ by construction.
The 0-scale $r$ will be related to the lifetime of the frogs as follows:
\begin{align}
  \ell(r) &:= \begin{cases} 
    \lceil r^\alpha \rceil & \text{if } \alpha \in (0, 2), \\
    r^2 & \text{if } \alpha > 2, \\
    \lceil r^2 / \log r \rceil & \text{if } \alpha = 2.
  \end{cases} \label{def: s(r)}
\end{align}
The idea is that $\ell(r)$ is the typical time needed for a random walker with kernel $Q$ to travel a distance $r$.
Later on, we need to choose $r$ larger than some constant $R_0 = R_0(\lambda, K, \alpha, d)$ in order to ensure the validity of Proposition~\ref{prop: BernPerc_internal_1} and \ref{prop: BernPerc_internal_2}, which will also restrict how short the lifespan of frogs can be.

To facilitate further analysis, we define
\begin{align} \label{def: activated_set_restricted}
    \widehat{\mathcal{A}}^{\lambda, \ell, B}_v &:= \left\{  u \in V \; :  \;
        \begin{array}{l}
        \text{ there exists a $(\lambda, \ell)$-chain\footnotemark } \left( v_i \right)_{i=0}^{m}\\
        \text{from }  v \text{ to } u \text{ such that } ( v_i )_{i=0}^{m} \subseteq B
        \end{array} \right\},
\end{align}\footnotetext{Recall that we say a finite sequence of vertices $\left( v_i \right)_{i=0}^m$ forms a \textit{$(\lambda, \ell)$-chain} from $v$ to $u$ if $v = v_0$, $u = v_m$, and $\mathrm{T}_\lambda(v_{i-1}, v_i) \le \ell$ for all $1 \le i \le m$, where $\mathrm{T}_\lambda(v, u)$ is the smallest number of steps needed for a particle in $\cW_v^\lambda$ to reach $u$. We refer the reader to \eqref{def: union_of_ranges_single_vertex}-\eqref{def: T_lambda} for more precise definitions.}representing the set of vertices that will eventually get activated, if we first label all particles outside of $B$ as removed and start the frog model dynamics from vertex $v$ with particle density $\lambda$ and lifespan $\ell$. When the parameters $\lambda$, $\ell$, and the set $B$ are clear from context, we simply write $\widehat{\mathcal{A}}_v$. Note that, we always have $\widehat{\mathcal{A}}^{\lambda, \ell, B}_v \subseteq B$.

\begin{definition} \label{def: good_vertex}
Let $B$ be a finite subset of $\Z^d$ or $\T^d_L$ and $x \in B$. We say $x$ is a \textbf{good} (or $(B, \lambda, \ell)$-\textbf{good}) vertex in $B$ if
\begin{align*}
    |\widehat{\mathcal{A}}_x^{\lambda, \ell, B}| \ge \frac{|B|}{4}.
\end{align*}
We denote the set of $(B, \lambda, \ell)$-good vertices by $\mathrm{Good}_{B}^{\lambda, \ell}$.
\end{definition}

Let $H(r)$ be a $d$-dimensional box on $\Z^d$ or $\T^d_L$ such that its side lengths are between $r$ and $2r$. We first study the existence of $\left(H(r), \lambda, \ell(r)\right)$-good vertices. The following proposition establishes that the probability of not finding a good vertex in a subset $A \subseteq H(r)$ decreases exponentially with respect to the size of $A$. This proposition is of independent interest.

\begin{proposition} \label{prop: BernPerc_internal_1}
Let $d \geq 2$, $\lambda > 0$, and $\alpha > 0$. Let $Q \in \mathcal{H}_{\alpha, K}$ (see \eqref{def: H_alpha}). Then, there exist an absolute constant $C > 0$, positive constants $c = c(K, \alpha, d)$ and $R_0 = R_0(\lambda, K, \alpha, d)$ such that for any positive integers $r$ with $r \ge R_0$ and any subset $A \subseteq H(r)$, we have
\begin{align}
    \P \left( \mathrm{Good}_{H(r)}^{\lambda, \ell(r)} \cap A = \emptyset \right) &\leq C \exp(-c \lambda |A|), \label{ineq: int-dynamic_uppper_bound}
\end{align}
where $\ell(r)$ is defined in \eqref{def: s(r)}.
\end{proposition}

Recall that $Q_L$ is the standard projection of $Q$ onto $\T^d_L$ defined in \eqref{def: Q_L}.
Note that by coupling radom walks on $\Z^d$ and walks on $\T^d_L$, Proposition~\ref{prop: BernPerc_internal_1}, which is stated on $\Z^d$, implies that the same result holds for the frog model on the torus $\T^d_L$ with transition kernel $Q_L$ (by replacing $\P$ with $\P_L$ and $r \ge R_0$ with $L \ge r \ge R_0$).
Therefore, in the sequel, when we are working with the frog model on $\T^d_L$, we will apply Proposition~\ref{prop: BernPerc_internal_1} to the torus $\T^d_L$ with the transition kernel $Q_L$ without explicitly mentioning this subtlety. 
% , we will apply Proposition~\ref{prop: BernPerc_internal_1} to the torus $\T^d_L$ with the transition kernel $Q_L$ without explicitly mentioning it. 
% we only show Proposition~\ref{prop: BernPerc_internal_1} for $\Z^d$ while later on we will apply it to the torus $\T^d_L$.
To prove it, we need a lemma about the range of a random walk. Recall that $\mathrm{R}(k)$ denotes the set of vertices visited by a random walk up to time $k$, and $\bfP_x$ denotes the probability measure associated with the random walk starting from vertex $x \in \Z^d$.

\begin{lemma} \label{lemma: range_estimates_main_text}
    Consider a random walk on $\Z^d$ driven by $Q$. Let $D \subseteq H(r)$ with $|D|/|H(r)| \leq 3/4$. Define $\widetilde{\mathrm{R}}(k) := \mathrm{R}(k) \cap (H(r) \setminus D)$. Then, there exist positive constants $c_0, c_1$ depending on $Q$ such that for any $r \in \N^+$ and any $x \in H(r)$:
    \begin{align} %\label{ineq: range_est_case1}
        \bfP_x \left( \, \left|\widetilde{\mathrm{R}}\left(\ell(r)\right) \right| \geq c_1 f(r) \right) \geq c_0,
    \end{align}
    where $f(r)$ is defined as follows with $\ell(r)$ defined in \eqref{def: s(r)}:
    \begin{align*}
        f(r) := \begin{cases}
            \ell(r) & \text{if } d \ge 3 \text{ or } (d = 2 \text{ and } \alpha < 2), \\
            \ell(r) / \log \log \ell(r) & \text{if } d = 2 \text{ and } \alpha = 2, \\
            \ell(r) / \log \ell(r) & \text{if } d = 2 \text{ and } \alpha > 2.
        \end{cases}
    \end{align*}
\end{lemma}

The proof of Lemma~\ref{lemma: range_estimates_main_text} is deferred to the supplementary material (see Lemma~\ref{lemma: range_estimates}). 

\begin{proof}[Proof of Proposition~\ref{prop: BernPerc_internal_1}]

Note that it suffices to prove inequality \eqref{ineq: int-dynamic_uppper_bound} for non-empty sets $A$ satisfying $|A| < |H(r)|/4$. For sets with $|A| \geq |H(r)|/4$, we can simply apply the proposition to a subset of $A$ containing exactly $\lfloor|H(r)|/4\rfloor$ vertices. The inequality \eqref{ineq: int-dynamic_uppper_bound} then holds for all $A \subseteq H(r)$ with the constant $c$ being replaced by $c/5$, provided that $r \geq R_0$ and $|H(r)| \geq 20$. 

Our goal is to find a vertex $x$ in $A$ such that 
\begin{align*}
    \left| \widehat{\mathcal{A}}_x^{\lambda, \ell(r), H(r)} \right| \geq \frac{|H(r)|}{4},
\end{align*}
which would make $x$ a good vertex according to Definition~\ref{def: good_vertex}. We abbreviate $\widehat{\mathcal{A}}_x^{\lambda, \ell(r), H(r)}$ as $\widehat{\mathcal{A}}_x$ since the parameters are clear from context.

We now give an overview of the proof. To search for good vertices, we construct an exploration process that systematically reveals $\cup_{x \in A} \widehat{\mathcal{A}}_x$. The key idea is as follows: In each step of the exploration, we select a new vertex $v$ from the current explored set, expose $\cR_{v, \lambda}(\ell(r)) \setminus A$ (recall that $\cR_{v, \lambda}(\ell)$ is the union of $\ell$-step ranges of the $\Poi(\lambda)$ particles innitalized on $v$ as introduced by \eqref{def: union_of_ranges_single_vertex}), and add these newly explored vertices to the explored set. We will show that the increment in the size of the explored set stochastically dominates a Bernoulli random variable multiplied by a positive constant as long as there are still at least $|H(r)|/4$ unexplored vertices in $H(r) \setminus A$. The parameter of this Bernoulli random variable depends only on $\lambda$ and $Q$, while the multiplying constant can be made arbitrarily large by choosing $r$ sufficiently large. Therefore, for fixed $\lambda$, if $r$ is large enough, the mean increment can be made at least $2$. This observation is crucial for proving that the failure to find a good vertex in $A$ is a large deviation event when $|A|$ is sufficiently large.

Fix an order on the vertices in $H(r)$. In each stage, the process recursively defines four vertex sets: 
\begin{itemize}[left=5pt, itemsep=0.1em, topsep=0.5em, rightmargin=20pt]
    \item[-] $A_i$: the \textbf{explored} set, containing all vertices reached by the exploration process upon the completion of stage $i$;
    \item[-] $\cV_i$: the \textbf{revealed} set, containing vertices from $A_i$ whose particles' trajectories have been examined;
    \item[-] $\cU_i$: the \textbf{unrevealed} set, containing vertices from $A_i$ whose particles' trajectories have not yet been examined, where $\cU_i = A_i \setminus \cV_i$;
    \item[-] $\mathcal{I}_i$: the \textbf{incremental} set, containing vertices newly added to the explored set at stage $i$.
\end{itemize}
The subscript $i \in \{0, 1, \ldots\}$ denotes the stage index of the exploration process. At stage 0, let 
$$A_0 = A, \, \cV_0 = \emptyset, \, \cU_0 = A, \textup{ and } \, \mathcal{I}_0 = \emptyset.$$ 
Suppose we have completed the $k$-th stage of the exploration process and have constructed the four types of sets $A_j$, $\cV_j$, $\cU_j$, and $\mathcal{I}_{j}$ for $j = 0, \ldots, k$. If $\cU_k = \emptyset$, then the process terminates after the $k$-th stage is completed. Otherwise, the process proceeds to the $(k{+}1)$-th stage: First, we pick $x_{k+1} \in \cU_k$ according to the following rules:
\begin{itemize}[left=15pt, itemsep=0.5em, topsep=1em, rightmargin=25pt]
	\item[(1)] if $\sqcup_{j = 0}^{k} \mathcal{I}_{j} \cap \cU_{k} = \emptyset$, choose the vertex in $A \cap \cU_k$ with the smallest order; 
	\item[(2)] if $\sqcup_{j = 0}^{k} \mathcal{I}_{j} \cap \cU_{k} \neq \emptyset$, choose the vertex in $\sqcup_{j = 0}^{k} \mathcal{I}_{j} \cap \cU_k$ with the smallest order,
\end{itemize}
where $\sqcup$ denotes disjoint union. Note that, in case (2), the vertex $x_{k+1}$ does not belong to $A$. Then, we complete the $(k{+}1)$-th stage by defining the following sets:
\begin{align*}
    &\mathcal{I}_{k+1} := H(r) \cap \left( \cR_{x_{k+1}, \lambda} \left( \ell(r) \right) \, \setminus \, A_k \right),  \\
    A_{k+1} := A_k \, \sqcup \, &\mathcal{I}_{k+1}, \hspace{3mm} \cV_{k+1} := \cV_k \sqcup \{ x_{k+1} \}, \hspace{3mm} \cU_{k+1} := A_{k+1} \setminus \cV_{k+1}. 
\end{align*}
We summarize some straightforward properties of the exploration process: for any $k \ge 1$,
\begin{itemize}[left=15pt, itemsep=0.3em, topsep=1em, rightmargin=25pt]
    \item[(i)] $A_k = \cU_k \sqcup \cV_k = A \sqcup \left( \sqcup_{j=1}^k \mathcal{I}_j \right)$;
    \item[(ii)] there exists a (random) vertex $x = x(k) \in A$ such that $\mathcal{U}_k \cap \bigcup_{j=1}^k \mathcal{I}_j \subseteq \widehat{\mathcal{A}}_x$.
\end{itemize}
Property (i) follows directly from the definitions. Property (ii) can be proven by induction using rules (1) and (2) for selecting $x_{k+1}$. We omit the detailed proof here.

Let $\cF_0$ be the trivial $\sigma$-field and, for $k \ge 1$, let $\cF_k := \sigma(\{ \mathcal{I}_{j} \}_{j = 1}^{k})$. Thus, $( \cF_k )_{k \ge 0}$ forms a filtration. We define two stopping times with respect to $( \cF_k )_{k \ge 0}$:
\begin{align*}
    \tau_1 := \inf \{ k: \cU_k = \emptyset \}, \hspace{1mm} \text{and} \hspace{2mm}
    \tau_2 := \inf \{ k: \left| A_k \right| > \frac{3 |H(r)|}{4} \}.
\end{align*}
Let $\tau := \min \{ \tau_1, \tau_2 \}$. Let $\Delta_k := \left| \mathcal{I}_{k} \right| = \left| A_{k} \setminus A_{k-1} \right|$, which is the size of the $k$-th incremental set. Moreover, let $S_k := \sum_{i=1}^{k \wedge \tau_1} \Delta_i$ be the total size of the incremental sets up to stage $k$. We now provide a sufficient condition for the existence of a good vertex in $A$. 

\begin{claim} \label{claim: internal_exploration}
    Assume $A$ is non-empty and $S_k \ge k$ for all $k \in \{ |A|, |A|{+}1, \ldots, \tau \}$. Then, there exists a vertex $x \in A$ such that $| \widehat{\mathcal{A}}_x | \ge |H(r)|/4$.
\end{claim}

\begin{proof}[Proof of Claim~\ref{claim: internal_exploration}]
    We first show that, under the assumption of Claim~\ref{claim: internal_exploration}, we must have $\tau = \tau_2$, which implies 
    \begin{align*}
        |A_\tau| > \frac{3|H(r)|}{4}.
    \end{align*}
    
    Since $A$ is non-empty, the exploration process does not terminate immediately. Moreover, since we reveal exactly one vertex in each stage, we must have $\tau_1 \geq |A|$, meaning the exploration process must complete at least $|A|$ stages. 
    
    Now, for any $k \geq |A|$, if $S_k \geq k$, then we have
    \begin{align*}
        |\mathcal{U}_k| = |A_k| - k = |A| + S_k - k \geq (S_k - k) + 1 \geq 1,
    \end{align*}
    which ensures that the process completes the $(k{+}1)$-th stage. Therefore, under our assumption that $S_k \geq k$ for all $k \in \{|A|, |A|+1, \ldots, \tau\}$, we must have $\tau_1 > \tau$. This means $\tau = \tau_2$.

    Next, we prove that there exists some $x^\star \in A$ such that
    \begin{align*}
        A_\tau \setminus \Bigl( \cV_{|A|} \cup A \Bigr) \subseteq \widehat{\mathcal{A}}_{x^\star}.
    \end{align*}
    This is sufficient to conclude the proof, because $|A_\tau| > 3|H(r)|/4$ implies
    $$|\widehat{\mathcal{A}}_{x^\star}| > 3|H(r)|/4 - 2|A| \geq |H(r)|/4.$$ 
    Indeed, it follows from two facts: (1) $\cV_{k} = \{ x_1, \ldots, x_k \}$, in particular $|\cV_{k}| = k$, and (2) $|A| \leq |H(r)|/4$, making $x^\star$ a good vertex.

    By property (ii), for any $k$, there exists some (random) $x \in A$ such that
    \[
        \mathcal{U}_{k} \cap \bigcup_{i=1}^{k} \mathcal{I}_{i} \subseteq \widehat{\mathcal{A}}_x.
    \]
    By property (i), we can rewrite the left-hand side as:
    \begin{align*}
        \mathcal{U}_{k} \cap \bigcup_{i=1}^{k} \mathcal{I}_{i} = \left(\bigcup_{i=1}^{k} \mathcal{I}_{i}\right) \, \setminus \, \mathcal{V}_k = A_k \, \setminus \, \bigg(\cV_k \cup A\bigg) \subseteq \widehat{\mathcal{A}}_x,
    \end{align*}
    Let $k = |A|$ and let $x^\star \in A$ be such that $A_{|A|} \setminus (\mathcal{V}_{|A|} \cup A) \subseteq \widehat{\mathcal{A}}_{x^\star}$. Under the assumption of Claim~\ref{claim: internal_exploration}, we now prove by induction that 
    $$A_k \setminus (\mathcal{V}_{|A|} \cup A) \subseteq \widehat{\mathcal{A}}_{x^\star} \text{ for all } k \in \{|A|, |A|+1, \ldots, \tau\}.$$ 
    The base case $k = |A|$ is already established.
    For the inductive step, assume $A_k \setminus (\mathcal{V}_{|A|} \cup A) \subseteq \widehat{\mathcal{A}}_{x^\star}$ for some $k \geq |A|$ with $S_k \geq k$. Since $|\bigcup_{i=1}^{k} \mathcal{I}_{i}| = S_k \geq k$, we know that $\bigcup_{i=1}^{k} \mathcal{I}_{i} \cap \mathcal{U}_{k}$ is non-empty, and by rule (2),
    $$x_{k+1} \in \bigcup_{i=1}^k \mathcal{I}_{i} \cap \mathcal{U}_{k} = A_k \, \setminus \, (\mathcal{V}_k \cup A) \subseteq \widehat{\mathcal{A}}_{x^\star}.$$
    Since $x_{k+1} \in \widehat{\mathcal{A}}_{x^\star}$, we have $\mathcal{I}_{k+1} \subseteq \widehat{\mathcal{A}}_{x^\star}$. Therefore,
    \begin{align*}
        A_{k+1} \, \setminus \, (\mathcal{V}_{|A|} \cup A) &= (A_k \cup \mathcal{I}_{k+1}) \, \setminus \, (\mathcal{V}_{|A|} \cup A) \subseteq \widehat{\mathcal{A}}_{x^\star},
    \end{align*}
    which completes the induction and thus the proof of Claim~\ref{claim: internal_exploration}.
\end{proof}

We continue the proof of Proposition~\ref{prop: BernPerc_internal_1}. Recall that we defined in Lemma~\ref{lemma: range_estimates_main_text}:
\begin{align*}
    f(r) := \begin{cases}
        \ell(r) & \text{if } d \ge 3 \text{ or } (d = 2 \text{ and } \alpha < 2), \\
        \ell(r) / \log \log \ell(r) & \text{if } d = 2 \text{ and } \alpha = 2, \\
        \ell(r) / \log \ell(r) & \text{if } d = 2 \text{ and } \alpha > 2.
    \end{cases}
\end{align*}
By Lemma~\ref{lemma: range_estimates_main_text} and Poisson thinning, on the event $\{ \tau \ge k \} \in \cF_{k-1}$, we have that
\begin{align}
    \bP_L \left( \Delta_{k} \ge c_1 f(r) \, | \cF_{k-1} \right) \ge 1 - \exp(-c_0 \lambda) \text{ a.s. }, \label{ineq: single_exploration_step_increment_bound}
\end{align}
where $c_0$ and $c_1$ are constants depending on $Q$ (the same constants as in Lemma~\ref{lemma: range_estimates_main_text}). Let $q := 1 - \exp(-c_0 \lambda)$ and let $( X_i )_{i=1}^\infty$ be a sequence of i.i.d. Bernoulli random variables with mean $q$ which are defined on a probability space equipped with measure $\mathbb{Q}$. Therefore, on the event $\tau \ge k$, we have that $\Delta_{k}$ stochastically dominates $c_1 f(r) X_1$. Define $T_k := c_1 f(r) \sum_{i=1}^{k} X_i$.
It then follows from Claim~\ref{claim: internal_exploration} that
\begin{align}
	\bP_L \left( \mathrm{Good}_{H(r)}^{\lambda, \ell(r)} \cap A = \emptyset \right)  \le \; \bP_L \bigg( \min_{k: |A| \le k \le \tau} \frac{S_k}{k} < 1  \bigg) \le \, \mathbb{Q}\left(\inf_{k: k \ge |A|} \frac{T_k}{k} < 1 \right) \label{eq: S_k_stoch_dom_T_k}
\end{align}
where the second inequality is a direct consequence of the stochastic domination. 
We now proceed in two cases. First, assume that $c_0 \lambda \le 1$. By choosing $r$ sufficiently large (depending on $\lambda$ and $Q$) such that 
$$c_1 f(r) (1 - \exp(-c_0 \lambda)) \ge 2,$$ 
we can apply \eqref{ineq: large_deviation_Bin_4} in Lemma~\ref{lemma: large_deviation_Bern} to \eqref{eq: S_k_stoch_dom_T_k}:
\begin{align}
    \mathbb{Q}\left(\inf_{k: k \ge |A|} \frac{T_k}{k} < 1 \right) \le \mathbb{Q} \left( \inf_{k: k \ge |A|} \frac{\sum_{i=1}^{k} X_i}{k} < \frac{q}{2} \right) \le 32 \exp\left(- \frac{|A| q }{16}\right),
\end{align}
where $q = 1 - \exp(-c_0 \lambda) \ge c_0 \lambda /2$, using $1 - e^{-x} \ge x/2$ for $x \in [0,1]$.

For the second case, assume that $c_0 \lambda > 1$. By possibly further enlarging $r$ such that $f(2) \ge 1/2$, we can ensure that
$$1 - \frac{1}{c_1 f(r)} \ge e^{-1} + \frac{1}{2} \ge \exp(-c_0 \lambda) \left( 1 + \frac{\exp(c_0 \lambda)}{2} \right).$$
We now rewrite \eqref{eq: S_k_stoch_dom_T_k} and apply \eqref{ineq: large_deviation_Bin_3} in Lemma~\ref{lemma: large_deviation_Bern}:
\begin{align*}
    \mathbb{Q}\left(\inf_{k: k \ge |A|} \frac{T_k}{k} < 1 \right) = \mathbb{Q} \left( \sup_{k: k \ge |A|} \frac{\sum_{i=1}^{k} Y_i}{k} > 1 - \frac{1}{c_1 f(r)} \right) \le C_1 \exp\left(- c' c_0 |A| \lambda \right),
\end{align*}
where $Y_i := 1 - X_i$, and $C_1$ and $c'$ are absolute constants. These constants are obtained through direct calculation following \eqref{ineq: large_deviation_Bin_3} with $\delta = \exp(c_0 \lambda) / 2$ and $p = \exp(-c_0 \lambda)$ together with the assumption that $c_0 \lambda > 1$.
Combining the two cases, we conclude the proof of Proposition~\ref{prop: BernPerc_internal_1}. 
\end{proof}

The existence of a single good vertex in a box $H(r)$ is insufficient for our purposes. We need to establish that with high probability, the set of good vertices in $H(r)$ takes ``almost" a constant fraction of $H(r)$, which is the content of the next corollary.

\begin{corollary} \label{cor: almost_constant_fraction_good_vertex}
Let $d \geq 2$, $\lambda > 0$, $\alpha > 0$, and $Q \in \mathcal{H}_{\alpha, K}$. Then, there exists positive constants $C_1 = C_1(K, \alpha, d, \lambda)$ and $C_2 = C_2(d)$ such that for any $r \ge R_0$,
\begin{align*}
    \bP \left( \, \left| \mathrm{Good}_{H(r)}^{\lambda, \ell(r)} \right| < \frac{r^d}{C_1 \log r} \, \right) \le C_2 \, r^{-d},
\end{align*}
where \( \ell(r) \) is defined in \eqref{def: s(r)} and \( R_0 \) is the same constant as in Proposition~\ref{prop: BernPerc_internal_1}.
\end{corollary}

\begin{proof}
We tile $H(r)$ with disjoint sub-boxes, each having side lengths between $(C^\star \log r)^{1/d}$ and $2 (C^\star \log r)^{1/d}$, where $C^\star := \frac{2 d}{c \lambda}$ and $c = c(Q)$ is the same constant in Proposition~\ref{prop: BernPerc_internal_1}.  

By Proposition~\ref{prop: BernPerc_internal_1}, for each sub-box $H'$, the probability that $H'$ contains none of the good vertices is 
\begin{align*}
\P(\mathrm{Good}_{H(r)}^{\lambda, \ell(r)} \cap H' = \emptyset) \leq C \exp(-c \lambda |H'|) \leq C r^{-2d},
\end{align*}
where $C > 0$ is the absolute constant that appears in Proposition~\ref{prop: BernPerc_internal_1}.

Let $F$ denote the event that there exists at least one sub-box that does not contain a good vertex. Applying a union bound over all sub-boxes, we obtain
\begin{align*}
\P(F) \leq (2r)^d \cdot C r^{-2d} \leq C_2 r^{-d},
\end{align*}
where $C_2 := C 2^d$. Moreover, on the complement of $F$, each sub-box contains at least one good vertex, implying that
\begin{align*}
\left|\mathrm{Good}_{H(r)}^{\lambda, \ell(r)}\right| \geq \frac{r^d}{2^d C^\star \log r} = \frac{r^d}{C_1 \log r},
\end{align*}
where $C_1 := 2^d C^\star$, which completes the proof.
\end{proof}

Similar to the discussion following Proposition~\ref{prop: BernPerc_internal_1}, although Corollary~\ref{cor: almost_constant_fraction_good_vertex} is stated for boxes on $\Z^d$, it actually implies the same upper bound holds for $H(r)$ on $\T^d_L$ for all $L \ge r$. From now on, we will consistently work with boxes on the torus $\T^d_L$.

To motivate our next result, we provide an overview of our strategy for constructing an auxiliary Bernoulli site percolation on $\T^d_m$.

The stochastic domination is constructed in two steps, each utilizing half of the particles. For each vertex $x$, recall that $\mathcal{W}^\lambda_x$ denotes the set of particles initially positioned at $x$ at time 0, with $|\mathcal{W}^\lambda_x| \sim \mathrm{Pois}(\lambda)$. Using Poisson thinning (independent of everything else), we split $\mathcal{W}^\lambda_x$ into two disjoint sets, $\mathcal{W}^{\mathrm{int}}_x$ and $\cW^{\mathrm{nbhd}}_x$, such that 
\begin{align*}
    |\cW^{\mathrm{int}}_x|, |\cW^{\mathrm{nbhd}}_x| \sim \Pois(\lambda/2) \text{ and are independent}.
\end{align*}

We denote $\cW_{B}^{\mathrm{int}} := \cup_{x \in B} \cW_{x}^{\mathrm{int}}$ and $\cW_{B}^{\mathrm{nbhd}} := \cup_{x \in B} \cW_{x}^{\mathrm{nbhd}}$. In the first step, we apply Corollary~\ref{cor: almost_constant_fraction_good_vertex} to each box $B_v$ in $\mathcal{P}_{L, r}$, utilizing particles in $\cW_{B_v}^{\mathrm{int}}$. Once at least a quarter of the vertices in $B_v$ are activated using particles in $\cW_{B_v}^{\mathrm{int}}$, we then apply the next proposition to the second set of particles $\cW_{B_v}^{\mathrm{nbhd}}$. Proposition~\ref{prop: BernPerc_internal_2} demonstrates that, with high probability, the second step activates all vertices in $B_v$ as well as all $2d$ neighboring boxes. Define 
\begin{align*}
    \widehat{B}_v := B_v \, \cup \, \left( \bigcup_{i = 1}^d \left( B_{v + e_i} \cup B_{v - e_i} \right) \right)
\end{align*}
where $v \in \T^d_m$. Additionally, for any $D \subseteq V$ and $k \in \N^+$, define
\begin{align*}
    \mathcal{R}_{D, \lambda}(k) := \bigcup_{v \in D} \mathcal{R}_{v, \lambda}(k),
\end{align*}
which represents the set of vertices activated by a particle in $D$ after $k$ steps. 

\begin{proposition} \label{prop: BernPerc_internal_2}
Let $\lambda > 0$, $d \geq 2$, and $\alpha > 0$.
Let $Q \in \mathcal{H}_{\alpha, K}$ and let $Q_L$ be defined as in \eqref{def: Q_L}. 
There exist positive constants $R = R(\lambda, K, \alpha, d)$, $C = C(d)$, and $c = c(K, \alpha, d)$, such that for 
any $r, L$ with $L \ge r \ge R$, 
any $v \in \T^d_m$, and 
any $D \subseteq B_v$ with $|D|/|B_v| \ge 1/4$, we have 
\begin{align} \label{eq: BernPerc_internal_2_result}
    \bP_L \left( \widehat{B}_v \not\subseteq \mathcal{R}_{D, \lambda}(\ell(r)) \, \right) \leq C \exp{\left( - c \lambda f(r) \right)},
\end{align}
where $ \ell(r) $ is defined in \eqref{def: s(r)} and $f(r)$ is defined as in Lemma~\ref{lemma: range_estimates_main_text}.
\end{proposition}
In order to prove Proposition~\ref{prop: BernPerc_internal_2}, we need the following random walk result. Let $B_x(r)$ denote the hypercube of side-length $2r$ centered at $x$ in $\T^d_L$. For a vertex $x$, let $\tau_x$ represent the hitting time of vertex $x$, and let $\bfP^L_x$ denote the probability measure associated with a single random walk on $\T^d_L$ starting from vertex $x$. The proof of the following lemma is deferred to Lemma~\ref{apx_lemma: random_walk_expected_hitting_size_lower_bound} in the supplementary material.

\begin{lemma} \label{lemma: random_walk_expected_hitting_size_main_text}
    Consider a random walk on $\Z^d$ driven by $Q \in \mathcal{H}_{\alpha, K}$ starting from $x \in \Z^d$.
    Then, there exists a constant $c' = c'(K, \alpha, d)$ such that for any integer $r > 0$ and any $D \subseteq B_x(4r)$ with $|D| \ge r^d/4$, we have
    \begin{align}
        \sum_{y \in D} \bfP_y \left( \tau_x \le \ell(r) \right) \ge c' f(r),
    \end{align}
    where $ \ell(r) $ is defined in \eqref{def: s(r)} and $f(r)$ is defined as in Lemma~\ref{lemma: range_estimates_main_text}.
\end{lemma}

\begin{proof}[Proof of Proposition~\ref{prop: BernPerc_internal_2}]
Observe that, for any $x \in \widehat{B}_v$, we have $D \subseteq B_x(4r)$ with $|D| \ge r^d/4$. Therefore, $D$ satisfies the assumptions in Lemma~\ref{lemma: random_walk_expected_hitting_size_main_text} for each $x \in \widehat{B}_v$. By Poisson thinning and a union bound, the probability on the left-hand side of \eqref{eq: BernPerc_internal_2_result} can be upper bounded by
    \begin{align}
        \sum_{x \in \widehat{B}_v} \exp \left(- \lambda \sum_{y \in D} \bfP^L_y(\tau_x \le \ell(r) ) \right) \leq (2d + 1) (2r)^d \exp \left( - c' \lambda f(r) \right),
    \end{align}
where we applied Lemma~\ref{lemma: random_walk_expected_hitting_size_main_text} with the observation that $\bfP^L_y(\tau_x \le \ell(r) ) \ge \bfP_y(\tau_x \le \ell(r) )$ where we abuse the notation by identifying $y$ and $x$ as vertices in both $\T^d_L$ and $\Z^d$.
It is straightforward to verify that if we choose $r$ sufficiently large such that $c' \lambda f(r) \ge 2d\log r$, then the above expression is bounded by $(2d+1) 2^d \exp(-c'\lambda f(r)/2)$, which concludes the proof of Proposition~\ref{prop: BernPerc_internal_2}.
\end{proof}

\subsection{From Long-Range Percolation to a Quickly Activated Set} \label{Sec: Backbone}
In this section, we construct an auxiliary long-range percolation model on $\T^d_m$ that is coupled with the frog model. 
For convenience, we use the same renormalization structure as in Section~\ref{Sec: BernPerc} (see definitions in \eqref{def: partition_const_level}-\eqref{def: box_const_level}) for scale-0 boxes, 
where each vertex on $\mathbb{T}^d_m$ corresponds to a box with side lengths between $r$ and $2r$ in the partition $\mathcal{P}_{L, r}$ of $\mathbb{T}^d_L$. 

Using the coupled long-range percolation clusters, 
we prove that when $\alpha \in (0, d)$ the frog model produces a random vertex set $S$ that is ``spatially homogeneous" and such that, once one vertex in $S$ is activated, every other vertex in $S$ becomes activated within time $O((\log L)^{\Delta + o(1)})$ with high probability as $L \to \infty$, where
\begin{align*}
    \Delta^{-1} := \log_2 \left(\frac{2d}{\alpha + d}\right).
\end{align*}

Our analysis builds upon a multi-scale argument developed in the proof of Proposition 4.4 in \cite{MR2850269}, which we adapt to our specific requirements. We point out a few important differences between our case and the original setup in \cite{MR2850269}: (1) we do not retain the original nearest-neighbor edges on the torus $\T^d_m$; instead, we consider a directed long-range percolation with probability close to 1 such that nearest-neighbor edges are present, and (2) we begin with a Bernoulli site percolation on $\T^d_m$ with density close to 1, and we only allow long-range connections between two open vertices. In essence, the coupled long-range percolation we construct is built upon supercritical Bernoulli percolation clusters.

Fix a positive integer $r$. Recall that $m = m(L) = \lfloor L/r \rfloor$. Let $\T^{[L]} := \T^d_{m}$ be the renormalized torus for $\T^d_L$.
We now define a directed Bernoulli long-range percolation on the renormalized torus $\T^{[L]}$ with the frog model on $\T^d_L$. While there are many ways to perform this coupling, the approach described below is particularly convenient for our purposes.

Let $g: \N \rightarrow (1, +\infty)$ be a function such that $\log g(r) / \log r \to 0$ as $r \to \infty$. 
Let $p: \N \rightarrow (0, 1)$ be a function such that $p(r) \to 0$ as $r \to \infty$.
The long-range percolation model will be constructed on $\T^d_m$ and will depend on the choice of $g$ and $p$.
For each vertex $v \in \T^d_m$, we declare $v$ to be \textbf{available} with probability at least $1 - p(r)$ and \textbf{unavailable} otherwise, independently for each $v$.
Let $X_v^{[L]}$ be the indicator that $v$ is available.
For each available vertex $v \in \T^{[L]}$, fix a subset $G_v \subseteq B_v$ such that 
\begin{align} \label{ineq: G_v_condition}
    |G_v| \ge \frac{r^d}{g(r)}.
\end{align}
In the next section, we will specifically choose $g(r) = C_1 \log r$ (where $C_1$ is the same constant as in Corollary~\ref{cor: almost_constant_fraction_good_vertex}) and set $G_v$ to be the set of good vertices in $B_v$, whenever the collection of good vertices satisfies condition \eqref{ineq: G_v_condition}. For now, we proceed with our construction of the long-range percolation model under the assumption \eqref{ineq: G_v_condition} without specifying $g$, $f$, and $G_v$.

% Let $B_u^{\circ}$ denote the collection of vertices in $B_u$ such that each vertex is at least $r/4$ distance away from $B_u^c$.
For $L \ge r$ and for $u, v \in \T^{[L]}$ with $u \neq v$, define
\begin{align*}
    Y_{(u, v)}^{[L]} := 
    \begin{cases}
        \1_{\left\{\sum_{x \in B_u} \sum_{y \in G_v} |\cW_{x,y}^{\lambda}| \ge 1\right\}} & \text{if } X_v^{[L]} = 1, \\
        0 & \text{if } X_v^{[L]} = 0,
    \end{cases}
\end{align*}
where $\cW_{x,y}^{\lambda}$ denotes the set of particles initially at $x$ that jump to $y$ in their first step. 
When $Y_{(u, v)}^{[L]} = 1$, we write $u \xRightarrow{} v$. 

\begin{lemma}
\label{lemma: long-range-percolation-coupling}
    Let $d\ge1$, $\alpha \in (0, d)$, and $Q \in \mathcal{H}_{\alpha, K}$. 
    Fix $r \in \N^+$ and $\lambda > 0$. 
    Then, $\{ Y_{(u, v)}^{[L]} \}_{u \neq v}$ are mutually independent for each $L \ge r$.
    Moreover, there exist $\beta_1$ with $\beta_1 \to \infty$ as $r \to \infty$,
    such that, for any $L \ge r$ and any distinct $u, v \in \T^{[L]}$, we have
    \begin{align} \label{eq: long-range-percolation-coupling}
        \P_{L}\left(Y_{(u, v)}^{[L]} = 1\right) \ge 1 - \exp\left(- \beta_1 |u - v|^{-s}\right),
    \end{align}
    where $|u - v|$ denotes the $\ell_\infty$-distance between $u$ and $v$ on $\T^{[L]}$ and $s := \alpha + d$.
\end{lemma}

\begin{proof}
    First, note that the independence of $\{ Y_{(u, v)}^{[L]} \}_{u \neq v}$ follows from Poisson thinning together with the independence of $X_v^{[L]}$'s. 
    Indeed, the collection $\{ |\cW_{x,y}^{\lambda}| : x, y \in \T^d_L \}$ are mutually independent Poisson random variables with parameters $\lambda Q_L(x,y)$, where $Q_L$ is the standard projection of $Q$ onto $\T^d_L$ as defined in \eqref{def: Q_L}.
        
    Next, observe that for any $u, v \in \T^{[L]}$ with $u \neq v$ and any $x \in B_u$, $y \in B_v$, we have
    \begin{align}
        |x - y| \le 2 r ( |u - v| + 2 ), \label{ineq: distance_equivalence}
    \end{align}
    where $|u - v|$ and $|x - y|$ denote the $\ell_\infty$-distances between $u$ and $v$ on $\T^{[L]}$ and between $x$ and $y$ on $\T^d_L$. Therefore, for such $x, y, u, v$, we have $|x - y| \asymp_d r |u - v|$ uniformly for all $L \ge r$. 
    
    Therefore, for any $L$ and any distinct $u, v \in \T^{[L]}$, we have the following lower bound:
    \begin{align*}
        &\bP_L\left(Y_{(u, v)}^{[L]} = 0\right) \le p(r) + \exp\left(- \lambda \sum_{x \in B_u} \sum_{y \in G_v} Q_L(x,y) \right) \\
        \le \; &p(r) + \exp\left(- \beta_0(r) |u - v|^{-s} \right) \le \exp\left(- \beta_1(r) |u - v|^{-s} \right),
    \end{align*}
    where $ \beta_0(r) := c_0 \lambda r^{2d-s} / g(r) $ with $c_0$ being a positive constant depending on $K,\alpha$, and $d$. 
    Moreover, $\beta_1(r) := \min\left\{ -\log(p(r)), \beta_0(r) \right\}$. 
    The inequality above uses the lower bound from Lemma~\ref{lemma: one-step_Z_to_T} and the fact that $|x - y| \asymp_d r |u - v|$. 
    Finally, since $\log g(r) / \log r \to 0$ as $r \to \infty$ and $p(r) \to 0$ as $r \to \infty$, it follows that both $\beta_1$ diverges to infinity as $r \to \infty$ when $s < 2d$. This completes the proof.
\end{proof}

Throughout the rest of Section~\ref{Sec: Backbone}, we will work directly with the directed Bernoulli long-range percolation model.
Let $\hat{\bP}_{m, \beta_1}$ denote the probability measure associated with an independent directed Bernoulli long-range percolation on $\T^d_m$ satisfying \eqref{eq: long-range-percolation-coupling} with parameters $\beta_1$. 
% When there is no ambiguity, we simply write $\hat{\bP}_m$ to replace $\hat{\bP}_{m, \beta_1}$. 
Note that, as pointed out in Lemma~\ref{lemma: long-range-percolation-coupling}, whenever we need to take $\beta_1$ sufficiently large for the coupled percolation model, we can do so by taking $r$ sufficiently large.

A \textbf{directed path} from vertex $x$ to vertex $y$ is a sequence of vertices $(v_0, v_1, \ldots, v_k)$ where $v_0 = x$, $v_k = y$, and there exists a directed edge from $v_{i-1}$ to $v_i$ (i.e., $v_{i-1} \Rightarrow v_i$) for each $i = 1, 2, \ldots, k$. The \textbf{length} of such a path is $k$, which represents the number of edges traversed.

Define the \textbf{chemical distance} $D(x,y)$ as the length of the shortest directed path from vertex $x$ to vertex $y$ in the long-range percolation cluster. If no such directed path exists, we set $D(x,y) = +\infty$.

\begin{definition}[Spatial Homogeneity] \label{def: Spatial_hom}
Let $\ell > 0$ and $\rho \in (0,1]$. A subset $S \subseteq \T^d_m$ is said to be \textup{Hom}$(\ell,\rho)$ if every hypercube $H \subseteq \T^d_m$ with side length $\ell$ contains at least a $\rho$-fraction of its vertices from $S$, i.e.,
\[
    |S \cap H| \geq \rho \, |H|.
\]
\end{definition}

\begin{definition}[Quick Connectivity] \label{def: good_connectivity}
Let $a>0$. A subset $S \subseteq \T^d_m$ is \textup{QC}$(a)$ (quickly connected) if
\[
    \max_{x,y\in S} D(x,y)\ \le\ (\log m)^{a}.
\]
Equivalently, for every ordered pair $x,y\in S$ there exists a directed path from $x$ to $y$ of length at most $(\log m)^a$. (In particular, this implies $D(x,y)<\infty$ for all $x,y\in S$.)
\end{definition}

We are now ready to state the main proposition in Section~\ref{Sec: Backbone}:

\begin{proposition} \label{prop: S_emerge} Let $d \ge 1$ and $s \in (d, 2d)$. Then, for any $\epsilon > 0$, there exist $\beta_0, C > 0$ depending on $\epsilon, d, s$ such that $\beta_1 > \beta_0$ implies
    \begin{align*}
        \lim_{m \rightarrow \infty} \hat{\bP}_{m,\beta_1} \left( \exists \, A \subseteq \T^d_m \text{ s.t. } A \text{ is } \textup{Hom$\left(\left\lfloor(\log m)^{\frac{1}{2d - s} + \epsilon}\right\rfloor, \rho_m\right)$} \text{ and } \textup{QC$\big(\Delta + \epsilon\big)$} \right) = 1
    \end{align*}
where $\rho_m = (\frac{1}{\log \log m})^{C}$.
\end{proposition}

We make the following choices of parameters for the proof of Proposition~\ref{prop: S_emerge}: Recall that $\Delta^{-1} = \log_2(2d/s)$. We first choose $\epsilon$, $\gamma \in (s/2d, 1)$, and $\zeta \in (\gamma, 1)$ to satisfy
\begin{align}
    \frac{1}{2d - s} + \epsilon < \Delta < \frac{1}{2d\gamma - s}, \quad \frac{\log 2}{\log 1/\gamma} < \Delta + \epsilon, \quad \Delta >\frac{1}{2d\zeta - s}. \label{def: epsilon_gamma_zeta} 
\end{align}
We want to choose $\gamma$ close to $s/2d$ and $\zeta$ close to $1$. 
This is indeed possible since we have $\Delta > 1/(2d - s)$ for any $s \in (d, 2d)$. This can be verified by taking the derivative of $f(s) := \log_2 (2d/s) - (2d - s)$ and showing that $f(s) < 0$ for any $s \in (d, 2d)$.
Next, we choose $\theta$ and $\eta$ such that
\begin{align}
    \theta > \frac{1}{2d\gamma - s}, \quad \frac{1}{\zeta(2d\zeta - s)} < \eta < (\frac{1}{2d-s} + \epsilon) \cdot \zeta \label{def: theta_eta}
\end{align}
We want to choose $\eta$ close to $1/(2d - s)$ and $\theta$ turns out to be large as a consequence of $\zeta$ being close to $1$ and $\gamma$ being close to $s/2d$.

We now define a sequence of integer scales $\{ m_k \}_k$ that grow faster than exponential. Let $m_0 := 2$ and define 
\begin{align}
    k_0 &:= \inf \{ k \ge 1: 2^{\zeta^{-k}} \ge (\log m)^\eta \},  \label{def: k_0} \\[0.3em] 
    k_1 &:= \sup \{ k \ge 1: 2^{\zeta^{-k}} \le (\log m)^\theta \}, \label{def: k_1}\\[0.3em]
    m_k &:= \begin{cases}
        \hspace{3mm} \left\lfloor 2^{\zeta^{-k}}/m_{k-1} \right\rfloor \cdot m_{k-1}, & \hspace{2mm} k = 1, \ldots, k_1, \\[1.0em]
        \left( \, \left\lfloor 2^{\zeta^{-k_1}\gamma^{-(k - k_1)}}/m_{k-1} \right\rfloor \cdot m_{k-1}  \, \right) \wedge m, & \hspace{2mm} k > k_1,
    \end{cases} \label{def: L_k} \\[0.3em]
    k_2 &:= \inf \{ k \ge 1: m_k = m \}. \label{def: k_2}
\end{align}

The following lemma provides bounds on $m_k$ that will be convenient for later use. The lemma can be proved by straightforward induction, and we omit its proof.

\begin{lemma} \label{lemma: L_k asympto}
Let $\zeta \in (0, 1)$. Then for any $k \in \{0, \ldots, k_1 \}$, we have
    \begin{align}
        c_\zeta 2^{\zeta^{-k}} \le m_k \le 2^{\zeta^{-k}},
    \end{align}
where $c_\zeta := 1 - 2^{\zeta - 1}$. Moreover, for any $k \in \{ k_1 + 1, \ldots, k_2 - 1 \}$, we have
    \begin{align}
        c_{\gamma} 2^{\zeta^{-k_1}\gamma^{-(k-k_1)}} \le m_k \le 2^{\zeta^{-k_1}\gamma^{-(k-k_1)}},
    \end{align}
where $c_{\gamma} := 1 - 2^{\gamma - 1}$.
\end{lemma}

We take the following partitions on each scale of the torus. For $k \le k_2 - 1$, define
\begin{align}
    \mathcal{P}_{m, m_k} &:= \hspace{1mm} \{ B_v : v \in \T^d_{\lfloor m/m_k \rfloor} \}, \hspace{3mm} \text{where} \label{def: partitions_all_scale_1} \\[0.5em]
    B_v := \prod_{i = 1}^d I_{v_i}^{(k)} \text{ and } I_i^{(k)} &:= 
    \begin{cases}
        [i m_k, (i+1) m_k), & \text{if } i < \lfloor m/m_k \rfloor - 1, \\[0.3em]
        [i m_k, m), & \text{if } i = \lfloor m/m_k \rfloor - 1,
    \end{cases} \label{def: partitions_all_scale}
\end{align}
Then, we say a hyperbox on $\T^d_m$ is an \textbf{$m_k$-box} if it belongs to the partition $\mathcal{P}_{m, m_k}$.
In \eqref{def: L_k}, we deliberately define $m_k$ as an integer multiple of $m_{k-1}$ to avoid divisibility problems, which ensures that each $m_{k-1}$-box in $\mathcal{P}_{m, m_{k-1}}$ is contained within exactly one $m_k$-box in $\mathcal{P}_{m, m_k}$. Definition \eqref{def: partitions_all_scale} partitions $\T^d_m$ into disjoint $m_k$-boxes. For boxes indexed by $v \in \T^d_{\lfloor m/m_k \rfloor}$ with $v_i = \lfloor m/m_k \rfloor - 1$ for some $i \in \{ 1, \ldots, d \}$, they may have uneven side lengths. However, the side lengths of those boxes are constrained to lie between $m_k$ and $2m_k$.

Fix $\delta \in (0,1/2)$. We define \textbf{nice} (or $\delta$-nice) boxes inductively:

\begin{definition}\label{def: nice_box}
For $k \in \{0,1,\ldots,k_2\}$, an $m_k$-box is called \textbf{nice} if it satisfies the following:

\noindent\textup{(1)} Base case $(k=0)$: An $m_0$-box is \textbf{nice} if and only if every ordered pair of distinct vertices in the box is connected by a directed edge.

\noindent\textup{(2)} Coarse scales $(k \in \{1,\ldots,k_0\})$: An $m_k$-box is \textbf{nice} if and only if both of the following hold:
\begin{enumerate}
    \item[\textup{I}.] At least a $(1{-}\delta)$-fraction of its $m_{k-1}$-sub-boxes are nice.
    \item[\textup{II}.] For any two nice $m_{k-1}$-sub-boxes $H,H'$ contained in it, there exist vertices $u_1, u_2 \in H$ and $v_1, v_2 \in H'$ such that $u_1 \Rightarrow v_1$ and $v_2 \Rightarrow u_2$ ($H$ and $H'$ are bidirectionally connected). Moreover, we require that the endpoints $u_1, u_2, v_1, v_2$ all lie in sub-boxes that are nice at every scale $j \in \{0,\ldots,k-1\}$.
\end{enumerate}

\noindent\textup{(3)} Dense scales $(k \ge k_0{+}1)$: An $m_k$-box is \textbf{nice} if and only if both of the following hold:
\begin{enumerate}
    \item[\textup{I}.] All of its $m_{k-1}$-sub-boxes are nice.
    \item[\textup{II}.] For any two $m_{k-1}$-sub-boxes $H,H'$ contained in it, there exist vertices $u_1, u_2 \in H$ and $v_1, v_2 \in H'$ such that $u_1 \Rightarrow v_1$ and $v_2 \Rightarrow u_2$, with endpoints $u_1, u_2, v_1, v_2$ lying in sub-boxes that are nice at every scale $j \in \{0,\ldots,k-1\}$.
\end{enumerate}
\end{definition}

Let $a_k$ denote the maximum over all $m_k$-boxes of the probability of failing to be nice under \(\hat{\bP}_{m, \beta_1}\). The following lemma upper bounds $a_k$ for all the coarse scales.

\begin{lemma} \label{lemma: recursion}
Let parameters $\gamma$, $\theta$, $\zeta$, and $\eta$ be chosen according to \eqref{def: epsilon_gamma_zeta}-\eqref{def: theta_eta}, and fix $\delta \in (0, 1/2)$. 
Then, there exists a positive constant $\beta_0$ such that $\beta_1 \ge \beta_0$ implies that
\begin{align}
    a_k \le \exp\left(- \beta_1 c^{k+1} m_k^{\zeta(2d\zeta - s)} \right)
\end{align}
holds for any $k \in \{0, ..., k_0\}$, where $c$ is a positive constant depending only on $\zeta, \delta, d, s$.
\end{lemma}

\begin{proof}
    Let $H$ be an $m_k$-box and let $n$ be the number of $m_{k-1}$-boxes contained in $H$.
    Label these sub-boxes by $H_1,\dots,H_n$.
    Let $\mathcal{G}$ be the event that $H$ is nice, and let $\mathcal{C}$ be the event that at least a $(1{-}\delta)$-fraction of the $n$ sub-boxes $H_1,\dots,H_n$ are nice.
    Then
    \begin{align}
        \hat{\bP}_{m, \beta_1}(\mathcal{G}^c) \le \hat{\bP}_{m, \beta_1}(\mathcal{C}^c) + \hat{\bP}_{m, \beta_1}(\mathcal{G}^c \mid \mathcal{C}). \label{ineq: two_bad_events_for_not_nice}
    \end{align}

    We first bound the probability that at least $\delta n$ of the $m_{k-1}$-boxes are not nice.
    The number of not-nice $m_{k-1}$-boxes is stochastically dominated by a Binomial random variable $X\sim\mathrm{Bin}(n,a_{k-1})$.
    By Chernoff bound, for any $c'>0$,
    \begin{align}
        \hat{\bP}_{m,\beta_1}(\mathcal{C}^c)
        &\le \hat{\bP}_{m,\beta_1}(X \ge \delta n)
        \le e^{-c' \delta n}\big(e^{c'} a_{k-1} + (1-a_{k-1})\big)^n \nonumber \\
        &\le a_{k-1}^{\delta n}(2-a_{k-1})^n
        \le (2 a_{k-1}^{\delta})^n,
    \end{align}
    where in the second line we chose $c'$ so that $e^{-c'}=a_{k-1}$.

    By Lemma \ref{lemma: L_k asympto}, the number \(n\) of \(m_{k-1}\)-sub-boxes in an \(m_k\)-box satisfies
    \begin{align}
    \left(\frac{m_k}{2m_{k-1}}\right)^d \le n \le \left(\frac{2m_k}{m_{k-1}}\right)^d
    \;\Longrightarrow\;
    c_2\,m_k^{d(1-\zeta)} \le n \le C_2\,m_k^{d(1-\zeta)}, \label{ineq: num_of_sub-boxes_bounds}
    \end{align}
    where $c_2:=2^{-d}c_\zeta^{d}$ and $C_2:=2^{d}c_\zeta^{-d}$.
    In particular, if $2 a_{k-1}^\delta < 1$, then
    \begin{align}
    \hat{\bP}_{m, \beta_1}(\mathcal{C}^c) \le \big(2 a_{k-1}^\delta\big)^{c_2\, m_k^{d(1-\zeta)}} \le \big(2 a_{k-1}^\delta\big)^{c_3\, (2^{K_1})^{\zeta^{-k}}}, \label{ineq: first_term_recursion}
    \end{align}
    where $c_3 := c_2 c_\zeta^{d(1-\zeta)}$ and $K_1 := d(1-\zeta)$.
    In the last step we used the lower bound on $m_k$ from Lemma \ref{lemma: L_k asympto}.

    We next bound the second term in \eqref{ineq: two_bad_events_for_not_nice}.
    For $k \in \{1,\ldots,k_0\}$, conditioned on $\mathcal{C}$, at least a $(1{-}\delta)$-fraction of the $m_{k-1}$-sub-boxes are nice.
    By a union bound over ordered pairs of nice sub-boxes and using independence of long-range edges together with the upper bound
    $\exp(-\beta_1 |u-v|^{-s})$ on the edge-close probability, we get
    \begin{align}
        \hat{\bP}_{m, \beta_1}(\mathcal{G}^c \mid \mathcal{C})
        &\le 2 \binom{n}{2}\,
             \exp\!\left(
                 -\,\beta_1\,
                 \frac{\bigl((m_{k-1})^{d}(1-\delta)^{k-1}\bigr)^2}{(d \cdot 2 m_k)^s}
             \right) \nonumber \\[0.3em]
        &\le n^2 \exp\!\left(
                 -\,(2d)^{-s}\,\beta_1\,c_\zeta^{2d}\,(1-\delta)^{2(k-1)}\,m_k^{\,2d\zeta - s}
             \right) \nonumber \\[0.3em]
        &\le C_2^2\, m_k^{\,2d(1-\zeta)}\,
             \exp\!\left(
                 -\,c_4 \beta_1 \,(1-\delta)^{2(k-1)}\,m_k^{\,2d\zeta - s}
             \right), \label{ineq: second_term_recursion}
    \end{align}
    where in the second line we used $2 \binom{n}{2}\le n^2$ and Lemma \ref{lemma: L_k asympto} to replace $(m_{k-1})^{2d}$ by $c_\zeta^{2d} m_k^{2d\zeta}$, and in the third line we used $n \le C_2\, m_k^{d(1-\zeta)}$ from \eqref{ineq: num_of_sub-boxes_bounds}; we also set $c_4 := (2d)^{-s} c_\zeta^{2d}$.

    Since we have $c_\zeta 2^{\zeta^{-k}} \le m_k \le 2^{\zeta^{-k}}$, $\delta < 1/2$, and $\beta_1 \ge 100$, \eqref{ineq: second_term_recursion} is upper bounded by
    \begin{align}
        & C_2^2\, 2^{2d(1-\zeta)\zeta^{-k}}\,
             \exp\!\left(
                 -\,c_4 \beta_1 c_\zeta^{2d\zeta - s} \,(1-\delta)^{2(k-1)}\,(2^{\zeta^{-k}})^{\,2d\zeta - s}
             \right) \nonumber \\[0.3em]
        \le\;& C_3 
             \exp\!\left(
                 -\,c_5 \beta_1 \,2^{-2k}\,(2^{2d\zeta - s})^{\zeta^{-k}}
             \right) 
        \le C_3
             \exp\!\left(
                 -\,c_6 \beta_1 \, (2^{\zeta(2d\zeta - s)})^{\zeta^{-k}}
             \right), \label{ineq: second_term_recursion_simplified} 
    \end{align}
    where $c_5 := c_4 c_\zeta^{2d\zeta - s}/2$, and $c_6, C_3$ are some positive constants depending only on $\zeta, d, s$.
    By observing that $2^{\zeta^{-k}}$ grows much faster than exponential functions in $k$, one can indeed verify the existence of $c_6$ and $C_3$ on $(0, \infty)$. 
    For later use, define $K_2 := \zeta(2d\zeta - s)$.

    We claim that, by choosing $\beta_1$ sufficiently large (depending on $\zeta, d, s$), there exists a constant $c_7\in(0,1)$, depending on $\zeta, d, s$, such that for all $k \in \{0, \ldots, k_0\}$,
    \begin{align}
        a_k \le \exp\!\big(- \beta_1 \delta^{k} c_7^{k+1} (2^{K_2})^{\zeta^{-k}}\big), \label{eq: induction_goal}
    \end{align}
    which implies the statement of the lemma since $m_k \le 2^{\zeta^{-k}}$ by Lemma \ref{lemma: L_k asympto}.

    We proceed by induction on $k$.
    For $k=0$, an $m_0$-box has side-length at most $3$, hence
    $a_0 \le 9^2 \exp(- \beta_1 / 4^s)$.
    Choosing $c_7 \le 2^{-(2s+1)-K_2}$ and $\beta_1$ large enough (depending on $s$) ensures $a_0 \le \exp(- \beta_1 c_7 2^{K_2})$, establishing the base case.

    Assume \eqref{eq: induction_goal} holds at level $k-1$.
    Substituting $a_{k-1}$ into \eqref{ineq: first_term_recursion} and \eqref{ineq: second_term_recursion_simplified}, we obtain
    \begin{align}
        a_k &\le (2 a_{k-1}^\delta)^{c_3 (2^{K_1})^{\zeta^{-k}}} + C_3 \exp\!\big(- c_6 \beta_1 (2^{K_2})^{\zeta^{-k}}\big) \nonumber \\
        &\le 2^{c_3 (2^{K_1})^{\zeta^{-k}}}
            \exp\!\Big(- \, c_3 \beta_1 \delta^{k} c_7^{k}\,(2^{K_2})^{\zeta^{-(k-1)}} (2^{K_1})^{\zeta^{-k}}\Big)
            + C_3 \exp\!\big(- c_6 \beta_1 (2^{K_2})^{\zeta^{-k}}\big). \nonumber
    \end{align}
    Since $\zeta K_2 + K_1 > K_2$ for any $\zeta\in(0,1)$ and $s\in(d,2d)$, the exponential function of the first term is at most 
    $\exp\!\big(-c_3 \beta_1 \delta^k c_7^{k} (2^{K_2})^{\zeta^{-k}}\big)$.
    Also, since $K_1 < K_2$, 
    for all $\beta_1 \ge 100$ (100 can be replaced by any sufficiently large positive constant), 
    there exists $C_4=C_4(\delta, \zeta,d,s)$ such that the first term is bounded by
    $C_4 \exp\!\big(- \tfrac{c_3}{2} \beta_1 \delta^k c_7^{k} (2^{K_2})^{\zeta^{-k}}\big)$.
    We remark that $\zeta K_2 + K_1 > K_2$ can be easily verified by noticing that it is equivalent to verifying $(1 - \zeta)(d - K_2) > 0$, which holds since $K_2 \le 2d - s < d$ and $\zeta \in (0,1)$.
    Also, one can verify that $K_1 < K_2$ by examining the derivative of the function $g(\zeta) := \zeta(2d\zeta - s) - d(1 - \zeta)$ and using the facts that $s < 2d$ and $\zeta > s/2d$.
    We omit the details here.
    Thus, by choosing $c_7$ such that $c_7 \le \min\{c_3/4,\,c_6/2,1\}$, we have
    \begin{align*}
        a_k
        \le \, &C_4 \exp\!\Big(- \tfrac{c_3}{2} \beta_1 \delta^k c_7^{k} (2^{K_2})^{\zeta^{-k}}\Big)
          + C_3 \exp\!\big(- c_6 \beta_1 (2^{K_2})^{\zeta^{-k}}\big) \\
        \le \, &(C_3+C_4)\exp\!\big(- 2 \beta_1 c_7^{k+1} \delta^k (2^{K_2})^{\zeta^{-k}}\big).
    \end{align*}
    Finally, by possibly increasing $\beta_0$ to absorb the prefactor $C_3+C_4$ yields
    $a_k \le \exp\!\big(- \beta_1 \delta^k c_7^{k+1} (2^{K_2})^{\zeta^{-k}}\big)$,
    which is \eqref{eq: induction_goal}.
    The induction is complete, and the lemma follows.
\end{proof}

\begin{lemma} \label{lemma: k_0_all_nice}
    Under the same assumptions and the choice of $\beta_0$ as in Lemma \ref{lemma: recursion}, if $\mathcal{G}_{k_0}$ denotes the event that all $m_{k_0}$-boxes are nice, then $\beta_1 \ge \beta_0$ implies
    \begin{align}
        \lim_{m \to \infty} \hat{\bP}_{m,\beta_1}(\mathcal{G}_{k_0}^c) = 0.
    \end{align}
\end{lemma}

\begin{proof}
    By Lemma \ref{lemma: recursion} and applying a union bound, we obtain
    \begin{align}
        \hat{\bP}_{m,\beta_1}(\mathcal{G}_{k_0}^c) \le m^d \exp( - \beta_1 c^{k_0+1} m_{k_0}^{\zeta(2d\zeta - s)} ). \label{ineq: upper_bound_for _k_0}
    \end{align}
By definition of $k_0$ and Lemma \ref{lemma: L_k asympto}, we have $ m_{k_0} \ge c_\zeta (\log m)^\eta $. It follows from assumption \eqref{def: theta_eta} that $ \eta (2d\zeta - s) \zeta > 1 $. Therefore, the right hand side of \eqref{ineq: upper_bound_for _k_0} goes to zero as $ m $ diverges. 
\end{proof}

\begin{lemma} \label{lemma: T_L_is_nice}
    Under the same assumptions and the choice of $\beta_0$ as in Lemma \ref{lemma: recursion}, let $\mathcal{G}$ denote the event that $\T^d_m$ is nice, i.e., the $m_{k_2}$-box is nice. Then $\beta_1 \ge \beta_0$ implies
    \begin{align}
        \lim_{m \to \infty} \hat{\bP}_{m,\beta_1}(\mathcal{G}^c) = 0.
    \end{align}
\end{lemma}

\begin{proof}
    Let $\mathcal{G}_k$ be the event that all $m_k$-boxes are nice. 
    By Definition \ref{def: nice_box}, we have $\mathcal{G}_{k_0} \supseteq \mathcal{G}_{k_0 + 1} \supseteq \cdots \supseteq \mathcal{G}_{k_2} = \mathcal{G}$. Therefore, we can write
    \begin{align}
        \hat{\bP}_{m, \beta_1}\left(\mathcal{G}^c \right) \le \hat{\bP}_{m, \beta_1}(\mathcal{G}_{k_0}^c) + \sum_{k = k_0}^{k_2 - 1} \hat{\bP}_{m, \beta_1}(\mathcal{G}^c_{k+1} | \mathcal{G}_k).
    \end{align}
    By Lemma \ref{lemma: k_0_all_nice}, we know that the first term approaches $0$ as $m$ tends to infinity. For $k \in \{k_0, \ldots, k_2 - 1 \}$, we have 
    \begin{align}
        \hat{\bP}_{m, \beta_1}(\mathcal{G}^c_{k+1} | \mathcal{G}_k) \le m^{2d} \cdot 2 \cdot \exp\left(- \beta_1 \frac{m_{k}^{2d}}{(d \cdot 2m_{k+1})^s} (1 - \delta)^{k_0}\right),
    \end{align}
    where $m^{2d}$ is an upper bound for the number of ordered pairs of $m_{k+1}$-boxes in $\T^d_m$ and $2$ accounts for the requirement of bidirectional connections. Moreover, the exponential term bounds above the probability that there is no desired directed edge from $B_1$ to $B_2$, where $B_1$ and $B_2$ are two arbitrary distinct $m_k$-boxes contained inside an $m_{k+1}$-box $B$. Indeed, $d \cdot 2 m_{k+1}$ is an upper bound for the $\ell_1$-distance between any two vertices in $B$ and $(m_k)^d (1 {-} \delta)^{k_0}$ is a lower bound for the number of vertices that are qualified to serve as endpoints of the directed connection in either $B_1$ or $B_2$. 

    By invoking Lemma \ref{lemma: k_0_all_nice}, we know that:
    \begin{align*}
        m_k &\ge \begin{cases}
            \, c_\zeta \, 2^{\zeta^{-k}}, & k = k_0, \ldots, k_1, \\[0.2em]
            \, c_\gamma \, 2^{\zeta^{-k_1} \gamma^{-(k - k_1)}}, & k = k_1 + 1, \ldots, k_2-1,
        \end{cases} 
        \\
        m_{k+1} &\le \begin{cases}
            \, 2^{\zeta^{-(k+1)}}, & k = k_0, \ldots, k_1 - 1, \\[0.2em]
            \, 2^{\zeta^{-k_1} \gamma^{-(k - k_1 + 1)}}, & k = k_1, \ldots, k_2 - 2. \\
        \end{cases}
    \end{align*}
    Let $c := c_\zeta \vee c_\gamma.$ It then follows that $m_k^{2d}/m_{k+1}^s$ is at least 
    \begin{align*}
        \begin{cases}
            \, c^{2d} \, \left( 2^{\zeta^{-(k+1)}} \right)^{2d\zeta - s} \ge c^{2d} \cdot (\log m)^{\eta (2d\zeta - s)} \cdot 2^{\zeta^{-(k - k_0 + 1)}}, & k = k_0, \ldots, k_1 - 1, \\[0.2em]
            \, c^{2d} \, \left( 2^{\zeta^{-k_1} \gamma^{-(k - k_1 + 1)}} \right)^{2d\gamma - s} \ge c^{2d} \cdot (\log m)^{\theta (2d \gamma - s)} \cdot 2^{\gamma^{-(k - k_1)}}, & k = k_1, \ldots, k_2 - 1.
        \end{cases}
    \end{align*}
    Since we have $\min \{ (2d\zeta - s) \eta, (2d\gamma - s) \theta \} > 1$, it follows that $m_k^{2d} / m_{k+1}^s \ge c^{2d} (\log m)^{1 + \epsilon'}$ for some small $\epsilon' \in (0, 1)$ and for all $k \in \{ k_0, \ldots, k_2 - 1 \}$. On the other hand, since $k_2 \asymp \log \log m$ and $k_0 \asymp \log \log \log m$, we have
    \begin{align*}
        \lim_{m \to \infty} \sum_{k = k_0}^{k_2 - 1} \hat{\bP}_{m, \beta_1}(\mathcal{G}^c_{k+1} | \mathcal{G}_k) &\le \lim_{m \to \infty} 2 k_2 m^{2d} \cdot \exp\left(- \beta_1 c^{2d} (1 - \delta)^{k_0} (\log m)^{1 + \epsilon'} \right) = 0,
    \end{align*}
    which completes the proof.
\end{proof}

\begin{proof}[Proof of Proposition~\ref{prop: S_emerge}]
    We condition on the occurrence of the event $\mathcal{G}$ defined in Lemma \ref{lemma: T_L_is_nice}.
    Let $A$ denote the set of vertices that are contained in nice $m_k$-boxes at all scales, i.e., for all $k \in \{ 0, \ldots, k_2 \}$. 
    We shall analyze both the diameter of the (long-range) percolation cluster induced by $A$ and the density of $A$ inside $\T^d_m$.

    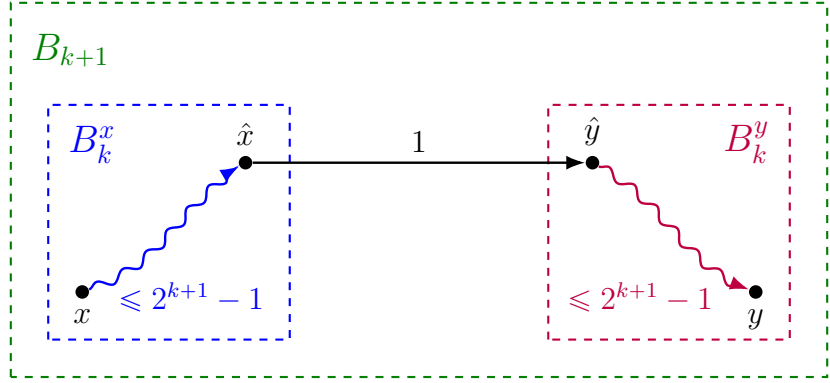
\begin{figure}
    \centering
    \begin{tikzpicture}[
        scale=0.9,
        every node/.style={font=\large},
        box/.style={dashed, line width=0.8pt},
        point/.style={circle, fill=black, inner sep=1.8pt},
        directed/.style={-{Latex[length=2.5mm]}, line width=0.9pt},
        wavy/.style={
            directed,
            decorate,
            decoration={snake, amplitude=0.6mm, segment length=4mm}
        }
    ]
    
    % Outer box B_{k+1}
    \draw[box, green!50!black] (0,0) rectangle (12,5.5);
    \node[green!50!black, anchor=south west, font=\Large] at (0.15,4.40) {$B_{k+1}$};
    
    % Left subbox B_k^x
    \draw[box, blue] (0.55,0.55) rectangle (4.1,4);
    \node[blue, anchor=south west, font=\Large] at (0.7,3.00) {$B_k^x$};
    
    % Right subbox B_k^y
    \draw[box, purple] (7.9,0.55) rectangle (11.45,4);
    \node[purple, anchor=south east, font=\Large] at (11.3,3.00) {$B_k^y$};
    
    % Points
    \node[point, label=below:{$x$}] (x) at (1.05,1.25) {};
    \node[point, label=above:{$\hat{x}$}] (xh) at (3.45,3.15) {};
    
    \node[point, label=above:{$\hat{y}$}] (yh) at (8.55,3.15) {};
    \node[point, label=below:{$y$}] (y) at (10.95,1.25) {};
    
    % Inductive paths inside the two m_k-boxes
    \draw[wavy, blue] (x) to[out=20,in=210] (xh);
    \draw[wavy, purple] (yh) to[out=-30,in=160] (y);
    
    % Connecting directed edge
    \draw[directed] (xh) -- (yh);
    \node[above] at ($(xh)!0.5!(yh)$) {$1$};
    
    % Length annotations
    \node[blue] at (2.65,1.18) {$\leq 2^{k+1}-1$};
    \node[purple] at (9.25,1.18) {$\leq 2^{k+1}-1$};
    
    % Bottom formula
    % \node[font=\Large] at (6,-0.75)
    % {
    %     $(2^{k+1}-1)+1+(2^{k+1}-1)=2^{k+2}-1$
    % };
    
    \end{tikzpicture}
    \caption{The induction step for constructing a directed path from $x$ to $y$ inside a nice $m_{k+1}$-box $B_{k+1}$. The blue (resp. red) directed wavy arrow from $x$ to $\hat{x}$ (resp. $y$ to $\hat{y}$) is a path consisting of at most $2^{k+1}{-}1$ open direct long-range edges. The black arrow from $\hat{x}$ to $\hat{y}$ is a single open directed long-range edge. In total, the figure illustrates an open path of length at most $2^{k+2}{-}1$ from $x$ to $y$ within $B_{k+1}$.}
    \label{fig:multiscale-directed-path}
    \end{figure}

    We first focus on the diameter. 
    Our first goal is to construct a directed path of length at most $2^{k_2 + 1} {-} 1$ from $x$ to $y$ for any ordered pair of distinct vertices $x, y \in A$.
    We proceed by induction on the scale index $k$.
    Consider a $m_k$-box $B_k$ and let $A(B_k)$ be the subset of $B_k$ whose elements are vertices
    contained in nice boxes at all scales less than or equal to $k$. 
    In particular, if $B_k$ is not nice, then $A(B_k)$ is empty. 
    When there is no ambiguity, we simply write $A_k$ for $A(B_k)$.
    We will show by induction that for any $k = 0, \ldots, k_2$, any $m_k$-box $B_k$, and any two distinct vertices $x, y \in A_k$, 
    there exists a directed path with length at most $2^{k+1} {-} 1$ connecting $x$ to $y$:
    \begin{itemize}[left=15pt, itemsep=0.3em, topsep=1em, rightmargin=25pt]
    \item For the base case, according to Definition \ref{def: nice_box}, 
    if $B_0$ is a nice $m_0$-box, then for any distinct $x, y \in B_0$, 
    there exists a directed edge (a path of length 1) connecting them, which proves the base case. 
    \item Let $k$ be such that $0 \le k < k_2$ and assume that for every nice $m_k$-box $B_k$ and every distinct $x, y \in A_k$, there exists a directed path with length at most $2^{k+1} - 1$ connecting $x$ to $y$. 
    Consider a nice $m_{k+1}$-box $B_{k+1}$. 
    For $x, y \in A_{k+1}$, let $B^x_k$ (resp. $B^y_k$) be the $m_k$-box containing $x$ (resp. $y$). 
    Since $x, y \in A_{k+1}$, both $B^x_k$ and $B^y_k$ are nice. 
    If $B^x_k = B^y_k$, then we are done by the induction hypothesis. If $B^x_k \neq B^y_k$, then by Definition~\ref{def: nice_box}, there exist vertices $\hat{x} \in A(B_k^x)$ and $\hat{y} \in A(B_k^y)$ such that the directed edge $(\hat{x}, \hat{y})$ is present. Since both $B_k^x$ and $B_{k+1}$ are nice, we have $A(B^x_k) = A_{k+1} \cap B_k^x$. By the induction hypothesis, there exists a directed path of length at most $2^{k+1} {-} 1$ connecting $x$ to $\hat{x}$.
    Likewise, there is another directed path of length at most $2^{k + 1} {-} 1$ connecting $\hat{y}$ to $y$. 
    Together with the directed edge $(\hat{x}, \hat{y})$, we conclude that $x$ and $y$ are connected by a directed path of length at most $2^{k + 2} {-} 1$. 
    It completes the induction step.
    \end{itemize}
    
    We now bound the quantity $2^{k_2 + 1} - 1$.
    It follows from the definitions of $k_1$ and $k_2$ that 
    \begin{align*}
        2^{k_2 - k_1} = \left(\gamma^{-(k_2-k_1)}\right)^{\log_{\gamma^{-1}} 2} = \left( \frac{\zeta^{-k_1} \gamma^{-(k_2 - k_1)}}{\zeta^{-k_1}} \right)^{\log_{\gamma^{-1}} 2} \le \left( \frac{1}{ \theta \zeta \gamma } \cdot \frac{\log m }{\log \log m} \right)^{\log_{\gamma^{-1}} 2}.
    \end{align*}
    Indeed, the definitions give that $2^{\zeta^{-k_1} \gamma^{-(k_2 - k_1 - 1)}} \le m$ and $2^{\zeta^{-(k_1+1)}} \ge (\log m)^\theta$.
    On the other hand, we have $2^{\zeta^{-k_1}} \le (\log m)^\theta$, which implies that
    \begin{align*}
        2^{k_1} = \left( \zeta^{-k_1} \right)^{\log_{\zeta^{-1}} 2} \le (\theta \log_2 \log m)^{\log_{\zeta^{-1}} 2}.
    \end{align*}
    Since $\gamma$ is chosen such that
    \begin{align}
        \log_{\gamma^{-1}} 2 < \Delta + \epsilon, \nonumber
    \end{align}
    it is straightforward to verify that $2^{k_2 + 1} < (\log m)^{\Delta + \epsilon}$ for sufficiently large $m$.

    Next, we analyze the density of $A$ inside $\T^d_m$. On the event $\mathcal{G}$, by Definition~\ref{def: nice_box}, every $m_k$-box is nice for all $k \in \{k_0, \ldots, k_2\}$. 
    Therefore, for every $k_0$-box $B_{k_0}$, on the event $\mathcal{G}$, we must have $A \cap B_{k_0} = A_{k_0}$, where $A_{k_0}$ denotes the set of vertices in $B_{k_0}$ that are contained in nice boxes at all scales up to $k_0$. By the definition of nice boxes from scale $0$ to scale $k_0$, we know that at each scale $k \in \{1, \ldots, k_0\}$, at most a $\delta$-fraction of the $m_{k-1}$-sub-boxes are not nice. It then follows that 
    \begin{align*}
        \frac{|A \cap B_{k_0}|}{|B_{k_0}|} = \frac{|A_{k_0}|}{|B_{k_0}|} \ge (1 - \delta)^{k_0}.
    \end{align*}
    Note that by the definition of $k_0$ in \eqref{def: k_0}, we have 
    $m_{k_0} \le (\log m)^{\eta \zeta^{-1}}.$
    It then follows from \eqref{def: theta_eta} that $\eta \zeta^{-1} < 1/(2d - s) + \epsilon.$ 
    Therefore, $m_{k_0} \ll (\log m)^{1/(2d - s) + \epsilon}$ when $m$ goes to infinity, where we write $f(m) \ll g(m)$ if $\lim_{m \to \infty} f(m)/g(m) = 0$.
    Moreover, the definition of $k_0$ also implies that
    \begin{align*}
        2^{\zeta^{-(k_0 - 1)}} < (\log m)^{\eta} \; \Rightarrow \; \zeta^{- k_0} < \frac{\eta}{\zeta \, \log 2} \cdot \log \log m,
    \end{align*}
    which further implies that, for sufficiently large $m$,
    \begin{align*}
        (1 - \delta)^{k_0} = \left(\zeta^{- k_0}\right)^{- \log_{\zeta} (1 - \delta)} > \left( \log \log m \right)^{- 2 \log_{\zeta} (1 - \delta)}.
    \end{align*}
    
    Let $C := 2 \log_{\zeta} (1 - \delta)$.
    Combining the two facts that (1) $m_{k_0} \ll (\log m)^{1/(2d - s) + \epsilon}$ and (2) $(1 - \delta)^{k_0} > (\log \log m)^{-C}$ shows that $A$ is \textup{Hom$\left((\log m)^{\frac{1}{2d - s} + \epsilon}, \rho_m\right)$} with $\rho_m = (\log \log m)^{-C}$. 
\end{proof}

\subsection{Upper Bounds on Activation Time Under Finite Lifespan} \label{Sec: upperbound_finite_lifespan}

In this section, we prove Theorem~\ref{Mainthm: activate_constant_fraction}, which shows that a ``large'' fraction of vertices in $\T^d_L$ are activated within polylogarithmic time when the particle lifespan is sufficiently large for $\alpha \in (0, d)$, and within linear time when $\alpha \ge d$.
We state a stronger version of Theorem~\ref{Mainthm: activate_constant_fraction} in both regimes $\alpha \in (0, d)$ and $\alpha \ge d$; see Proposition~\ref{prop: activate_constant_fraction} and \ref{prop: activate_constant_fraction_alpha>=d}. The proof of Proposition~\ref{prop: activate_constant_fraction} relies on both Section~\ref{Sec: BernPerc} and \ref{Sec: Backbone}, while the proof of Proposition~\ref{prop: activate_constant_fraction_alpha>=d} relies only on Section~\ref{Sec: BernPerc}.

Recall that $\ell_o$ denotes the lifespan of the planted particle at the origin $o$. In the sequel, we will take $\ell_o$ to depend on the side length $L$ of the torus. We write $\ell_o \gg 1$ to indicate that $\ell_o \to \infty$ as $L \to \infty$.
\begin{proposition} \label{prop: activate_constant_fraction}
    Let $\lambda>0$, $d \ge 2$, and $\alpha \in (0, d)$.
    Let $Q \in \mathcal{H}_{\alpha, K}$ and let $Q_L$ be defined as in \eqref{def: Q_L}. 
    Assume that $\ell_o \gg 1$.
    For any $\rho \in (0, 1)$ and $\epsilon > 0$, there exist positive $C_1$ and $N_0$ such that for any particle lifespan $\ell \ge N_0$, we have 
    \begin{align}
        \lim_{L \rightarrow \infty} \bP^+_L \left( \,
        \mathcal{A}\left((\log L)^{\Delta + \epsilon}\right) \textup{ is } \textup{Hom$\left(\left\lfloor C_1(\log L)^{\frac{1}{d-1}} \right\rfloor, \rho\right)$} \, 
        \right) = 1.
    \end{align}
\end{proposition}

\begin{proposition} \label{prop: activate_constant_fraction_alpha>=d}
    Let $\lambda>0$, $d \ge 2$, and $\alpha \ge d$.
    Let $Q \in \mathcal{H}_{\alpha, K}$ and let $Q_L$ be defined as in \eqref{def: Q_L}. 
    Assume that $\ell_o \gg 1$.
    For any $\rho \in (0, 1)$, there exist positive $C_1, C_2,$ and $N_0$ such that for any particle lifespan $\ell \ge N_0$, we have 
    \begin{align}
        \lim_{L \rightarrow \infty} \bP^+_L \left( \,
            \mathcal{A}\left(C_2 L\right) \textup{ is } \textup{Hom$\left(\left\lfloor C_1 (\log L)^{\frac{1}{d-1}} \right\rfloor, \rho\right)$} \, 
        \right) = 1.
    \end{align}
\end{proposition}

Let $\T^d_m$ be the renormalized torus at scale-0 defined as in \eqref{def: m}-\eqref{def: partition_const_level} for some sufficiently large $r \le L$ to be determined later, where each vertex $v \in \T^d_m$ is associated with a box $B_v$ of side-length between $r$ and $2r$ as in \eqref{def: box_const_level}. 

\begin{definition}[Fantastic vertex] \label{def: fantastic_vertex}
We say $v \in \T^d_m$ is $\mathbf{fantastic}$ if the planted particle $w_0^o$ visits at least $\phi(s)$ number of vertices in $B_v$, where
    \begin{align}
        \phi(s) = \begin{cases}
            \, \frac{c_1 s}{2}, &\textup{if $d \ge 3$ or $(d = 2$ and $\alpha > 2)$,} \\[0.2em]
            \, \frac{c_1 s}{2 \log \log s}, &\textup{if $d = 2$ and $\alpha = 2$,} \\[0.2em]
            \, \frac{c_1 s}{2 \log s}, &\textup{if $d = 2$ and $\alpha < 2$,}
        \end{cases} \quad
        s = \begin{cases}
        \, r^2, &\textup{if } \alpha > 2, \\[0.2em]
        \, \frac{r^2}{\log r}, &\textup{if } \alpha = 2, \\[0.2em]
        \, r^\alpha, &\textup{if } \alpha < 2,
    \end{cases} 
    \end{align}
    and $c_1$ is the same constant in Lemma~\ref{lemma: range_estimates_main_text}. Let $\mathrm{Fant}$ denote the set of all $\mathrm{fantastic}$ vertices.
\end{definition}

The next lemma shows that when $\ell_o \gg 1$, with high probability the number of fantastic vertices diverges as $L \to \infty$.

\begin{lemma} \label{lemma: planted_to_fantastic}
    Let $\lambda > 0$, $d \ge 2$, and $\alpha > 0$.
    If $\ell_o \gg 1$, then for any $C > 0$, we have
    \begin{align}
    	\lim_{L \rightarrow \infty} \bP_L^+\left( \, \left|\mathrm{Fant}\right| > C \, \right) = 1.
    \end{align}
\end{lemma}
Lemma~\ref{lemma: planted_to_fantastic} can be proved using Lemma~\ref{lemma: range_estimates_main_text} and the following observation. Let $X$ be a walk starting from $o$. Let $\tau_0 := 0$, and let $v_0 \in \T^d_m$ be the (unique) vertex such that $o \in B_{v_0}$. Define
\[\tau_1 := \inf \{t > 2s : X_t \in B_v \textup{ for some } v \neq v_0\}.\]
Inductively, for $i \ge 1$, let $v_i \in \T^d_m$ be the vertex such that $X_{\tau_i} \in B_{v_i}$, and set
\[\tau_{i+1} := \inf \left\{t > \tau_i + 2s : X_t \in B_v \textup{ for some } v \notin \{v_0, v_1, \ldots, v_i\} \right\}.\]
One can show that for any fixed $k$ we have $\tau_k \le \ell_o$ with overwhelming probability as $L \to \infty$.
On the event $\{\tau_k \le \ell_o\}$, it follows from Lemma~\ref{lemma: range_estimates_main_text} that $|\mathrm{Fant}|$ stochastically dominates a $\mathrm{Bin}(k, p)$ random variable for some constant $p \in (0, 1)$.\footnote{Note that, although Lemma~\ref{lemma: range_estimates_main_text} is stated on $\Z^d$, by a standard coupling argument, the same result holds for a random walk driven by $Q_L$ on $\T^d_L$.} By taking $k$ sufficiently large, we can make $\bP(\mathrm{Bin}(k, p) \le C)$ arbitrarily small. We omit the details for brevity.
	
\begin{proof}[Proof of Proposition~\ref{prop: activate_constant_fraction} and \ref{prop: activate_constant_fraction_alpha>=d}]
	Assume without loss of generality that $1 \ll \ell_o \le (\log L)^{\Delta}$. 
    (If $\ell_o$ is larger, we may truncate the planted walk after $(\log L)^{\Delta}$ steps; this can only make activation harder.)
    For the purpose of this proof, it is convenient to use Poisson thinning to split the particles at each vertex into two independent collections $\cW_{x}^{(1)}$ and $\cW_{x}^{(2)}$, each containing $\mathrm{Pois}(\lambda/2)$ particles.
    The first half, $\cW_{x}^{(1)}$, is used to construct an auxiliary Bernoulli site percolation on $\T^d_m$, primarily utilizing the results presented in Section \ref{Sec: BernPerc}. 
    The remaining half, $\cW_{x}^{(2)}$, is applied to execute the Bernoulli long-range percolation strategy, as discussed in Section \ref{Sec: Backbone}. 
    The detailed proof of Proposition~\ref{prop: activate_constant_fraction} commences as follows. 
	
	Let $r$ be a positive integer to be determined, which will be taken sufficiently large.
    Throughout the proof, we always choose $m = \lfloor L / r \rfloor$ and the particle lifespan as $\ell = \ell(r)$ (see \eqref{def: s(r)}). 
    Let $\T^d_m$ be the renormalized torus at scale-0 defined for the chosen $r$. 
	
	Firstly, we construct an auxiliary Bernoulli site percolation on $\T^d_m$, utilizing only particles in $\cW_{x}^{(1)}$. 
    Recall that for each box $B_v$, we defined $\widehat{B}_v$ as the union of $B_v$ and its $2d$ neighboring boxes. 
    For any $x \in B_v$, we use $\widehat{\mathcal{A}}_x^{(1)}$ to denote the set of vertices activated by the frog model dynamics with lifespan $\ell$ starting with all particles in $\cW_{x}^{(1)}$ being activated, all particles in $\cup_{y \in B_v, y \neq x} \cW_{y}^{(1)}$ being sleeping/inactive, and no particle exists outside of $B_v$.
    Let $\cW_{B_v}^{(1)} := \cup_{y \in B_v} \cW_{y}^{(1)}$.
    % when initially all particles outside of $\cW_{B_v}^{(1)} := \cup_{y \in B_v} \cW_{y}^{(1)}$ are labeled as removed, all particles in $\cW_{x}^{(1)}$ (at $x$) are activated at time 0, and all other particles are sleeping/inactive at time 0.
    We now say that a vertex $v \in \T^d_m$ is \textbf{open} if one of the following conditions is satisfied: 
    \begin{itemize}[left=15pt, itemsep=0.3em, topsep=1em, rightmargin=25pt]
    \item[•] $v \in \mathrm{Fant}$ and there exists $x \in B_v$ such that $w_0^o$ visits $x$ and $\widehat{\mathcal{A}}_x^{(1)} \supseteq \widehat{B}_v$,
    \item[•] $v \in \mathrm{Fant}^c$ and there exists $x \in B_v$ such that $\widehat{\mathcal{A}}_x^{(1)} \supseteq \widehat{B}_v$,
    \end{itemize}
    % where, for any $x \in B_v$, we use $\widehat{\mathcal{A}}_x^{(1)}$ to denote the set of vertices activated by time $\ell$ when initially all particles outside of $\cW_{B_v}^{(1)} := \cup_{y \in B_v} \cW_{y}^{(1)}$ are labeled as removed, all particles in $\cW_{x}^{(1)}$ (at $x$) are activated at time 0, and all other particles are sleeping/inactive at time 0.
    Note that $\widehat{\mathcal{A}}_x^{(1)}$ has the same distribution as $\mathcal{A}_x^{\lambda/2, \ell, B_v}$ defined in \eqref{def: activated_set_restricted}.
    
    Let $\widehat{\mathcal{F}}^{o}$ be the sigma-algebra generated by the random walk trajectory of the planted particle $w_0^o$. 
    For $u \in \T^d_m$, let $\mathcal{F}^{u}$ be the sigma-algebra generated by the initial configuration of particles in $B_u$ (a Poisson random variable per site) and the random walk trajectories taken by particles in $\cW_{B_u}^{(1)}$. 
    Moreover, let $\mathcal{F}^{U} := \sigma(\bigcup_{u \in U} \mathcal{F}^{u})$ and we aim to show that for any $p \in (0, 1)$, if $r$ is large enough, then
    \begin{align} \label{ineq: donination_conditional}
        \bP_L^+ \left(v \text{ is open } \middle\vert \sigma(\widehat{\mathcal{F}}^{o} \cup \mathcal{F}^{\T^d_m \setminus \{ v \}}) \right) \ge p 
    \end{align}
    for every $v \in \T^d_m$. 
    This allows us to conclude that, conditioned on $\widehat{\mathcal{F}}^{o}$, $\{ \1_{v \text{ is open}} \}_{v \in \T^d_m}$ stochastically dominates an independent Bernoulli site percolation with density $p$ on $\T^d_m$. Indeed, Lemma 1.1 in \cite{MR1428500} states that for a sequence of 0-1 valued random variables $\{ Z_i \}_{i = 1}^{\infty}$, if, for all $i$, we have 
    $$\bP(Z_i = 1 | \sigma(Z_1, ..., Z_{i-1})) \ge p,$$ 
    then $\{ Z_i \}_{i = 1}^{\infty}$ stochastically dominates a sequence of independent Bernoulli random variables with parameter $p$. 

    We now prove equation \eqref{ineq: donination_conditional}.
    For every $x \in B_v$, we use Poisson thinning to further divide $\cW_{x}^{(1)}$ into two independent collections $\cW_{x}^{\mathrm{int}}$ and $\cW_{x}^{\mathrm{nbh}}$, each having $\mathrm{Pois}(\lambda/4)$ particles.
    Let $\mathrm{Good}^{\mathrm{int}}_{B_v}$ be the collection of good vertices in $B_v$ with respect to particles in $\cW_{B_v}^{\mathrm{int}} := \cup_{y \in B_v} \cW_{y}^{\mathrm{int}}$ defined as in Definition~\ref{def: good_vertex}. 
    Therefore, $\mathrm{Good}^{\mathrm{int}}_{B_v}$ has the same distribution as $\mathrm{Good}^{\lambda/4, \ell}_{B_v}$ (see Definition~\ref{def: good_vertex}). 
    We first analyze the case when $v \in \mathrm{Fant}$.
    On the event $\{ v \in \mathrm{Fant}\}$, we have
    \begin{align} 
        &\bP_L^+\left(v \text{ is not open } | \widehat{\mathcal{F}}^{o}, \mathcal{F}^{\T^d_m \setminus \{ v \}}\right) \nonumber \\
        \le \ &\bP_L^+\left( \mathrm{Good}^{\mathrm{int}}_{B_v} \cap \mathrm{R}_{w_0^{o}}(\ell_o) = \varnothing \middle\vert \widehat{\mathcal{F}}^{o}, \mathcal{F}^{\T^d_m \setminus \{ v \}} \right) \label{term: Fant_Good_not_empty} \\
        & \qquad + \frac{\bP_L^+ \left( v \text{ is not open}, \mathrm{Good}^{\mathrm{int}}_{B_v} \cap \mathrm{R}_{w_0^{o}}(\ell_o) \neq \varnothing \middle\vert \widehat{\mathcal{F}}^{o}, \mathcal{F}^{\T^d_m \setminus \{ v \}} \right)}{\bP_L^+ \left( \mathrm{Good}^{\mathrm{int}}_{B_v} \cap \mathrm{R}_{w_0^{o}}(\ell_o) \neq \varnothing \middle\vert \widehat{\mathcal{F}}^{o}, \mathcal{F}^{\T^d_m \setminus \{ v \}} \right)} \label{term: Fant_non_trans} \\
        \le \ & C \exp(-c \lambda \phi(s(r))) \, + \, C' \exp(-c' \lambda f(r)), \label{ineq: Fant_bounds}
    \end{align}
    where $c, c', C, C' > 0$ are constants independent of $r$, $f(r)$ is defined as in Lemma~\ref{lemma: range_estimates_main_text}, and $s(r)$ is defined as in Definition~\ref{def: fantastic_vertex}. More specifically, the upper bound for \eqref{term: Fant_Good_not_empty} follows from Proposition~\ref{prop: BernPerc_internal_1} together with the fact that $|\mathrm{R}_{w_0^{o}}(\ell_o) \cap B_v| \ge \phi(s(r))$ on the event $\{ v \in \mathrm{Fant}\}$.
    Also, the upper bound for \eqref{term: Fant_non_trans} follows from Proposition~\ref{prop: BernPerc_internal_2} applied to particles in $\cW_{B_v}^{\mathrm{nbh}}$: conditioned on the event that $\mathrm{Good}^{\mathrm{int}}_{B_v} \cap \mathrm{R}_{w_0^{o}}(\ell_o) \neq \varnothing$, there exists some $x \in B_v$ such that it is good with respect to particles in $\cW_{B_v}^{\mathrm{int}}$. In particular, $x$ activates at least a quarter of the vertices in $\widehat{B}_v$ using particles in $\cW_{B_v}^{\mathrm{int}}$. Let $D_x$ denote the set of those activated vertices. Then, we apply Proposition~\ref{prop: BernPerc_internal_2} to particles in $\cW_{B_v}^{\mathrm{nbh}}$ with $D = D_x$.
    By taking $r$ sufficiently large, we can make the left-hand side of \eqref{ineq: donination_conditional} larger than any given $p \in (0, 1)$.
    The case when $v \in \mathrm{Fant}^c$ can be treated similarly.
    The only difference is that, in equation \eqref{term: Fant_Good_not_empty} and \eqref{term: Fant_non_trans}, one should remove the intersection with $\mathrm{R}_{w_0^{o}}(\ell_o)$.
    This completes the proof of the stochastic domination.

    Consider the auxiliary Bernoulli site percolation on $\T^d_m$ given by $\{ \1_{\{v \text{ is open} \}}\}_{v \in \T^d_m}$. 
    We have proved that, conditioned on $\widehat{\mathcal{F}}^{o}$, almost surely it stochastically dominates a Bernoulli site percolation with retention probability $p$. 
    Let $\mathrm{GC} = \mathrm{GC}(m)$ be the largest connected component in the auxiliary Bernoulli site percolation on $\T^d_m$. 
    Let 
    \begin{align}
        A := \bigcup_{v \in \mathrm{GC}} B_v.
    \end{align}
    We will first show that, by choosing $r$ sufficiently large, with high probability, $A$ has the desired spatial homogeneous property.
    By (i) of Theorem~\ref{thm: percolation_reference}, which concerns the spacial density and homogeneity of $\mathrm{GC}$, we know that with high probability, in every box of side length $n = \lceil C(d, p) (\log m)^{\frac{1}{d-1}} \rceil$, $|\mathrm{GC}|$ is at least $c(p) n^d$, where $c(p) \rightarrow 1$ as $p \rightarrow 1$. Note that, by choosing $r$ large enough, we can make $p$ arbitrarily close to 1. Therefore, for any given $\rho \in (0, 1)$, we can choose $r$ sufficiently large such that, with high probability, $A$ is $\mathrm{Hom}(C_1 (\log L)^\frac{1}{d-1}, \rho)$, where $C_1 = C_1(\rho, d, r)$ is a positive constant. 

    We now complete the proof of Proposition~\ref{prop: activate_constant_fraction_alpha>=d} by showing that, with high probability, $A \subseteq \mathcal{A}(C_2 L)$ for some $C_2>0$ and for all $\alpha > 0$. 

    By Lemma~\ref{lemma: planted_to_fantastic}, we know that for any $C > 0$, with high probability, we have $|\mathrm{Fant}| > C$ as $L \rightarrow \infty$. 
    By (ii) in Theorem~\ref{thm: percolation_reference}, we have that, with high probability, $\mathrm{Fant} \cap \mathrm{GC}$ is non-empty.  
    Indeed, we can write $\bP_L^+(\mathrm{Fant} \cap \mathrm{GC} = \varnothing) \le \bP_L^+(|\mathrm{Fant}| \le C) + \bP_L^+(\mathrm{Fant} \cap \mathrm{GC} = \varnothing \, \vert \, |\mathrm{Fant}| > C)$, where the first term on the right-hand side vanishes as $L$ diverges to infinity by Lemma~\ref{lemma: planted_to_fantastic}, and the second term on the right-hand side can be made arbitrarily small by taking $C$ sufficiently large according to (ii) in Theorem~\ref{thm: percolation_reference}. 
    Since the auxiliary Bernoulli site percolation is supercritical, Theorem~\ref{Thm: Perc_Distance} together with a straightforward union bound imply that there exists a constant $C_0$ such that, with high probability, for all pairs of vertices$u$ and $v$ on the renormalized torus $\T^d_m$ that are connected by a percolation cluster, the shortest path connecting $u$ and $v$ has length at most $C_0 m$. 
    One can verify this by using a union bound and applying Theorem~\ref{Thm: Perc_Distance} to all pairs of vertices that are separated by a graph distance of at least $m/8$. 
    On the event that $\mathrm{Fant} \cap \mathrm{GC} \neq \varnothing$, we have that for any $x \in A$, the maximal activation time from the origin \(o\) to \(x\) is upper bounded by \(\ell_o + C_0 (2r)^d \ell(r) m\). 
    Note that $(2r)^d \ell(r)$ is a deterministic upper bound for the time needed to activate an open box $B_v$ together with its $2d$ neighboring boxes using particles in $\cW_{B_v}^{(1)}$, starting from one activated good vertex in $B_v$. 
    Since $m = \lfloor L / r \rfloor$ and we can choose $r = R(\rho, \lambda, d, \alpha, K)$ sufficiently large but fixed, we have that $\ell_o + C_0 (2r)^d \ell(r) m \le C_2 L$ for some constant $C_2 = C_2(\rho, \lambda, d, \alpha, K) > 0$ when $L$ is sufficiently large. 
    Moreover, we require the particle lifespan to be at least $N_0 = \ell(R)$, where $\ell(r)$ is defined as in \eqref{def: s(r)}.
    This completes the proof of Proposition~\ref{prop: activate_constant_fraction_alpha>=d}.
    
    We now turn to the case where \(\alpha \in (0, d)\). We show that for any \(\epsilon > 0\), by choosing \(r\) sufficiently large, with high probability we have \(A \subseteq \mathcal{A}((\log L)^{\Delta + \epsilon})\). Note that in the proof of Proposition~\ref{prop: activate_constant_fraction_alpha>=d} we only used the particles in $\cW_x^{(1)}$. We now leverage the results from Secion~\ref{Sec: Backbone} and use the second set of $\mathrm{Pois}(\lambda/2)$ particles per site, namely $\cW_x^{(2)}$. 

    We first recall the discussion around equation \eqref{ineq: G_v_condition} to connect the frog model with the long-range percolation while making the choice of $G_v$ more explicit.
    Let $g(r) = C_3 \log r$, where $C_3$ is the same constant as $C_1$ in Corollary~\ref{cor: almost_constant_fraction_good_vertex}.
    Let $X_{v}^{[L]}$ be the indicator of the event $\{ v \text{ is open}, |\mathrm{Good}^{\mathrm{int}}_{B_v}| \ge r^d / g(r) \}$ and define
    \[
        q(r) := \sup_{v \in \T^d_m} \bP_L\left(X_{v}^{[L]} = 0\right).
    \]
    It follows from equation \eqref{ineq: Fant_bounds} and Corollary~\ref{cor: almost_constant_fraction_good_vertex} that $q(r) \to 0$ as $r \to \infty$, uniformly in $L$. 
    For every $v \in \T^d_m$ with $X_{v}^{[L]} = 1$, we choose 
    $$G_v = \mathrm{Good}^{\mathrm{int}}_{B_v}.$$ 
    Note that $\mathrm{Good}^{\mathrm{int}}_{B_v}$ only depends on particles in $\cW_{B_v}^{(1)}$.
    Then, we draw a directed edge from $v$ to $w$ on $\T^d_m$, denoted by $v \Rightarrow w$, whenever (1) $X_{w}^{[L]} = 1$ and (2) there exists a particle in $\cW_{x}^{(2)}$ for some $x$ in $B_v$ that jumps to some $y$ in $G_w$ at its first step. Define
    $$Y_{(v, w)}^{[L]} := \mathbbm{1}_{\{v \Rightarrow w\}}.$$
    % In this case, we write $Y_{(v, w)}^{[L]} = 1$.
    Therefore, we are under the same setup constructed before Lemma~\ref{lemma: long-range-percolation-coupling}. In particular, the lemma implies that 
    \begin{align*} 
        \P_{L}^+\left(Y_{(u, v)}^{[L]} = 1\right) \ge 1 - \exp\left(- \beta_1 |u - v|^{-s}\right),
    \end{align*}
    where $s := \alpha + d$ and $\beta_1$ can be made arbitrarily large by taking $r$ sufficiently large.

    By Proposition~\ref{prop: S_emerge}, for any $\epsilon > 0$, when $r$ is sufficiently large (so that $\beta_1$ is sufficiently large), with high probability, a random set \(S \subseteq \T^d_m\) that is 
    $$\text{Hom}((\log m)^{1/(2d - s) + \epsilon}, (1/\log \log m)^C)$$ 
    emerges, where $C$ is a positive constant. Moreover, for every two vertices \(x, y \in S\), we have
    \begin{align*}
        d(x, y) \le (\log m)^{\Delta + \frac{\epsilon}{2}},
    \end{align*}
    where \(d(x, y)\) denotes the chemical (graph) distance between any two vertices \(x, y\) on \(\T^d_m\) within the long-range percolation clusters. 
    Let us denote the event that such $S$ emerges by $\mathcal{E}_1$.
    Since $\Delta > 1/(2d - s)$ for every $s \in (d, 2d)$, it follows that such \(S\) is 
    $$\mathrm{Hom}((\log m)^{\Delta}, (1/\log \log m)^C/2)$$
    for all large enough $m$.
    
    Let $\mathcal{E}_2 \subseteq \mathcal{E}_1$ be the event that for every box \(B\) with side length \(\lceil (\log m)^{\Delta} \rceil\) on \(\T^d_m\), the intersection \(S \cap \mathrm{GC} \cap B\) is non-empty.
    Part (ii) of Theorem~\ref{thm: percolation_reference}, which gives us a stretched exponential tail in terms of $|U|$ on the probability of $\mathrm{GC} \cap U$ being empty for general vertex set $U$, together with a union bound gives that $\bP_L^+(\mathcal{E}_2 \mid \mathcal{E}_1)$ converges to 1 as $L$ (and hence $m$) diverges. Indeed, conditional on the vertex indicators $\{X_v^{[L]}\}_{v \in \T^d_m}$, the long-range edges (and hence $S$) are generated from $\cW^{(2)}$ independently of the auxiliary site percolation (and hence $\mathrm{GC}$), which is generated from $\cW^{(1)}$. In particular, for every $0 < d' < d$, there exists $\beta_1$ such that
    \begin{align*}
        \bP_L^+( \, \mathcal{E}_2^c \, | \, \mathcal{E}_1 \,) \le m^{-d} \exp{\left(- \beta_1 \frac{2(\log m)^{\Delta(d' - 1)}}{(\log \log m)^{C}}\right)}.
    \end{align*}
    Choosing $d'$ such that \((d'-1) \cdot \Delta > 1\), this probability tends to 0 as $L$ (and hence $m$) diverges.

    Additionally, part (iii) of Theorem~\ref{thm: percolation_reference} implies that for some positive constant $C'$, with high probability the following event occurs: 
    for every box \(B\) with side length \(\lceil (\log m)^{\Delta} \rceil\) on \(\T^d_m\), we have
    \begin{align*}
        \mathrm{diam}(\mathrm{GC} \cap B) \le C'(\log m)^{\Delta},
    \end{align*}
    where $\mathrm{diam}(D)$ is the largest graph distance between any pair of vertices in $D$ within the auxiliary Bernoulli site percolation clusters, where the connecting paths may leave $D$. 
    Let us denote the above event by $\mathcal{E}_3$.

    In summary, on the event $\mathcal{E}_2 \cap \mathcal{E}_3 \cap \{ \mathrm{Fant} \cap \mathrm{GC} \neq \varnothing  \}$, for every \(x \in A\), we can bound the activation time $\mathrm{AT}(o, x)$ as follows:
    \begin{align}
        \mathrm{AT}(o, x) 
        &\le \ell_{o} + 2 \cdot (2r)^d \ell(r) \cdot \max_B \left( \mathrm{diam}(\mathrm{GC} \cap B) \right) + (2r)^d \ell(r) \cdot \max_{x, y \in S} d(x, y) \nonumber \\
        &\le (\log L)^{\Delta} + C_4 (\log L)^{\Delta} + (\log L)^{\Delta + \frac{\epsilon}{2}}, \label{ineq: AT_decompose}
    \end{align}
    where the last inequality holds for sufficiently large $L$ and $C_4$ is a positive constant depending on $r, d,$ and $\alpha$. Moreover, the maximum in the first inequality is taken over all boxes \(B\) with side length \(\lceil (\log m)^{\Delta} \rceil\). Therefore, for $L$ large enough, the right-hand side of \eqref{ineq: AT_decompose} is upper bounded by \((\log L)^{\Delta + \epsilon}\). This completes the proof of Proposition~\ref{prop: activate_constant_fraction}.
\end{proof}

\subsection{Upper Bounds on Activation Time Under Diverging Lifespan} \label{Sec: punchline}

    In this section, we prove Theorem~\ref{MainThm: infinite_l_activation_time}, which gives an upper bound on the activation time for the SI frog model on \(\mathbb{T}^d_L\) (i.e., all frogs have infinite lifespan).
    The theorem states that, with high probability, the time needed for the SI model to activate all vertices of \(\mathbb{T}^d_L\) is at most \((\log L)^{\Delta + \epsilon}\) when \(\alpha \in (0, d)\), and at most \(O(L)\) when \(\alpha \ge d\).
    In fact, we prove the stronger statement Theorem~\ref{thm: infinite_l_activation_time}, which shows that the same conclusion holds for the SIR frog model with a slowly diverging lifespan \(\ell = \ell(L)\).
    The proof will be based on Proposition~\ref{prop: activate_constant_fraction} and \ref{prop: activate_constant_fraction_alpha>=d} established in the previous section.

    We now make the choice of diverging lifespan \(\ell = \ell(L)\) more explicit.
    Define $\ell_1: \N \rightarrow \N$ by
    \begin{align}
        \ell_1(L) := 
            \begin{cases}
                \; r_1(L)^2 & \textup{if } \alpha > 2, \\
                \left\lfloor r_1(L)^2/\log r_1(L) \right\rfloor & \textup{if } \alpha = 2, \\
                \left\lfloor r_1(L)^\alpha \right\rfloor & \textup{if } \alpha < 2,
            \end{cases}
    \end{align}
    where \( r_1(L) = (\log (L + 2))^{\frac{1}{d-1}} \).
    For the frog model on \(\mathbb{T}^d_L\), the lifespan of all particles (including the planted particle) is set to be \(\ell(L)\), where
    \begin{align} \label{def: diverging_lifespan_ell}
        \ell(L) \gg \max\{\ell_1(L), \log L \log \log L \},
    \end{align}
    i.e., \(\ell(L)/\ell_1(L) \to \infty\) and \(\ell(L)/(\log L \log \log L) \to \infty\) as \(L \to \infty\).

    \begin{theorem} \label{thm: infinite_l_activation_time}
        Let $\lambda>0$, $d \ge 2$, and $\alpha > 0$.
        Let $Q \in \mathcal{H}_{\alpha, K}$ and let $Q_L$ be defined as in \eqref{def: Q_L}.
        Assume $\ell(L)$ satisfies \eqref{def: diverging_lifespan_ell}.
        If \( \alpha \in (0, d) \), then, for any \( \epsilon > 0 \), we have
            \begin{align}
                \lim_{L \rightarrow \infty} \P_L^+\left(\max_{x \in \mathbb{T}^d_L} \mathrm{AT}_{\ell(L)}(o, x) \le (\log L)^{\Delta + \epsilon} \right) = 1.
            \end{align}
        If \( \alpha \ge d \), then there exists a positive constant \( C \), depending on $\lambda, d, \alpha,$ and $K$, such that
            \begin{align}
                \lim_{L \rightarrow \infty} \P_L^+\left(\max_{x \in \mathbb{T}^d_L} \mathrm{AT}_{\ell(L)}(o, x) \le C L \right) = 1.
            \end{align}
    \end{theorem}

    \begin{lemma} \label{lemma: Delta_g_1/d-1}
    For every $d \ge 2$ and every $0 < \alpha < d$, we have 
    \begin{align}
        \Delta > 
        \begin{cases}
            \frac{2}{d - 1} \text{ for } \alpha \ge 2, \\
            \frac{\alpha}{d-1} \text{ for } \alpha < 2.
        \end{cases}
    \end{align}
    \end{lemma}

    The proof of Lemma~\ref{lemma: Delta_g_1/d-1} is an elementary calculation, and we leave it to the supplementary material (see Section~4 of \cite{angel2026supplement}).

    \begin{proof}[Proof of Theorem~\ref{thm: infinite_l_activation_time}]
        By Lemma~\ref{lemma: Delta_g_1/d-1}, we have $(\log L)^{\Delta} \gg \ell_1(L)$.
        Since increasing the lifespan can only decrease activation times, it suffices to consider $\ell(L) \le (\log L)^{\Delta}$.
        We again use Poisson thinning to split the particles at each vertex into two independent collections $\cW_{x}^{(1)}$ and $\cW_{x}^{(2)}$, each containing $\mathrm{Pois}(\lambda/2)$ particles.
        The first half, $\cW_{x}^{(1)}$, is used to apply Proposition~\ref{prop: activate_constant_fraction} and \ref{prop: activate_constant_fraction_alpha>=d}. 
        The remaining half, $\cW_{x}^{(2)}$, is used to activate the entire torus $\mathbb{T}^d_L$ starting from a spatially homogeneous set of activated vertices.

        We first describe the use of particles in $\cW_{x}^{(1)}$ together with the planted particle $w_0^o$. 
        We apply Proposition~\ref{prop: activate_constant_fraction} when $\alpha \in (0, d)$ and Proposition~\ref{prop: activate_constant_fraction_alpha>=d} when $\alpha \ge d$ for $\cW_{x}^{(1)}$ and the planted particle $w_0^o$.
        More specifically, we apply the propositions under the following setting:
        \begin{itemize}
            \item At time $0$, only the planted particle $w_0^o$ is active at the origin $o$, while all other particles (those in $\cW_{x}^{(1)}$ for all $x$) are sleeping.
            \item Choose $\rho = 1/2$, and replace $\epsilon$ in Proposition~\ref{prop: activate_constant_fraction} by $\epsilon/2$.
            \item The lifespan of $w_0^o$ is $\ell(L)$, while the lifespan of each particle in $\cW_{x}^{(1)}$ is a fixed but sufficiently large constant, as required in the propositions.
        \end{itemize}
        Let $C_1$ and $C_2$ be the same constants as in Proposition~\ref{prop: activate_constant_fraction} and \ref{prop: activate_constant_fraction_alpha>=d} with respect to the corresponding choice of parameters. 
        Moreover, let $\mathcal{A}^{(1)}$ be the set of vertices activated by time $(\log L)^{\Delta + \frac{\epsilon}{2}}$ when $\alpha \in (0, d)$ and $C_2 L$ when $\alpha \ge d$, using only particles in $\cW_{x}^{(1)}$ and the planted particle $w_0^o$.
        Then, with high probability, 
        $$\mathcal{A}^{(1)} \text{ is } \mathrm{Hom}\left(C_1 (\log L)^{\frac{1}{d-1}}, 1/2\right).$$
        Let $\mathcal{E}$ denote the event that $\mathcal{A}^{(1)}$ has the above spatially homogeneous property.
        Let $\mathcal{F}^{(1)}$ denote the sigma-algebra generated by the planted particle $w_0^o$ and all particles in $\cW_{x}^{(1)}$ for all $x$.

        Let us define the typical distance traveled by an $\ell(L)$-step random walk as follows:
        \begin{align*}
            r(L) := \begin{cases}
                \sqrt{\ell(L)} &\text{if } \alpha > 2, \\
                \sqrt{\ell(L) \log \ell(L)} &\text{if } \alpha =2, \\
                (\ell(L))^{1/\alpha} &\text{if } \alpha < 2.
            \end{cases}
        \end{align*}
        Due to \eqref{def: diverging_lifespan_ell}, it is straightforward to verify that $r(L) \gg (\log L)^{1/(d-1)}.$

        For $A \subseteq \mathbb{T}^d_L$, let $\mathcal{R}_{A}^{(2)}$ be the union of the ranges of $\ell(L)$-step random walks initiated by particles in $\cW_{x}^{(2)}$ for all $x \in A$. Note that particles in $\cW_{x}^{(2)}$ have density $\lambda/2$ and are independent of those in $\cW_{x}^{(1)}$ and the planted particle $w_0^o$.
        On the event $\mathcal{E}$, for any $x \in \mathbb{T}^d_L$, we have that 
        \begin{align*}
            \P_{L}^+\left( x \notin \mathcal{R}_{\mathcal{A}^{(1)}}^{(2)} \, \middle\vert \, \mathcal{F}^{(1)} \right) &= \exp\left(-\frac{\lambda}{2} \sum_{y \in \mathcal{A}^{(1)}} \mathbf{P}_y^L(\tau_x \le \ell(L))\right).
        \end{align*}
        Note that $\mathbf{P}_y^L(\tau_x \le \ell(L)) \ge \mathbf{P}_y(\tau_x \le \ell(L))$, with a slight abuse of notation by identifying vertices on $\mathbb{T}^d_L$ with those in $[0, L)^d \cap \mathbb{Z}^d$.
        On the event $\mathcal{E}$, since $r(L) \gg C_1 (\log L)^{1/(d-1)}$, $\mathcal{A}^{(1)}$ occupies at least 1/3 of the vertices in $B_{r(L)}(x)$ for all sufficiently large $L$. 
        Hence, by Lemma~\ref{apx_lemma: random_walk_expected_hitting_size_lower_bound}, there is some constant $c = c(d, \alpha, K) > 0$ such that
        \begin{align*}
            \sum_{y \in \mathcal{A}^{(1)}  \cap B_{r(L)}(x) } \mathbf{P}_y(\tau_x \le \ell(L)) \ge \begin{cases}
            c \ell(L) & \text{ if } d \ge 3 \text{ or } (d = 2, \alpha < 2), \\
            c \ell(L)/\log \log \ell(L) & \text{ if } d = \alpha = 2, \\
            c \ell(L)/\log \ell(L) & \text{ if } d = 2, \alpha > 2,
            \end{cases}
        \end{align*}

        Given that $\ell(L) \gg \log L \log \log L$, we have
        \begin{align*}
            \P_{L}^+\left( \exists \, x \in \mathbb{T}^d_L \text{ such that } x \notin \mathcal{R}_{\mathcal{A}^{(1)}}^{(2)} \, \middle\vert \, \mathcal{E} \right) &\le L^d \exp\left(-c' \lambda \ell(L) / f(\ell(L))\right) \rightarrow 0,
        \end{align*}
        where $f(\ell(L))$ is either 1, $\log \log \ell(L)$, or $\log \ell(L)$ depending on the dimension and $\alpha$. 

        Therefore, on the event $\mathcal{E} \cap \{ x \in \mathcal{R}_{\mathcal{A}^{(1)}}^{(2)} \text{ for all } x \}$, all vertices in $\mathbb{T}^d_L$ are activated by time $(\log L)^{\Delta + \epsilon/2} + \ell(L)$ if $\alpha \in (0, d)$ and by time $C_2 L + \ell(L)$ if $\alpha \ge d$.
        Since $\ell(L) \le (\log L)^{\Delta}$, we have $(\log L)^{\Delta + \epsilon/2} + \ell(L) \le (\log L)^{\Delta + \epsilon}$ for all sufficiently large $L$. Also, $(\log L)^{\Delta} = o(L)$, so we may absorb $\ell(L)$ into the constant in the case $\alpha \ge d$.
    \end{proof}

\section{Lower Bounds on Activation Time}
\label{sec:activation_LB}
In Section~\ref{sec:activation_UB}, we established upper bounds on the time needed to activate a large fraction of the vertices in $\T^d_L$ when the lifespan is large but fixed.
These bounds are polylogarithmic when $\alpha \in (0, d)$ and linear when $\alpha \ge d$.

In this section, we complement these findings by deriving corresponding lower bounds.
In Section~\ref{sec:activation_LB_small_alpha}, we prove a polylogarithmic lower bound with a matching exponent for the activation time in the regime $\alpha \in (0, d)$.
In Section~\ref{sec:activation_LB_large_alpha}, we outline how to adapt the approach of \cite{berger2004lower}, originally developed for long-range percolation, to obtain a linear lower bound on the activation time when $\alpha > d$. 
Lower bounds on $\mathrm{AT}$ in the following two cases remain open: (1) $\alpha = d$ with finite lifespan (see 
Conjecture~\ref{Conj: alpha_equals_d}); (2) infinite lifespan (i.e., the SI frog model).
Throughout this section, we use $|x|$ to denote the $\ell^\infty$-norm of $x \in \Z^d$.

\subsection{Case I: $0 < \alpha < d$} 
\label{sec:activation_LB_small_alpha}
In this subsection, we establish Theorem~\ref{MainThm: LowerBound_Smaller_d} by proving the following stronger result, Theorem~\ref{Thm: LowerBound_Smaller_d}.
Let $B(L) := [-L, L]^d \cap \Z^d$.
We remind the reader that 
$$\Delta^{-1} := \log_2(\frac{2d}{\alpha + d}).$$

\begin{theorem} \label{Thm: LowerBound_Smaller_d}
    Let $d \geq 1$, $\lambda > 0$, $\ell \in \N^+$, $\alpha \in (0, d)$, and $Q \in \mathcal{H}_{\alpha, K}$. 
    Then, there exists a positive constant $c_1 = c_1(d, \alpha, K, \lambda, \ell)$ such that
\begin{align}
    \lim_{M \rightarrow \infty} \bP \left( \min_{x : |x| \ge M} \frac{\mathrm{AT}_\ell(x)}{(\log |x|)^\Delta} \le c_1  \right) = 0. 
\end{align}
\end{theorem}

\begin{proof}[Proof of Theorem~\ref{MainThm: LowerBound_Smaller_d} using Theorem~\ref{Thm: LowerBound_Smaller_d}]
    Define 
    $$E_M := \left\{ \mathrm{AT}_\ell(x) > c_1 (\log |x|)^\Delta \text{ for all } x \text{ with } |x| \ge M \right\}.$$
    Observe that $E_1 \subseteq E_2 \subseteq \cdots$ and it follows from Theorem~\ref{Thm: LowerBound_Smaller_d} that $\bP( \cup_{M=1}^\infty E_M ) = 1$.
    Therefore, with probability 1, there exists some (random) $M_0$ such that $\mathrm{AT}_\ell(x) > c_1 (\log |x|)^\Delta$ for all $x$ with $|x| \ge M_0$.
    We prove Theorem~\ref{MainThm: LowerBound_Smaller_d} by taking $c = c_1/2$.
    First, the observation above implies that $\mathcal{A}(c_1 (\log L)^\Delta) \subseteq B(L)$ for all $L \ge M_0$.
    Moreover, we also have 
    \begin{align*}
    \mathcal{A}(t) 
        &\subseteq \{ x : |x| < M_0 \} \cup \{ x : |x| \ge M_0, \, \mathrm{AT}_\ell(x) \le t \} \\
        &\subseteq \{ x : |x| < M_0 \} \cup \{ x : |x| \ge M_0, \, |x| < \exp((t/c_1)^{1/\Delta}) \}
    \end{align*}
    for all $t > 0$. Let $t = \frac{c_1}{2} (\log L)^\Delta$. Then, for all sufficiently large $L$, we have $\mathcal{A}(t) \subseteq B(L^{1 - \epsilon})$ for some $\epsilon \in (0, 1)$ depending on $\Delta$. In particular, almost surely, we have $|\mathcal{A}(t)| = o(L^d)$ as $L \to \infty$. This completes the proof.
\end{proof}

It turns out to be more convenient to work with the infection distance $\mathrm{D}_{\ell}(o, x)$ defined in equation \eqref{def: infection_distance}. 
Recall that $\mathrm{D}_{\ell}(o, x)$ is the minimum number of particles used along an infection chain from the origin $o$ to $x$ when each particle has lifespan $\ell$. Moreover, $\mathrm{AT}_{\ell}(x) \ge \mathrm{D}_{\ell}(o, x)$ for all $x \neq o$.
The next proposition upper bounds the probability that the infection distance from the origin to $x$ is at most a given integer $m$. 
Its proof is adapted from Theorem 1.2 in Trapman \cite{MR2663638} and Theorem 3.1 in Biskup \cite{MR2850269}, 
which study the chemical distance in long-range percolation on $\Z^d$. 
Compared with the original arguments, our setting involves two additional difficulties:
(1) we need a suitable way to apply the van den Berg--Kesten--Reimer inequality to the frog model; and
(2) in the frog model, the directed edges are not mutually independent.

\begin{proposition} \label{Prop: LB_main}
    Let $d \geq 1$, $\lambda > 0$, $\ell \in \N^+$, $\alpha \in (0, d)$, and $Q \in \mathcal{H}_{\alpha, K}$.
    There exist positive constants $C_1$, $C_2$, and $M$ (depending on $d, \alpha, K, \lambda, \ell$) such that, for all $m \ge M$ and all $x \in \Z^d$ with $x \neq o$, we have
    \begin{align}
        \bP(\mathrm{D}_{\ell}(o, x) \le m) \le \frac{C_1}{|x|^{d + \alpha}} e^{C_2 m^{1/\Delta}}. \label{ineq: LB_main_thm}
    \end{align}
\end{proposition}

\begin{proof}[Proof of Theorem~\ref{Thm: LowerBound_Smaller_d} given Theorem~\ref{Prop: LB_main}]
    Let $m = c_1 (\log |x|)^{\Delta}$ for some $c_1 > 0$ to be chosen later. Then, by \eqref{ineq: LB_main_thm}, for all $|x|$ larger than some positive constant $M$, we have that 
    \begin{align} \label{ineq: LB_main}
        \bP\left(\mathrm{AT}_{\ell}(x) \le c_1 (\log |x|)^{\Delta}\right) \le \bP\left( \mathrm{D}_{\ell}(o, x) \le c_1 (\log |x|)^{\Delta} \right) \le C_1 {|x|}^{C_2 c_1^{1/\Delta} - (d + \alpha)}.
    \end{align}
    Choose $c_1 > 0$ such that $C_2 c_1^{1/\Delta} < \alpha$. Then, by using a union bound over all vertices $x$ on $\Z^d$ with $|x| \ge M$, we obtain 
    \begin{align*}
        1 - \bP \left( E_M \right) \le \sum_{k = M}^{\infty} C(d) \, C_1 \, k^{C_2 c_1^{1/\Delta} - (d + \alpha)} k^{d - 1} \le C' M^{C_2 c_1^{1/\Delta} - \alpha},
    \end{align*}
    where $C(d)$ is a dimensional constant such that the term $C(d) k^{d-1}$ is an upper bound for the number of vertices in $\Z^d$ with $|x| = k$ and $C'$ is a positive constant depending on $d, \alpha, K, \lambda,$ and $\ell$. The right-hand side converges to 0 as $M$ goes to infinity. This completes the proof.
\end{proof}

We now prove Proposition~\ref{Prop: LB_main}, which relies on Reimer's inequality \cite{MR1751301}. 
To apply Reimer's inequality, we first set up an finite product space by restricting to infection chains using only frogs within a large but finite box. 
Fix a finite box $B(L_0) := [-L_0, L_0]^d \cap \Z^d$.
Recall that $\mathcal{R}_{x, \ell, \lambda}$ is the set of vertices visited by at least one particle started from $x$; equivalently, it is the union of the ranges of $\ell$-step random walks initiated by $\Pois(\lambda)$ particles at $x$.
Let $\Omega_{x}$ denote the collection of all possible subsets of $B(L_0)$ that contain $x$, where each $\omega_x \in \Omega_x$ corresponds to a possible realization of $\mathcal{R}_{x, \ell, \lambda} \cap B(L_0)$. Let $\mathcal{F}_x$ be the power set of $\Omega_x$.  
Let $\bQ_x$ be the probability measure given by the law of $\mathcal{R}_{x, \ell, \lambda} \cap B(L_0)$.
For each $x \in B(L_0)$, we have a probability space $(\Omega_x, \mathcal{F}_x, \bQ_x)$.
Then, let $(\Omega_{L_0}, \mathcal{F}_{L_0}, \bQ_{L_0})$ be the product space over all $x \in B(L_0)$. 
For any $\omega \in \Omega_{L_0}$ and $S \subseteq B(L_0)$, define the cylinder set
\[C(\omega, S) := \{ \omega' \in \Omega_{L_0} : \omega_x' = \omega_x \text{ for all } x \in S \}.
\]
Let $A, B \subseteq \Omega_{L_0}$ be two events. Then, the disjoint occurrence of $A$ and $B$, denoted by $A \circ B$, is defined as 
\begin{align*}
    A \circ B := \left\{ \omega \in \Omega_{L_0} : \exists S \subseteq B(L_0) \text{ such that } C(\omega, S) \subseteq A, C(\omega, S^c) \subseteq B \right\},
    \end{align*}
where $S^c := B(L_0) \setminus S$.
Reimer's inequality states that
\begin{align}
    \bQ_{L_0}(A \circ B) \le \bQ_{L_0}(A) \, \bQ_{L_0}(B). \label{ineq: Riemer_inequality}
\end{align}
On $\Omega_{L_0}$, we define the \textbf{$L_0$-approximate infection distance} $\widetilde{\mathrm{D}}(x, y) = \widetilde{\mathrm{D}}(x, y; L_0)$ between any two vertices $x, y \in B(L_0)$ as the number of edges traversed by the shortest directed path from $x$ to $y$ in the random directed graph on $B(L_0)$ whose directed edges are given by
\[
u \xrightarrow[]{} v \quad \Longleftrightarrow \quad v \in \omega_u,
\qquad u,v\in B(L_0), u \neq v.
\]
Note that $u \rightarrow v$ is also an event defined on the original probability space. The relevant probability space will be clear from the context whenever we write $u \rightarrow v$.
Note that we use the same notation in both cases since $\bQ_{L_0}(u \rightarrow v) = \bP(u \rightarrow v)$ for any $u, v \in B(L_0)$. It is straightforward to verify that $\widetilde{\mathrm{D}}(x, y; L_0)$ converges almost surely to $\mathrm{D}_{\ell}(x, y)$ as $L_0 \to \infty$.

\begin{lemma} \label{lemma: LB_lemma_1}
    Let $B_k := \{x \in \Z^d : \mathrm{D}_{\ell}(o, x) \le k \}$. Then, there exists a positive $C = C(d, \alpha, K, \ell)$ such that for any $m \ge 1$ and any $x \in \Z^d$ with $x \neq o$, we have
    \begin{align*}
        \P \left( \mathrm{D}_\ell(o, x) \le m \right) \le \lambda C \left(\frac{m}{|x|}\right)^{d + \alpha} \, \sum_{j=0}^{m-1} \mathbb{E} [|B_j|] \ \mathbb{E} [|B_{m-j-1}|].
    \end{align*}
\end{lemma}

\begin{proof}
    Fix some $L_0 \ge |x|$.
    If $\widetilde{\mathrm{D}}(o, x) \le m$, then there is a self-avoiding directed path from $o$ to $x$ with at most $m$ edges in the random directed graph on $B(L_0)$ defined above. In particular, there is an edge on the path, say $y \rightarrow z$, such that $|y - z| \ge |x|/m$.
    If the first such edge occurs at the $j+1$-th step, then there exist two directed paths using disjoint sets of vertices, one from $o$ to $y$ and the other from $z$ to $x$, such that they use at most $j$ and $m - j - 1$ edges, respectively. Therefore, we can upper bound $\bQ_{L_0} ( \widetilde{\mathrm{D}}(o, x) \le m )$ by
    \begin{align}
       &\sum_{j = 0}^{m-1} \sum_{\substack{y, z \in B(L_0) \\ |y-z| \ge |x| / m}} \bQ_{L_0} \left( \left\{ \widetilde{\mathrm{D}}(o,y) \le j \right\} \circ \left\{ y \rightarrow z \right\} \circ \left\{ \widetilde{\mathrm{D}}(z,x) \le m - j - 1 \right\} \right) \nonumber \\
       \le &\sum_{j = 0}^{m-1} \sum_{\substack{y, z \in B(L_0) \\ |y-z| \ge |x| / m}} \bQ_{L_0} \left( \left\{ \widetilde{\mathrm{D}}(o,y) \le j \right\} \right) \, \bP \left( y \rightarrow z \right) \, \bQ_{L_0} \left( \left\{ \widetilde{\mathrm{D}}(z,x) \le m - j - 1 \right\} \right), \label{LB_lemma_1_line00}
    \end{align}
    where we used Reimer's inequality. 
    Note that $\bP \left( y \rightarrow z \right) = 1 - \exp(-\lambda \bfP_y(\tau_z \le \ell))$. 
    Moreover, by Lemma~\ref{lemma: hitting_probability_upper_bound_on_Z}, there exists a constant $C = C(d, \alpha, K, \ell)$ such that $\bfP_y(\tau_z \le \ell) \le C |y - z|^{-(d + \alpha)}$. Therefore, for all $y, z$ satisfying $|y - z| \ge |x|/m$, we have $\bP \left( y \rightarrow z \right) \le \lambda C (|x|/m)^{-(d + \alpha)}$. By substituting this bound into \eqref{LB_lemma_1_line00}, we obtain
    \begin{align*}
        \bQ_{L_0} \left( \widetilde{\mathrm{D}}(o, x) \le m \right) &\le \lambda C \left( \frac{|x|}{m} \right)^{-(d + \alpha)} \sum_{j=0}^{m-1} \E_{\bQ_{L_0}} \left[ \, \left|\widetilde{B}_j(o) \right| \, \right] \, \E_{\bQ_{L_0}} \left[ \, \left| \widetilde{B}_{m - j - 1}(x) \right| \, \right],
    \end{align*}
    where $\widetilde{B}_k(x) := \{ y \in B(L_0) : \widetilde{\mathrm{D}}(x, y) \le k \}$. Note that, in the last line, we used the fact that $\bfP_u(\tau_v \le \ell) = \bfP_v(\tau_u \le \ell)$ for any $u, v \in \Z^d$, which allows us to exchange the roles of $x$ and $z$ in equation \ref{LB_lemma_1_line00}. Moreover, it is clear that $\widetilde{\mathrm{D}}(u, v; L_0)$ stochastically dominates $\mathrm{D}_{\ell}(u, v)$ for any $u, v \in B(L_0)$. Therefore, we always have $\E_{\bQ_{L_0}} [ |\widetilde{B}_k(u)| ] \le \E [|B_k| ]$ for any $u \in B(L_0)$ and any $k \in \N^+$. By taking $L_0 \to \infty$, we complete the proof.
\end{proof}

The next lemma is a direct computation that relates estimates on the probability of $\left\{ \mathrm{D}_{\ell}(o, x) \le k \right\}$ to the expected volume of the ball $B_k$. 
\begin{lemma} \label{lemma: LB_lemma_2}
    There exists a positive constant $C_0 = C_0(\alpha, d)$ such that for any $k \ge 1$ and $M > 0$, if
    $\bP \left( \mathrm{D}_{\ell}(o, x) \le k \right) \le (M/|x|)^{\alpha + d}$ holds for all $x \in \Z^d$ with $x \neq o$, then $\mathbb{E} [|B_k|] \le C_0 M^d$.
\end{lemma}

\begin{proof}
    We have 
    \begin{align*}
        \bE[|B_k|] = \sum_{x \in \Z^d} \bP(\mathrm{D}_{\ell}(o, x) \le k) \le \sum_{\substack{ x \in \Z^d \\ |x| \le M }} 1 + \sum_{\substack{x \in \Z^d \\ |x| > M}}  \left( \frac{M}{|x|} \right)^{\alpha + d}.
    \end{align*}
    Note that the first term is at most $C_1 M^d$, while the second term is at most $C_2 M^{-\alpha} M^{\alpha + d} = C_2 M^d$ for some positive constants $C_1, C_2$ depending on $\alpha$ and $d$. This completes the proof.
\end{proof}

The following lemma is taken from \cite{MR2850269}, Lemma 3.4. We refer the reader to \cite{MR2850269} for a proof. 

\begin{lemma} \label{lemma: LB_lemma_3}
    Let $s := \alpha + d$. For each $p > \frac{s + 1}{2d - s}$ and each $C_4 > 0$ there is $C_3 = C_3(C_4, p) \in (0, \infty)$ such that for each $C' \ge C_4$ the quantity 
    \begin{align}
    M(m) := \frac{1}{C_3} \ (m+1)^{-p} e^{C' m^{1/\Delta}}, \label{def: M_function}
    \end{align}
    satisfies 
    \begin{align*}
    \sum_{j=0}^{m} M(j)^d M(m-j)^d \le (m+1)^{-s} M(m+1)^s, 
    \end{align*}
    where $m \in \N^+$.
\end{lemma}

\begin{proof}[Proof of Proposition~\ref{Prop: LB_main}]
    Let $s = \alpha + d$ and $q = \frac{1}{2d-s}$. 
    Fix some $p > \frac{s+1}{2d-s}$ and $C_4 > 0$.
    Let $C_3 = C_3(C_4, p)$ be the same as in Lemma~\ref{lemma: LB_lemma_3} and let $C' > C_4$ be some constant to be determined later. Then, we use the same $M(m)$ as defined in equation \eqref{def: M_function}.
    Lastly, let $C > 0$ be chosen such that $\bP(u \rightarrow v) \le \lambda C |u - v|^{-(d + \alpha)}$ for any $u, v \in \Z^d$ (see Lemma~\ref{lemma: hitting_probability_upper_bound_on_Z}), and let $C_0 = C_0(d, \alpha)$ be the same constant as in Lemma~\ref{lemma: LB_lemma_2}.

    We will prove by induction on $n$ that for any $n \ge 1$ and any $x \in \Z^d$ with $x \neq o$, we have
    \begin{align}
        \bP(\mathrm{D}_{\ell}(o, x) \le n) \le \left( \frac{(\lambda C)^{-q} C_0^{-2q} M(n)}{|x|} \right)^s. \label{eq: induction_hypo}
    \end{align}

    When $n = 1$, if we choose $C'$ such that $e^{C'} \ge 2^p C_3 C_0^{2q} (\lambda C)^{q + \frac{1}{s}}$, then \eqref{eq: induction_hypo} follows from $\bP(o \rightarrow x) \le \lambda C |x|^{-(d + \alpha)}$.
    Assume \eqref{eq: induction_hypo} holds for all $n \le m$. Then, by using Lemma~\ref{lemma: LB_lemma_1}, Lemma~\ref{lemma: LB_lemma_2}, and Lemma~\ref{lemma: LB_lemma_3} in order, we have
    \begin{align*}
        & \bP(\mathrm{D}_{\ell}(o, x) \le m+1) \, \le \, C \lambda \left(\frac{m+1}{|x|}\right)^{s} \sum_{j=0}^{m} \bE[|B_j|] \bE[|B_{m-j}|] \\
        \le \, & \left(\lambda C\right)^{1 - 2dq} C_0^{2 - 4dq} \, \left(\frac{m+1}{|x|}\right)^{s} \, \sum_{j=0}^{m} M(j)^d M(m-j)^d \\
        \le \, & \left( \lambda C \cdot C_0^{2} \right)^{1 - 2dq} \, \left( \frac{M(m+1)}{|x|} \right)^s.
    \end{align*}
    Since $q = 1/(2d - s)$, we have $1 - 2dq = -qs$. Thus, the induction step is complete and so is the proof of Proposition~\ref{Prop: LB_main}.
\end{proof}

\subsection{Case II: $\alpha > d$} 
\label{sec:activation_LB_large_alpha}
In this subsection, we focus on the case $\alpha > d$. Our main objective is to prove Theorem~\ref{Thm: lower_bound_alpha_ge_d}. We first state the main result of this subsection and show that it immediately implies Theorem~\ref{MainThm: lower_bound_alpha_ge_d}. 

\begin{theorem} \label{Thm: lower_bound_alpha_ge_d}
    Let $d \geq 1$, $\lambda > 0$, $\ell \in \N^+$, $\alpha > d$, $K > 0$, and $Q \in \mathcal{H}_{\alpha, K}$.
    Then, almost surely, we have
    \begin{align}
        \liminf_{|x| \rightarrow \infty} \frac{\mathrm{AT}_\ell(x)}{|x|}  > 0.
    \end{align}
\end{theorem}

\begin{proof}[Proof of Theorem~\ref{MainThm: lower_bound_alpha_ge_d} using Theorem~\ref{Thm: lower_bound_alpha_ge_d}]
    Let $W := \inf_{x : x \neq o} \frac{\mathrm{AT}_\ell(x)}{|x|}$, which is positive almost surely by Theorem~\ref{Thm: lower_bound_alpha_ge_d}. Take $c_0$ such that $\P(W \le c_0) < \epsilon$. 
    On the event $\{W > c_0\}$, we have $\mathrm{AT}_\ell(x) > c_0 |x|$ for all $x \neq o$. 
    In particular, this implies that for any $L > 0$, we have $\mathcal{A}(c_0 L) \subseteq B(L)$.
    Moreover, we have 
    \begin{align*}
        \frac{\left|\mathcal{A}\left(\frac{c_0\epsilon^{1/d}}{3}L\right)\right|}{L^d} \le \frac{\left|B\left(\frac{\epsilon^{1/d}}{3}L\right)\right|}{L^d} \le \frac{\left(3 \cdot \frac{\epsilon^{1/d}L}{3}\right)^d}{L^d} = \epsilon.
    \end{align*}
    By setting $c = \frac{c_0\epsilon^{1/d}}{3}$, we complete the proof.
\end{proof}

The proof of Theorem~\ref{Thm: lower_bound_alpha_ge_d} is an adaptation of the approach used in \cite{berger2004lower} for long-range percolation. We shall provide details of the modifications to adapt their proof to our setting, while omitting some details that are essentially the same as those in \cite{berger2004lower}.

% In the preamble:
% \usepackage{tikz}
% \usetikzlibrary{arrows.meta}

\begin{figure}[ht]
\centering
\begin{tikzpicture}[
    x=1.05cm,
    y=1.05cm,
    every node/.style={font=\small},
    childgrid/.style={gray!55, line width=0.25pt},
    innergrid/.style={blue!35, line width=0.25pt},
    kblock/.style={black, very thick},
    innerblock/.style={blue!70!black, very thick},
    shiftedinner/.style={orange!85!black, very thick, dashed},
    shiftarrow/.style={-{Latex[length=2.2mm]}, orange!85!black, thick},
    dim/.style={<->, >=Latex, thin}
]

% For the picture only, take a = A_{k-1}=1 and C_k=5, so A_k=5.
\pgfmathsetmacro{\a}{1}
\pgfmathsetmacro{\C}{5}
\pgfmathsetmacro{\Ak}{\C*\a}
\pgfmathsetmacro{\half}{0.5*\a}

% Q^{int} = [a/2, A_k-a/2]^2
\pgfmathsetmacro{\IntLo}{\half}
\pgfmathsetmacro{\IntHi}{\Ak-\half}

% Choose v = (a/2,-a/2)
\pgfmathsetmacro{\vx}{\half}
\pgfmathsetmacro{\vy}{-\half}

% Q^{int}+v
\pgfmathsetmacro{\ShiftLoX}{\IntLo+\vx}
\pgfmathsetmacro{\ShiftLoY}{\IntLo+\vy}
\pgfmathsetmacro{\ShiftHiX}{\IntHi+\vx}
\pgfmathsetmacro{\ShiftHiY}{\IntHi+\vy}

% ------------------------------------------------------------
% Main regions
% ------------------------------------------------------------

% Background for Q
\fill[gray!5] (0,0) rectangle (\Ak,\Ak);

% Fill Q^{int}
\fill[blue!12] (\IntLo,\IntLo) rectangle (\IntHi,\IntHi);

% Fill shifted Q^{int}+v
\fill[orange!25, opacity=0.35]
    (\ShiftLoX,\ShiftLoY) rectangle (\ShiftHiX,\ShiftHiY);

% Grid of the original (k-1)-children of Q
\draw[step=\a, childgrid] (0,0) grid (\Ak,\Ak);

% Optional shifted grid inside Q^{int}, showing side length A_{k-1}
\draw[step=\a, innergrid] (\IntLo,\IntLo) grid (\IntHi,\IntHi);

% Boundaries
\draw[kblock] (0,0) rectangle (\Ak,\Ak);
\draw[innerblock] (\IntLo,\IntLo) rectangle (\IntHi,\IntHi);
\draw[shiftedinner] (\ShiftLoX,\ShiftLoY) rectangle (\ShiftHiX,\ShiftHiY);

% ------------------------------------------------------------
% Labels
% ------------------------------------------------------------

\node[anchor=south west, fill=white, inner sep=1pt]
    at (\Ak+0.05,\Ak+0.03) {$Q$};

\node[blue!70!black, anchor=north east, fill=white, inner sep=1.2pt]
    at (\IntHi-0.06,\IntHi-0.06) {$Q^{\mathrm{int}}$};

\node[orange!85!black, anchor=west, fill=white, inner sep=1.2pt]
    at (\ShiftHiX+0.08,\ShiftHiY-0.25) {$Q^{\mathrm{int}}+v$};

\node[anchor=north east] at (-0.05,-0.05) {$x$};

% % Mark lower-left corners
% \fill[blue!70!black] (\IntLo,\IntLo) circle (1.5pt);
% \fill[orange!85!black] (\ShiftLoX,\ShiftLoY) circle (1.5pt);

% % Shift arrow
% \draw[shiftarrow] (\IntLo,\IntLo) -- (\ShiftLoX,\ShiftLoY)
%     node[midway, right=2pt, fill=white, inner sep=1pt] {$v$};

% ------------------------------------------------------------
% Dimension annotations
% ------------------------------------------------------------

% Whole side length A_k
\draw[dim] (0,-0.42) -- (\Ak,-0.42)
    node[midway, below] {$A_k=C_kA_{k-1}$};

% % One child side length A_{k-1}
% \draw[dim] (0,-0.42) -- (\a,-0.42)
%     node[midway, above] {$A_{k-1}$};

% Horizontal half-margin
\draw[dim] (0,\Ak+0.32) -- (\a,\Ak+0.32)
    node[midway, above] {$A_{k-1}$};

% Vertical half-margin
\draw[dim] (-0.38,0) -- (-0.38,\half)
    node[midway, left] {$A_{k-1}/2$};

% ------------------------------------------------------------
% Legend
% ------------------------------------------------------------

\begin{scope}[shift={(0,\Ak+0.95)}]
    \draw[kblock] (-0.5,0) rectangle (-0.18,0.32);
    \node[anchor=west] at (-0.1,0.16) {$k$-block $Q$};

    \draw[innerblock] (1.95,0) rectangle (2.27,0.32);
    \node[anchor=west] at (2.38,0.16) {$Q^{\mathrm{int}}$};

    \draw[shiftedinner] (3.70,0) rectangle (4.02,0.32);
    \node[anchor=west] at (4.13,0.16) {shifted $Q^{\mathrm{int}}+v$};
\end{scope}

\end{tikzpicture}
\caption{A two-dimensional illustration of a \(k\)-block \(Q\), its inner block \(Q^{\mathrm{int}}\), and one shifted inner block \(Q^{\mathrm{int}}+v\), with \(v=(A_{k-1}/2,-A_{k-1}/2)\in\{0,\pm A_{k-1}/2\}^d\). The gray grid represents the original \((k{-}1)\)-children of \(Q\).}
\end{figure}
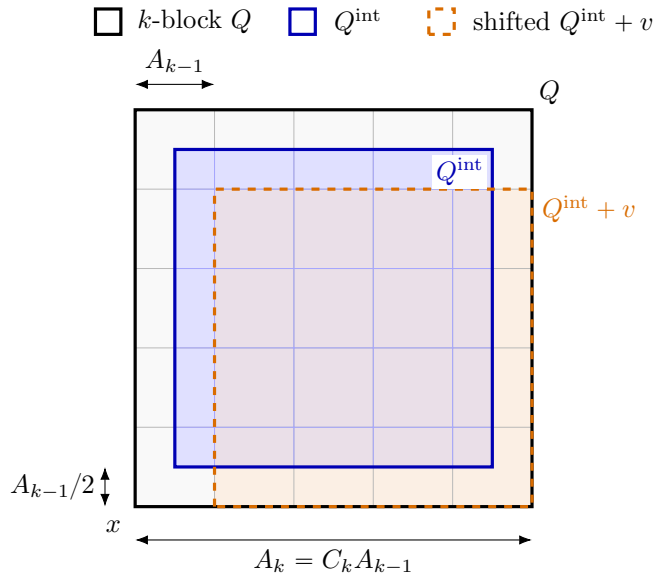

We first construct a renormalization structure similar to that in \cite{berger2004lower}. Let $C_0 = M$ be a positive even integer and $C_n = n^2$ for $n \ge 1$, and define
\begin{align*}
    A_n = \prod_{i = 0}^{n} C_i = M \, (n!)^2.
\end{align*}
Later we will choose $M$ sufficiently large depending on $d, \alpha, K, \lambda,$ and $\ell$.
An $n$-block is a set of the form $x + [0, A_n)^d \cap \Z^d$ for some $x \in \Z^d$. The children of an $n$-block $Q$ are the $C_n^d$ many $(n{-}1)$-blocks contained in $Q$, which can be written as
\begin{align*}
    \left\{ x + A_{n-1} v + [0, A_{n-1})^d \cap \Z^d \ | \ v \in [0, C_n)^d \cap \Z^d \right\}.
\end{align*}
Moreover, let 
\begin{align*}
    Q^{\mathrm{int}} := x + \frac{A_{n-1}}{2} \mathbf{1} + [0, A_{n} - \frac{A_{n-1}}{2})^d \cap \Z^d, \quad \mathbf{1} = (1, ..., 1).
\end{align*}
Observe that there are $(C_{n} {-} 1)^d$ $(n{-}1)$-blocks in $Q^{int}$. 
As before, we draw a directed edge from $x$ to $y$ if and only if $x \rightarrow y$, where recall that $x \rightarrow y$ means there exists at least one particle initially at $x$ that visits $y$ within its lifespan $\ell$. As discussed in the previous section (also see Lemma~\ref{lemma: hitting_probability_upper_bound_on_Z}), there exists a constant $C = C(d, \alpha, K, \ell)$ such that for any $x, y \in \Z^d$, we have
\begin{align} \label{ineq: long-range-edge-upper-bound}
    \bP(x \rightarrow y) \le C \lambda |x - y|^{-(d + \alpha)}.
\end{align}
Moreover, $x \rightarrow y$ and $u \rightarrow v$ are independent events for any distinct $x, u \in \Z^d$.

\begin{definition}[Good blocks]
    \label{def: Berger_good_box_definition}
    Let $Q$ be a $k$-block

    \begin{enumerate}[label=\textup{(\roman*)}, leftmargin=2.2em, itemsep=2pt]
    \item A $0$-block $Q$ is \textbf{good} if there is no directed edge $(u,w)$ with
    $u,w\in Q$ and $|u-w|>A_0/100$.

    \item For $k\ge1$, a $k$-block $Q$ is \textbf{good} if the following hold:
    \begin{enumerate}[label=\textup{(\alph*)}, leftmargin=2.2em, itemsep=2pt]
    \item there is no directed edge $(u,w)$ with
    $u,w\in Q$ and $|u-w|>A_{k-1}/100$;
    \item among the $(k{-}1)$-blocks (children) of $Q$,
    all but at most one are good;
    \item for every
    $v\in\{0,\pm A_{k-1}/2\}^d$, the shifted inner block $Q^{\mathrm{int}}+v$
    also satisfies \textup{(b)} (i.e., among its $(k{-}1)$-subblocks, all but at most
    one are good).
    \end{enumerate}
    \end{enumerate}
\end{definition}

The following lemma is parallel to Lemma 1 in Berger's paper \cite{berger2004lower}. Since we modified the definition of a good block, we give a proof here for completeness.

\begin{lemma} \label{lemma: sum_faliure_finite}
    Let $d \ge 1$ and $\alpha > d$.
    Let $P_k$ be the probability of a $k$-block being not good as defined in Definition~\ref{def: Berger_good_box_definition}.
    If $M$ is sufficiently large, then
    \begin{align}
        \sum_{k = 0}^{+\infty} P_k < \infty.
    \end{align}
\end{lemma}

\begin{proof}
    We show by induction that there exists $C'>0$ such that for all $k\ge 0$,
    \[
    P_k \le C' e^{-k}.
    \]
    From Definition~\ref{def: Berger_good_box_definition}, for $k\ge1$, we have the following recursive inequality:
    \begin{align*} 
        P_{k} 
        &\le \lambda C A_k^{2d} \left(\frac{100}{A_{k-1}}\right)^{d + \alpha} + (3^d + 1)  \binom{C_k^d}{2} P_{k-1}^2 \\ 
        &\le \lambda C 100^{d + \alpha} M^{d - \alpha} (k!)^{2d - 2\alpha} k^{2d + 2\alpha} + (3^d + 1) k^{4d} P_{k-1}^{2}. 
    \end{align*}
    Since $\alpha>d$, we have $2d-2\alpha<0$, and hence
    \[
    (k!)^{2d-2\alpha} k^{2d+2\alpha}
    = \exp\!\big(-(2\alpha-2d)\log(k!) + O(\log k)\big)
    \le e^{-3k/2}
    \]
    for all $k$ large enough. Absorbing finitely many small $k$ into the constant,
    there exists $C_1>0$ (depending on $d, \alpha, C, \lambda$) such that
    \[
    \lambda C 100^{d+\alpha} M^{d-\alpha} (k!)^{2d-2\alpha} k^{2d+2\alpha}
    \le C_1 M^{d-\alpha} e^{-3k/2}.
    \]
    Assume inductively that $P_{k-1}\le C' e^{-(k-1)}$. Then, for any $k \ge 1$, we have
    \[
    (3^d+1) k^{4d} P_{k-1}^2
    \le (3^d+1) C'^2 k^{4d} e^{-2k+2}
    \le C_2 C'^2 e^{-3k/2},
    \]
    for some constant $C_2>0$.
    Combining the bounds,
    \[
    P_k \le (C_1 M^{d - \alpha} + C_2 C'^2) \, e^{-3k/2}.
    \]
    Choosing $M$ large enough and $C'>0$ small so that
    \[
    (C_1 M^{d - \alpha} + C_2 C'^2) e^{-1/2} \le C',
    \]
    we obtain $P_k \le C' e^{-k}$, closing the induction.
\end{proof}

Given Lemma~\ref{lemma: sum_faliure_finite}, Theorem~\ref{Thm: lower_bound_alpha_ge_d} follows by the same argument as in \cite{berger2004lower}. More precisely, Lemma 2, Proposition 3, and Lemma 3 of \cite{berger2004lower}, together with their proofs, apply verbatim in our setting. We therefore do not repeat the details, and instead briefly summarize the argument. Sections 3--5 of \cite{berger2004lower} establish a deterministic multiscale estimate: if a block is good, then any path crossing a macroscopic portion of that block must have length at least a fixed positive multiple of its Euclidean displacement. This lower bound is then extended from paths confined to a single block to arbitrary paths by using neighboring half-translates and larger enclosing good blocks. Since the probabilities of bad blocks are summable, the Borel--Cantelli lemma implies that only finitely many relevant bad blocks occur.
We refer the reader to Sections 3--5 of \cite{berger2004lower} for the complete argument.

\section{Cover Lifespan on the Torus} \label{section: Susceptibility}

Throughout this section, we let $d \ge 2$, $\lambda > 0$, and fix a heavy-tailed kernel $Q \in \mathcal{H}_{\alpha, K}$ with $\alpha > 0$ and $K > 0$.
In this section, we prove Theorem~\ref{MainThm: Susceptibility} which provides concentration inequalities on the cover lifespan $\mathcal{L} = \mathcal{L}(L)$, where 
\begin{align*}
    \mathcal{L}(L) := \inf \left\{ \ell \in \mathbb{N} : \mathcal{A}^{\lambda, \ell}_o(\infty) = \T^d_L \right\}.
\end{align*}
In words, the random variable $\mathcal{L}$ is the minimum lifespan needed for the frog model with the planted particle started from $o$ with particle density $\lambda$ to eventually activate every vertex of the torus $\T^d_L$ (see Section~\ref{subsection: FrogModelFormal} for the detailed construction).
Moreover, we recall the definition of $\ell^*$:
\begin{align} \label{def: t*}
    \ell^* := \inf \left\{ \ell \in \N^+ : \frac{\lambda \ell}{G_{\ell}} \ge d \log L \right\},
\end{align}
where $G_n := \sum_{i = 0}^{n} Q^i(o, o)$ is the on-diagonal Green's function up to time $n$ for the random walk with transition kernel $Q$ on $\Z^d$.

In the next two subsections, we will show that $\mathcal{L}$ is asymptotically concentrated around $\ell^*$ as $L \rightarrow \infty$.
Before turning to the proofs, we record the following asymptotics of $\ell^*$ for each choice of $d$ and $\alpha$:
\begin{align} \label{eq: Susc_asymp_t*}
	\ell^* \asymp_{Q, d, \lambda} \begin{cases}
		\log L \log \log L &\text{~if~} d = 2 \text{ and } \alpha > 2, \\
		\log L \log \log \log L &\text{~if~} d = 2 \text{ and } \alpha = 2, \\
		\log L &\text{~if~} d \ge 3 \text{ or } (d = 2 \text{ and } \alpha < 2).
	\end{cases}
\end{align}
The proof of equation~\eqref{eq: Susc_asymp_t*} is a straightforward application of Lemma~\ref{lemma: green_function_on_Z}, so we omit it.

In addition to the quantitative description of $\mathcal{L}$ above, our argument reveals a close connection between the cover lifespan of the frog model and the cover time of random walks on the torus. Let $\tau_{\mathrm{cov}}$ be the minimum time required for $\Pois(\lambda L^d)$ stationary discrete-time random walks, all activated at time $0$, to cover the torus $\T^d_L$. We show that the cover lifespan $\mathcal{L}$ is asymptotically concentrated around $\tau_{\mathrm{cov}}$ as $L \rightarrow \infty$ under the natural coupling. 

The proof proceeds as follows. In Section~\ref{subsection: SusceptibilityUpper}, we establish an asymptotic upper bound on $\mathcal{L}$ by adapting methods introduced in Section~\ref{Sec: BernPerc}. In Section~\ref{subsection: SusceptibilityLower}, we prove that, for any $\epsilon > 0$, we have $\tau_{\mathrm{cov}} \ge (1 - \epsilon) \ell^*$ with high probability. Combining these results with the fact that $\tau_{\mathrm{cov}}$ is stochastically dominated\footnote{Using Poisson thinning, it is not difficult to see that $\tau_{\mathrm{cov}}$ has the same distribution as the cover lifespan of the frog model when all particles are active at time $0$.} by $\mathcal{L}$, we conclude that both $\mathcal{L}$ and $\tau_{\mathrm{cov}}$ are asymptotically concentrated around $\ell^*$ as $L \rightarrow \infty$.

\subsection{Upper Bound} \label{subsection: SusceptibilityUpper}
In this subsection, we prove the following upper bound on the cover lifespan.

\begin{proposition} \label{prop: Susc_upper}
Let $d \ge 2$ and $\alpha > 0$.
Let $\lambda > 0$, $K > 0$, and $Q \in \mathcal{H}_{\alpha, K}$. 
Let $Q_L$ be as in equation \eqref{def: Q_L}. Then, for any $c > 0$,
\begin{align*}
    \bP_L^+\left(\mathcal{L} \ge (1+c)\ell^*\right) \rightarrow 0,
\end{align*}
as $L \to \infty$.
\end{proposition}

One quantity that we will use repeatedly in this subsection is the \textbf{typical distance} traveled by a random walk with transition kernel $Q$ in $\ell$ steps, defined by
\begin{align} \label{def: typical_distance_susc}
    r(\ell) := \begin{cases}
        \lceil \sqrt{\ell} \rceil &\text{ if } \alpha > 2, \\
        \lceil \sqrt{\ell \log \ell} \rceil &\text{ if } \alpha = 2, \\
        \lceil \ell^{1/\alpha} \rceil &\text{ if } \alpha < 2.
    \end{cases}
\end{align}
Combining equation \eqref{eq: Susc_asymp_t*} with equation \eqref{def: typical_distance_susc}, we obtain the following asymptotics for the typical distance traveled by the random walk on $\Z^d$, for $d \ge 2$, up to time $\ell^*$:
% \begin{align} \label{eq: typical_distance_at_t*_easy}
%     r(\ell^*) \asymp_{Q, d, \lambda} \begin{cases}
%         (\log L \log \log L)^{1/2} &\text{~if~} \hspace{3mm} d = 3 \text{ and } \alpha = 2, \\
%         (\log L)^{1/\alpha} &\text{~if~} \hspace{2mm}
%         \begin{aligned}
%             &(d = 2 \text{ and } \alpha \in (0, 1)) \text{~or~} \\ 
%             &(d \ge 3 \text{ and } \alpha \in (0, 2)).
%         \end{aligned}
%     \end{cases}
% \end{align}
\begin{align} \label{eq: typical_distance_at_t*}
    r^* := r(\ell^*) \asymp_{Q, d, \lambda} \begin{cases}
        (\log L)^{1/2} &\text{~if~} \; d \ge 3 \text{ and } \alpha > 2, \\[1mm]
        (\log L \log \log L)^{1/2} &\text{~if~} \;
        \begin{aligned}
            &d = 2 \text{ and } \alpha > 2 \text{, or}\\ 
            &d \ge 3 \text{ and } \alpha = 2,
        \end{aligned} \\[1mm]
        (\log L \log \log L \log \log \log L)^{1/2} &\text{~if~} \; d = 2 \text{ and } \alpha = 2, \\[1mm]
        (\log L)^{1/\alpha} &\text{~if~} \; d \ge 2 \text{ and } \alpha \in (0, 2).
    \end{cases}
\end{align}

In the rest of this subsection, we prove Proposition~\ref{prop: Susc_upper} in a manner similar in spirit to the proofs in Section~\ref{Sec: BernPerc}. The main difference is that, instead of using boxes of fixed side length in the renormalized lattice, we use boxes of side length $r^*$, which diverges as $L$ tends to infinity. The advantage of this choice is that, rather than only requiring certain ``good'' events in each box to occur with probability larger than $p_c$, we obtain that their probabilities are close to 1. 
Moreover, we can quantitatively control the $o(1)$ term, which allows us to take a union bound over all boxes in the renormalized lattice. We now turn to the details.

Partition $\T^d_L$ into boxes of side length $r^*$, allowing the boundary boxes to have side lengths between $r^*$ and $2r^*$. Let $m := \lfloor L/r^* \rfloor$, and denote the resulting partition by $\mathcal{P}_{L, r^*} := \{ B_v : v \in \T^d_m \}$, constructed as in equation \eqref{def: box_const_level} with $r$ replaced by $r^*$.

Let $\epsilon$ be a small positive constant to be chosen later.
At each site $x \in \T^d_L$, we use Poisson thinning to divide the particles into three independent groups: two groups with Pois$(\epsilon \lambda)$ particles each, and one group with Pois$((1 - 2\epsilon) \lambda)$ particles. Denote these three groups at site $x$ by $\cW^{(1)}_x$, $\cW^{(2)}_x$, and $\cW^{(3)}_x$, respectively.

We use the first group of particles in $B_v$ to define \textbf{good} vertices, analogously to Definition~\ref{def: good_vertex}.
Let $\delta$ be a small positive constant to be chosen later, and fix $v \in \T^d_m$ and $x \in B_v$.
Consider the frog model that uses only the first group of particles in $B_v$, namely $\cup_{x \in B_v} \cW^{(1)}_x$, with all lifespans equal to $\ell^*$, and is started from the vertex $x$. Let $\mathcal{A}^*_x$ denote the set of vertices in $B_v$ that are eventually activated under the above setup. We call $x$ \textbf{good} if $|\mathcal{A}^*_x| \ge (1-\delta)|B_v|$. Let $\mathrm{Good}^{(1)}_v$ denote the set of good vertices in $B_v$. By adapting the proof of Proposition~\ref{prop: BernPerc_internal_1}, we obtain the following result:

\begin{corollary} \label{cor: Susc_internal}
    Let $d \geq 2$, $\delta \in (0, 1)$, $\epsilon \in (0, 1/2)$, $\lambda > 0$, $K > 0$, and $\alpha > 0$. Let $Q \in \mathcal{H}_{\alpha, K}$ (see \eqref{def: H_alpha}).
    Let $Q_L$ be defined as in equation \eqref{def: Q_L}.
    Let $\ell^*$, $r^*$, and $m$ be defined as above.
    Then there exist positive constants $L_0$, $C$, and $c$ such that, for all $L \ge L_0$, $v \in \T^d_m$, and $B \subseteq B_v$, we have
    \begin{align}
        \bP_L \left( \mathrm{Good}^{(1)}_v \cap B = \varnothing \right) \leq C \exp\left(-c \epsilon \lambda |B|\right). \label{ineq: Susc_int-dynamic_uppper_bound}
    \end{align}
\end{corollary}

The next corollary concerns the second set of particles in $B_v$, namely $\cup_{x \in B_v} \cW^{(2)}_x$.
This collection has particle density $\epsilon \lambda$ at each site.
Fix $\delta \in (0, 1)$ and $v \in \T^d_m$, and let $A \subseteq B_v$.
Throughout, we assume that $|A| / |B_v| \ge 1 - \delta$.
One should think of $A$ as the set of vertices in $B_v$ activated by the frog model that uses only the first set of particles and starts from a good vertex in $B_v$.
For each $\ell \in \N$, define
$$ \mathcal{R}^{(2)}_v(A, \ell) := \bigcup_{x \in A} \bigcup_{\omega \in \cW^{(2)}_x} \left\{ \omega(0), \ldots, \omega(\ell) \right\}. $$
When $\ell = \ell^*$, we abbreviate $\mathcal{R}^{(2)}_v(A, \ell^*)$ to $\mathcal{R}^{(2)}_v(A)$.
In other words, $\mathcal{R}^{(2)}_v(A)$ is the set of vertices visited within time $\ell^*$ by particles in $\cup_{x \in A} \cW^{(2)}_x$.
We will show that, with high probability, $\mathcal{R}^{(2)}_v(A)$ covers at least a $(1{-}\delta)$-fraction of the vertices in every box neighboring $B_v$.

\begin{corollary} \label{cor: Susc_internal_2}
    Let $d \geq 2$, $\delta \in (0, 1)$, $\epsilon > 0$, $\lambda > 0$, $K > 0$, and $\alpha > 0$. Let $Q \in \mathcal{H}_{\alpha, K}$ (see \eqref{def: H_alpha}).
    Let $Q_L$ be defined as in equation \eqref{def: Q_L}.
    Let $\ell^*$, $r^*$, and $m$ be as defined above.
    Then, for every $v \in \T^d_m$ and uniformly for every $A \subseteq B_v$ with $|A|/|B_v| \ge 1 - \delta$, we have
    \begin{align}
        L^d \, \bP_L \left( \frac{|\mathcal{R}^{(2)}_v(A) \cap B_{v + e}|}{|B_{v + e}|} \le 1 - \delta \text{ for some } e \in \{ \pm e_1, \ldots, \pm e_d \} \right) \rightarrow 0, \label{ineq: Susc_int-dynamic_uppper_bound_2}
    \end{align}
    as $L \to \infty$, where $e_1, \ldots, e_d$ denote the standard basis vectors of $\Z^d$.
\end{corollary}

\begin{proof}
    Fix $e \in \{ \pm e_1, \ldots, \pm e_d \}$.
    Let $u \in (0, 1)$, and let $M$ be a positive integer.
    We say that a set $A \subseteq B_{v + e}$ is \textbf{$u$-nice} if 
    $$\max_{\substack{x, y \in A \\ x \neq y}} G_{\ell^*}(x, y) \le u G_{\ell^*}.$$
    We will use the following lemma.
    \begin{lemma} \label{lemma: good_set_intersect_high_density_set}
        Let $d \ge 2$.
        For every $\delta \in (0, 1)$, every $u \in (0, 1)$, and every $M \in \N_+$, there exists $\hat{L}_0 > 0$ such that for all $L \ge \hat{L}_0$ and all $A' \subseteq B_{v + e}$ with $|A'|/|B_{v + e}| \ge \delta$, $A'$ must contain a $u$-nice set of cardinality $M$.
    \end{lemma}
    We first complete the proof of Corollary~\ref{cor: Susc_internal_2}, assuming the lemma above. 
    Let $A_1 \subseteq B_{v+e}$ be a $u$-nice set of cardinality $M$.
    Define
    $$T_{A_1} := \inf \{ k \in \N : A_1 \cap \mathcal{R}^{(2)}_v(A, k) \neq \varnothing \},$$
    which is the first time that particles in $\cup_{x \in A} \cW^{(2)}_x$ hit $A_1$.
    By Poisson thinning, we have
    \begin{align*}
        &\bP_L \left( A_1 \cap \mathcal{R}^{(2)}_v(A) = \varnothing \right) \le \exp\Bigg( - \epsilon \lambda \sum_{x \in A} \bfP_x(T_{A_1} \le \ell^*) \Bigg).
    \end{align*}
    Moreover, by Lemma~\ref{lemma: cover_to_green}, we have
    \begin{align*}
        \sum_{x \in A} \bfP_x(T_{A_1} \le \ell^*) \ge \frac{G_{\ell^*}(A_1, A)}{\max_{x \in A_1} G_{\ell^*}(x, A_1) } \ge \frac{c' M \ell^*}{(u M + 1) G_{\ell^*}}, 
    \end{align*}
    where $G_{n}(B, C) := \sum_{x\in B, y\in C} G_n(x,y)$. Note that the second inequality follows from the definition of a $u$-nice set and the fact that $G_{\ell^*}(x, A) \ge c' \ell^*$ for some constant $c' = c'(\delta)$ and all $x \in B_{v + e}$ (see equation~\eqref{ineq: heat_kernel_typical} in Theorem~\ref{thm: heat_kernel_typical}). By choosing $u \le c' \epsilon/4$ and $M \ge 4/(c' \epsilon)$, we obtain
    \begin{align*}
        \epsilon \lambda \sum_{x \in A} \bfP_x(T_{A_1} \le \ell^*) \ge 2d \log L.
    \end{align*}
    Fix such a choice of $u$ and $M$ for the remainder of the proof. It then follows from Lemma~\ref{lemma: good_set_intersect_high_density_set} that, for sufficiently large $L$,
    \begin{align*}
        &\bP_L \left( \frac{|\mathcal{R}^{(2)}_v(A) \cap B_{v + e}|}{|B_{v + e}|} \le 1 - \delta \right) \\
        \le & \; \bP_L \left( \exists A' \subseteq B_{v + e} \text{ such that } A' \cap \mathcal{R}^{(2)}_v(A) = \varnothing \text{ and } |A'|/|B_{v + e}| \ge \delta \right) \\
        \le & \; \bP_L \left( \exists A_1 \subseteq B_{v + e} \text{ such that } A_1 \cap \mathcal{R}^{(2)}_v(A) = \varnothing \text{ and } A_1 \text{ is $(u, M)$-nice} \right) \\
        \le & \; \binom{(r^*)^d}{M} \cdot L^{-2d} \le (r^*)^{dM} \cdot L^{-2d}.
    \end{align*}
    Using equation \eqref{eq: typical_distance_at_t*}, we complete the proof by taking a union bound over all $e$ in $\{ \pm e_1, \ldots, \pm e_d \}$.
\end{proof}

\begin{proof}[Proof of Lemma~\ref{lemma: good_set_intersect_high_density_set}]
    For every $a \in \T^d_L$, we define the set 
    $$ V_a := \{ b : b \neq a, G_{\ell^*}(a, b) > u G_{\ell^*} \}. $$
    On the one hand, we have $\sum_{b \in V_a} G_{\ell^*}(a, b) > |V_a| \cdot u G_{\ell^*}$ by definition. 
    On the other hand, $\sum_{b \in V_a} G_{\ell^*}(a, b) \le \ell^*$.
    Therefore, $|V_a| \le \ell^*/(u G_{\ell^*})$ for every $a$. By equation \eqref{def: typical_distance_susc} and Lemma~\ref{lemma: green_function_on_Z}, we have $\ell^*/(u G_{\ell^*}) \ll (r^*)^d$ whenever $d \ge 2$. Thus, for any $u, M$ and any $q \in (0, 1)$, there exists $\hat{L}_0$ such that for any $L \ge \hat{L}_0$ and any $a$, we have $|V_a| \le q(r^*)^d - 1$. 
    Now consider the graph $\widehat{G} = (V, \widehat{E})$ where $V$ is the vertex set of $T^d_L$ and 
    $$ \widehat{E} := \{ \{a, b\} : a \neq b, G_{\ell^*}(a,b) > u G_{\ell^*} \}. $$
    Set $q = \delta / M$.
    We make the following two observations: (1) $|A'| \ge \delta \cdot |B_{v+e}| \ge \delta \cdot (r^*)^d$ and (2) there exists $\hat{L}_0$ such that for all $L \ge \hat{L}_0$, every ball of radius 1 on $\widehat{G}$ has size at most $q \cdot (r^*)^d$. Therefore, $A'$ contains an independent set of size at least $\delta/q = M$. Taking an independent set of size $M$ in $A'$, it is straightforward to verify that their centers form a $u$-nice set.
\end{proof}

We now handle the planted particle. 
Let $\ell_o = \ell^*$ and choose positive integers $r' = r'(L)$ such that $1 \ll r' \ll r^*$ (for instance, we can choose $r' := \lceil (\log L)^{1/3} \rceil$). 
We aim to use the planted particle with lifespan $\ell_o$ to activate a good vertex in some box $B_v$. 
To exploit the fact that, with high probability, there are many good vertices in each box $B_v$, 
we further partition each $B_v$ into boxes of side length $r'$:
we try to tile it with boxes of side-length $r'$;
if $r^*$ is not divisible by $r'$, we can still partition $B_v$ into $(\lfloor r^*/r' \rfloor - 1)^d$ boxes of side length $r'$, 
and the remaining boxes next to the boundary of $B_v$ can have side lengths between $r'$ and $2r'$. 
We use $\mathcal{G}'_v$ to denote the collection of sub-boxes in $B_v$ after this partition.
Let $\mathcal{G}' := \bigcup_{v \in \T^d_m} \mathcal{G}'_v$. 

We define the corresponding typical time required for the random walk to travel distance $r'$ as
\begin{align} \label{def: Susc_typical_time}
    \ell' = \begin{cases}
        (r')^2, &\textup{ if } \alpha > 2, \\
        (r')^2/\log r', &\textup{ if } \alpha = 2, \\
        (r')^\alpha, &\textup{ if } \alpha < 2.
    \end{cases} 
\end{align}
Moreover, we define 
\begin{align}
    \phi(s) := \begin{cases}
        c_1 s, &\textup{ if $d \ge 3$ or $d = 2$ and $\alpha > 2$,} \\
        c_1 s / \log \log s, &\textup{ if $d = 2$ and $\alpha = 2$,} \\
        c_1 s / \log s, &\textup{ if $d = 2$ and $\alpha < 2$,}
    \end{cases}
\end{align}
where $c_1$ is chosen as the same constant as in Lemma~\ref{lemma: range_estimates}. 

A box $B' \in \mathcal{G'}$ is called \textbf{fantastic} if the planted particle $w_0^o$ visits at least $\phi(\ell')$ vertices in $B'$ within time $\ell_o$. Let $\mathsf{Fant}$ denote the set of all fantastic boxes in $\mathcal{G}'$. Define
\begin{align*}
    \tau_n := \inf \left\{ k \in \N : 
    \begin{array}{c}
        k \ge \tau_{n-1} + 2 \ell' \text{ and } w_0^o \text{ visits }\\
        \text{ a previously unvisited box in } \mathcal{G}' \text{ at time } k 
    \end{array}
    \right\},
\end{align*}
with $\tau_0 := 0$. For every $n \in \N_+$, let $B'_n$ be the box in $\mathcal{G}'$ visited by the planted particle at time $\tau_n$.
By the strong Markov property and Lemma~\ref{lemma: range_estimates}, for any $n \in \N_+$, there exists a constant $c_0 > 0$ such that $B'_n$ is fantastic with probability at least $c_0$. 
(Note that, although Lemma~\ref{lemma: range_estimates} is stated for the random walk on $\Z^d$, the same result for the random walk on $\T^d_L$ follows directly by examining the standard coupling between the two random walks.)
Moreover, since $\ell' \ll \ell_o$, with high probability there are diverging numbers of $n$ for which $\tau_n \le \ell_o$.
Thus, we have the following result. We omit the detailed proof.

\begin{lemma} \label{lemma: Susc_planted_to_fantastic}
    Let $d \ge 2$, $\lambda > 0$, $K > 0$, and $\alpha > 0$. 
    For any $C > 0$, we have that 
    \begin{align}
    	\lim_{L \rightarrow \infty} \bP_L(|\mathsf{Fant}| \ge C) = 1.
    \end{align}
\end{lemma}

Combining Corollary~\ref{cor: Susc_internal} and Lemma~\ref{lemma: Susc_planted_to_fantastic}, with high probability, we know that the planted particle can activate a good vertex of some box $B_v$ and, thus, can activate $(1{-}\delta)$-fraction of the box $B_v$. 

For the next lemma, we recall and introduce a few notions. 
Recall that $\{\omega^x_i\}_{i \in \N_+, x \in \T^d_L}$ is a collection of independent discrete-time random walks, 
where the superscript $x$ denotes the starting position.
For each $\rho > 0$, let $\{\mathcal{N}_{\rho, x}\}_{x \in \T^d_L}$ be a collection of independent Pois$(\rho)$ random variables. 
For each $y \in \T^d_L$, define
\begin{align*}
    T_y(\rho, A) := \inf \left\{ k \in \N : y \in \bigcup_{x \in A} \bigcup_{i = 1}^{\mathcal{N}_{\rho, x}} \bigcup_{j = 0}^{k} \omega^x_i(j) \right\}, \qquad \tau_{\mathrm{cov}}^{\rho, A} := \max_{y \in \T^d_L} T_y(\rho, A).
\end{align*}
In other words, $\tau_{\mathrm{cov}}^{\rho, A}$ is the time required to cover the torus when one starts with an independent Pois$(\rho)$ number of random walkers at each site of $A$ and no walkers at sites outside $A$.

\begin{lemma} \label{lemma: cover_from_high_density}
    Let $\ell^*$ and $m$ be as defined above.
    For any $c > 0$, there exist $\delta > 0$ and $\epsilon > 0$ depending on $d, \lambda, K$, and $\alpha$ such that, if $A \subseteq \T^d_L$ satisfies
    \begin{align*}
        \frac{|A \cap B_v|}{|B_v|} \ge 1-\delta \qquad \text{for every } v \in \T^d_m,
    \end{align*}
    then 
    \begin{align}
        \bP_L \left( \tau_{\mathrm{cov}}^{(1-\epsilon)\lambda, A} \ge (1 + c) \ell^* \right) \rightarrow 0 \quad \text{as } L \to \infty.
    \end{align}
\end{lemma}

We would like to point out that Lemma~\ref{lemma: cover_from_high_density} was the original motivation for Proposition~\ref{prop: Susc_upper}.
At a high level, Lemma~\ref{lemma: cover_from_high_density} reduces the proof of Proposition~\ref{prop: Susc_upper} to showing that using a Pois$(\epsilon \lambda)$ number of particles per site already activates a set $A^*$ that is spatially homogeneous and has high density in the above sense.
Once this is established, the proof of Proposition~\ref{prop: Susc_upper} is completed by an application of Lemma~\ref{lemma: cover_from_high_density}. 
Before diving into the proof of Lemma~\ref{lemma: cover_from_high_density}, we first complete the proof of Proposition~\ref{prop: Susc_upper} using Lemma~\ref{lemma: cover_from_high_density}.

\begin{proof}[Proof of Proposition~\ref{prop: Susc_upper}]
    Let $\delta, \epsilon \in (0, 1)$ be two small constants to be determined later.
    Recall that we have a planted particle $w_0^o$ at the origin with lifespan $\ell_o = \ell^*$, and for each $x \in \T^d_L$, we use Poisson thinning to split the particles initially generated at $x$ into three independent sets of particles, where the first two sets have Pois$(\epsilon \lambda)$ particles per site and the third has Pois$((1{-}2\epsilon) \lambda)$ particles per site. Let $U_o$ be the event that the planted particle $w_0^o$ visits at least one good vertex in some box $B_{v}$.
    On $U_o$, let $v_1$ be the index of the box containing the first good vertex visited by the planted particle $w_0^o$.
    By Corollary~\ref{cor: Susc_internal} and Lemma~\ref{lemma: Susc_planted_to_fantastic}, we have $\bP_L^+(U_o) \rightarrow 1$ as $L \rightarrow \infty$. 
    
    Assume that $U_o$ occurs. 
    % By \cref{cor: Susc_internal_2}, for every $v \in \T^d_m$ and every deterministic set $A \subseteq B_v$ with $|A|/|B_v| \ge 1 - \delta$, the event
    %     \[
    %         \frac{|\mathcal{R}^{(2)}_v(A) \cap B_{v + e}|}{|B_{v + e}|} > 1 - \delta
    %         \qquad \text{for every } e \in \{\pm e_1, \ldots, \pm e_d\}
    %     \]
    % has complement probability $o(L^{-d})$, uniformly over all such $v$ and $A$.
    Fix a total order on the set of edges of $\T^d_m$. 
    Using the first two sets of particles, we now define an exploration process on $\T^d_m$ starting from the box $B_{v_1}$. 
    This construction produces an event $U_b$ such that, on $U_b$, at least a $(1-\delta)$-fraction of the vertices in every box is activated by the frog model that uses only the planted particle and the first two sets of particles.
    Let $x_1$ be the first good vertex visited by the planted particle, so $x_1 \in B_{v_1}$.
    At the first stage, set $V_1 := \{ v_1 \}$ and $A_1 := \mathcal{A}^*_{x_1}$.
    Recall that $\mathcal{A}^*_{x_1}$ is the set of vertices in $B_{v_1}$ activated by the frog model started from $x_1$ using only the first collection of particles in $B_{v_1}$, namely the particles arising from the first Pois$(\epsilon \lambda)$ thinning. Since $x_1$ is good, we have $|A_1|/|B_{v_1}| \ge 1-\delta$.
    Finally, let $U^{(1)}$ be the event that
    \[
        \frac{|\mathcal{R}^{(2)}_{v_1}(A_1) \cap B_{v_1 + e}|}{|B_{v_1 + e}|} > 1 - \delta
        \qquad \text{for every } e \in \{\pm e_1, \ldots, \pm e_d\}.
    \]
    
    Now suppose that the first $i$ stages have been completed, and that $V_i = \{v_1, \ldots, v_i\}$ together with $A_1, \ldots, A_i$ and $U^{(1)}, \ldots, U^{(i)}$ has already been defined.
    On the event $U^{(1)} \cap \cdots \cap U^{(i)}$, we define the $(i{+}1)$-st stage as follows.
    Choose $j_i \in \{1, \ldots, i\}$ and $v_{i+1} \in \T^d_m \setminus V_i$ so that $v_{i+1} \sim v_{j_i}$ and the edge $\{v_{j_i}, v_{i+1}\}$ is the first, in the fixed total order, among all edges with one endpoint in $V_i$ and the other in $\T^d_m \setminus V_i$.
    Since $\T^d_m$ is connected, such a choice is always possible. Moreover, because we are on the event $U^{(j_i)}$, the set $\mathcal{R}^{(2)}_{v_{j_i}}(A_{j_i}) \cap B_{v_{i+1}}$ occupies more than a $(1-\delta)$-fraction of the vertices in $B_{v_{i+1}}$.
    We therefore define 
    \[
        A_{i+1} := \mathcal{R}^{(2)}_{v_{j_i}}(A_{j_i}) \cap B_{v_{i+1}}, \qquad
        V_{i+1} := V_i \cup \{ v_{i+1}\}.
    \]
    We then let $U^{(i+1)}$ be the event that
    \[
        \frac{|\mathcal{R}^{(2)}_{v_{i+1}}(A_{i+1}) \cap B_{v_{i+1} + e}|}{|B_{v_{i+1} + e}|} > 1 - \delta
        \qquad \text{for every } e \in \{\pm e_1, \ldots, \pm e_d\}.
    \]
    The exploration stops after stage $m^d$. We then set
    \[
        U_b := \bigcap_{i=1}^{m^d} U^{(i)}.
    \]
    On the event $U_o \cap U_b$, define
    \[
        A^* := \bigcup_{i=1}^{m^d} A_i.
    \]
    By a union bound and Corollary~\ref{cor: Susc_internal_2},
    \begin{align*}
        \bP_L^+\left(U_b^c \, \middle\vert \, U_o \right)
        &\le \sum_{i=1}^{m^d} \bP_L^+\left((U^{(i)})^c \, \middle\vert \, U_o, U^{(1)}, \ldots, U^{(i-1)} \right),
    \end{align*}
    as $L \to \infty$. Hence $\bP_L^+(U_o \cap U_b) \to 1$ as $L \to \infty$. 

    Finally, on the event $U := U_o \cap U_b$, we use the final set of particles, which has Pois$((1{-}2\epsilon)\lambda)$ particles per site, to activate all remaining vertices, with all frogs in $A^*$ initially active. Recall that $\mathcal{W}_x^{(3)}$ denotes the third set of particles initially generated at $x$. 
    For any $y \in \T^d_L$, we define 
    \begin{align}
    T^{(3)}_{y} := \inf \left\{ k \in \N : y \in \bigcup_{x \in A^*} \ \bigcup_{\omega \in \cW_x^{(3)}} \ \bigcup_{j = 0}^{k} \omega(j) \right\} \; \text{ and } \; \tau_{\mathrm{cov}}^{(3)} := \max_{y \in \T^d_L} T_y^{(3)}
    \end{align}

    Since, for every $v \in \T^d_L$, we have that $A^* \cap B_v$ occupies at least a $(1{-}\delta)$-fraction of the vertices in $B_v$, it follows from Lemma~\ref{lemma: cover_from_high_density} (applied with $\epsilon$ replaced by $2\epsilon$) that, for any $c > 0$, there exist $\delta, \epsilon > 0$ such that 
    \[
        \bP_L^+ \left( \tau_{\mathrm{cov}}^{(3)} \ge (1+c) \ell^* \, \middle\vert \, U \right) \rightarrow 0,
    \]
    as $L$ goes to infinity. Therefore, we complete the proof by writing 
    \[
        \bP_L^+ \left( \mathcal{L} > (1 + c) \ell^* \right) 
        \le \ \bP_L^+\left(U^c\right) + \bP_L^+\left( \tau_{\mathrm{cov}}^{(3)} \ge (1+c) \ell^* \, \middle\vert \, U \right).
    \]
\end{proof}

We now finish this subsection by proving Lemma~\ref{lemma: cover_from_high_density}.

\begin{proof}[Proof of Lemma~\ref{lemma: cover_from_high_density}]
    Let $K$ be a large positive constant to be determined later. 
    Let $v_y$ be the unique $v \in \T^d_m$ such that $y \in B_v$. 
    Define $B_y^{(K)} := \cup_{v \in v_y + [-K, K]^d} B_v$. 
    Thus, for every $c > 0$ and every $A \subseteq \T^d_L$ satisfying $|A \cap B_v|/|B_v| \ge 1-\delta$ for all $v \in \T^d_m$, we have 
    \begin{align*}
        \bP_L \bigl( \, T_{y}\left((1{-}\epsilon)\lambda, A \, \bigr) > (1 + c) \, \ell^* \right) \le \exp\left(- (1{-}\epsilon) \lambda \sum_{x \in B_y^{(K)} \cap A} \bfP_x\left(\tau_y \le (1 + c)\ell^*\right) \right),
    \end{align*}
    where $\tau_y$ is the hitting time of $y$ by a random walk on $\Z^d$ and $\bfP_x$ is the law of the random walk starting from $x$.
    By Lemma \ref{lemma: cover_to_green}, we have 
    \begin{align*}
    &\sum_{x \in B_y^{(K)} \cap A^*} \bfP_x\left(\tau_y \le (1 + c)\ell^*\right) \ge \frac{G_{(1 + c)\ell^*}\left(y, B_y^{(K)} \cap A^*\right)}{G_{(1 + c)\ell^*}} \\
    = \ & \frac{G_{(1 + c)\ell^*}\left( y, \Z^d \right) - G_{(1 + c)\ell^*}\left(y, (B_y^{(K)})^c\right) - G_{(1 + c)\ell^*}\left(y, (A^*)^c \cap B_y^{(K)}\right)}{G_{(1 + c)\ell^*}}.
    \end{align*}
    We analyze the three terms in the numerator separately.
    For the first term in the numerator, we have $G_{(1 + c)\ell^*}(y, \Z^d) = (1 + c)\ell^*$.
    For the second term, it follows from equation \eqref{ineq: heat_kernel_intypical} that for any $\epsilon_1 > 0$, we can choose $K$ large enough such that 
    $$ G_{(1 + c)\ell^*}\left(y, (B_y^{(K)})^c\right) \le \epsilon_1 (1 + c)\ell^*.$$
    For the third term, it follows from equation \eqref{ineq: heat_kernel_typical} and the fact that $(A^*)^c$ occupies at most a $\delta$-fraction of each box that, for any $\epsilon_2 > 0$ and $K > 0$, there exists $C = C(\epsilon_2, K)$ such that  
    \begin{align*}
        G_{(1 + c)\ell^*}\left(y, (A^*)^c \cap B_y^{(K)}\right) \le \ &\epsilon_2 (1 + c) \ell^* + \sum_{k = \epsilon_2 (1 + c) C \ell^*}^{(1 + c)\ell^*} \sum_{x \in (A^*)^c \cap B_y^{(K)}} Q^k(x, y) \\
        \le \ &\epsilon_2 (1 + c) \ell^* + (1 + c) \ell^* \cdot \delta |B_y^{(K)}| \cdot \frac{C(\epsilon_2, K)}{|B_y^{(K)}|} \\
        = \ &\epsilon_2 (1 + c) \ell^* + \delta C(\epsilon_2, K) (1 + c) \ell^*.
    \end{align*}
    Moreover, for the denominator, by Lemma~\ref{lemma: green_function_compare_n_and_Cn}, for any $\epsilon_3 > 0$ and any sufficiently large $L$ (depending on $\epsilon_3$), we have
    \begin{align*}
        G_{(1 + c)\ell^*} \le G_\ell^* / (1 - \epsilon_3).
    \end{align*}
    Therefore,
    \begin{align*}
        &(1{-}\epsilon) \lambda \sum_{x \in B_y^{(K)} \cap A^*} \bfP_x\left(\tau_y \le (1 + c)\ell^*\right) \\
        \ge \, &(1{-}\epsilon) \cdot \left(1 - \epsilon_1 - \epsilon_2 - \delta C(\epsilon_2, K)\right) \cdot (1 - \epsilon_3) \cdot (1+c) \cdot \frac{\lambda \ell^*}{G_{\ell^*}}. 
    \end{align*}
    By the definition of $\ell^*$, we have $\lambda \ell^* / G_{\ell^*} \ge d \log L$. 
    Let $\hat{c} := 1 - (1+c/2)/(1+c)$.
    Thus, for any $c > 0$, by first choosing $\epsilon, \epsilon_1, \epsilon_2, \epsilon_3$ to be smaller than $\hat{c}/5$ and then choosing $\delta$ to be smaller than $\hat{c}/5C(\epsilon_2, K)$, the quantity above is at least $(1 + c/2) d \log L$ for all sufficiently large $L$. Therefore, we have
    \begin{align*}
        \bP_L \left( \tau_{\mathrm{cov}}^{(1-\epsilon)\lambda, A} \ge (1 + c) \ell^* \right)
        % \bP_L^+ \left( \mathcal{L} > (1 + c) \ell^* \right) 
        % \le \ & \bP_L^+(U^c) + L^d \bP_L^+\left(T^{(3)}_y > (1 + c) \ell^* \, \middle| \, U \right) \\
        \le \ L^d \cdot \exp(- (1 + c/2) d \log L) \rightarrow 0,
    \end{align*}
    as $L$ goes to infinity. 
\end{proof}

\subsection{Lower Bound via Cover Time} \label{subsection: SusceptibilityLower}

In this subsection, we prove the lower bound in Theorem~\ref{MainThm: Susceptibility} by studying the cover time of $\T^d_L$ by Pois$(\lambda L^d)$ number of walkers starting from the uniform distribution. We denote this quantity by $\tau_{\mathrm{cov}}$. 
Observe that $\tau_{\mathrm{cov}}$ is stochastically dominated by the cover lifespan $\mathcal{L}$.
Indeed, starting with Pois$(\lambda L^d)$ stationary walkers on $\T^d_L$ at time 0 is equivalent to the frog model with particle density $\lambda$ and all particles being active at time 0 (without the planted particle) due to Poisson thinning.
The following result completes the proof of the lower bound in Theorem~\ref{MainThm: Susceptibility} by showing that $\tau_{\mathrm{cov}}/\ell^*$ converges to 1 in probability.

\begin{theorem} \label{thm: covertime}
    Let $\lambda > 0$, $d \ge 2$, $\alpha > 0$, and $K > 0$. Let $Q \in \mathcal{H}_{\alpha, K}$ and let $Q_L$ be defined as in \eqref{def: Q_L}. 
    Then, for any $c > 0$, we have
    \begin{align}
        \bP_L \left(   (1 - c) \ell^* \le \tau_{\mathrm{cov}} \le (1 + c) \ell^* \right) \rightarrow 1 \quad \text{as} \quad L \rightarrow \infty.
    \end{align}
\end{theorem}

The upper bound in Theorem~\ref{thm: covertime} is also a consequence of results in Section~\ref{subsection: SusceptibilityUpper} together with the stochastic domination discussed above. Nevertheless, we still provide a proof since it is short and direct. The more involved part is the lower bound, which will be our main focus below.

Before proving Theorem~\ref{thm: covertime}, we introduce some useful notation. For each $j \in \N^+$, we use $\xi^j = (\xi_n^j)_{n = 0}^\infty$ to denote the trajectory of the $j$-th walker starting from stationarity. Let $\mathcal{R}_\ell^j := \cup_{n = 0}^{\ell} \{ \xi_{n}^j \}$ be the trace of the $j$-th walker until time $\ell$. Let $\mathcal{N}$ be an independent Pois$(\lambda L^d)$ random variable. Then, the cover time $\tau_{\mathrm{cov}}$ can be formally defined as follows:
\begin{align*}
    \tau_{\mathrm{cov}} = \inf\left\{ \ell \in \N: \cup_{j = 1}^{\mathcal{N}} \mathcal{R}_\ell^j = \T^d_L \right\}.
\end{align*}
Moreover, we denote the uncovered set until time $\ell$ by $U_\ell := \T^d_L \setminus ( \cup_{j = 1}^{\mathcal{N}} \mathcal{R}_\ell^j )$. We start with estimating the first moment of $|U_\ell|$ around time $\ell^*$.

\begin{proposition} \label{prop: expected_uncovered}
    Let $\lambda > 0$, $d \ge 2$, $\alpha > 0$, and $K > 0$. Let $Q \in \mathcal{H}_{\alpha, K}$ and let $Q_L$ be defined as in \eqref{def: Q_L}.
    Then, the following statements hold: 
    \begin{itemize}[left=15pt, itemsep=0.1em, topsep=0.1em, rightmargin=25pt]
        \item[(i)] for any $c > 0$, if $\ell_1 = \lceil \ell^*(1 + c) \rceil$, then there exists $L_1 \in \N$ such that $L \ge L_1$ implies
            \begin{align*}
                \bE_L[|U_{\ell_1}|] \le L^{-cd/2};
            \end{align*}
        \item[(ii)] for any $c \in (0, 1)$, if $\ell_2 = \lfloor \ell^*(1 - c) \rfloor$, then there exists $L_2 \in \N$ such that $L \ge L_2$ implies
            \begin{align*}
                \bE_L[|U_{\ell_2}|] \ge L^{cd/2}.
            \end{align*}
    \end{itemize}
\end{proposition}

\begin{remark}
Since $|U_{\ell_1}|$ is an $\N$-valued random variable, (i) together with Markov's inequality imply that $\bP_L(\tau_{\mathrm{cov}} > (1+c)\ell^*) \rightarrow 0$ as $L \rightarrow \infty$, which is the desired upper bound in Theorem~\ref{thm: covertime}. However, to obtain the lower bound, (ii) is not sufficient and we need to control the variance of $|U_{\ell_2}|$, which is the content of Lemma~\ref{lemma: variance_uncovered}. 
\end{remark}

\begin{proof}[Proof of Proposition~\ref{prop: expected_uncovered}]
    Recall that $\mathcal{N}$ is an independent Pois$(\lambda L^d)$ random variable.
    For $x \in \T^d_L$, define the first hitting time of $x$ by the collection of $\mathcal{N}$ stationary walkers as follows:
    \begin{align*}
        \widehat{\tau}_x := \inf \left\{ \ell \in \N: x \in \cup_{j = 1}^{\mathcal{N}} \ \mathcal{R}^j_\ell \right\},
    \end{align*}
    and define $\tau_x$ as the hitting time of $x$ by a single random walk starting from stationarity.
    Let $\bfP^{(L)}_\pi$ (resp. $\bfP^{(L)}_x$) denote the probability measure associated with a single stationary random walk (resp. random walk starting from $x$) on $\T^d_L$.
    We first express the expected size of the uncovered set up to time $\ell$ as a sum of indicators and obtain
    \begin{align*}
        \bE_L\left[|U_{\ell}|\right] = \bE_L \left[ \sum_{x \in \T^d_L} \mathbf{1}_{\{\widehat{\tau}_x > \ell\}} \right] = L^d \exp\left(-\lambda L^d \, \bfP^{(L)}_{\pi}(\tau_{o} \le \ell)\right),
    \end{align*}
    where the second equality uses Poisson thinning and the fact that $\widehat{\tau}_x \stackrel{d}{=} \widehat{\tau}_y$ and $\bfP^{(L)}_\pi(\tau_{x} \le \ell) = \bfP^{(L)}_\pi(\tau_{y} \le \ell)$ for any $x, y \in \T^d_L$. Observe that for any $\ell \in \N$, we have
    \begin{align} \label{eq: 4.6}
        \bfP^{(L)}_{\pi}(\tau_o \le \ell) = \frac{1}{L^d} \sum_{x \in \T^d_L} \bfP^{(L)}_x(\tau_o \le \ell) \ge \frac{\sum_{x \in \T^d_L} G_\ell^{(L)}(o, x)}{L^d \cdot G_\ell^{(L)}} \ge \frac{\ell}{L^d \cdot G_\ell^{(L)}},
    \end{align}
    where the first inequality follows from Lemma~\ref{lemma: cover_to_green} and $G_\ell^{(L)}(x, y) := \sum_{k = 0}^{\ell} Q_L^k(x, y)$ is the Green function on $\T^d_L$ with $G_\ell^{(L)} := G_\ell^{(L)}(o, o)$.
    Therefore, we have 
    \begin{align} \label{eq: 4.7}
        \log \bE_L [|U_{\ell_1}|] &\le d \log L - \frac{\lambda \ell_1}{G_{\ell_1}^{(L)}} \le d \log L ( 1 - \frac{G_{\ell^*}}{G_{\ell_1}^{(L)}} - \frac{c G_{\ell^*}}{G_{\ell_1}^{(L)}}),
    \end{align}
    where in the last inequality, we use $\lambda \ell^* \ge d \log L \cdot G_{\ell^*}$ by the definition of $\ell^*$.
    We will show that for any $c > 0$, there exists $L_0$ such that for any $L \ge L_0$, we have $G_{\ell^*}/G_{\ell_1}^{(L)} \ge (2 + c)/(2 + 2c)$. This together with \eqref{eq: 4.7} then implies the desired result in (i). 
    By Lemma~\ref{lemma: heat_kernel_from_Z_to_T}, we have
    \begin{align*}
        G_{\ell_1}^{(L)} &= \sum_{k = 0}^{\ell_1} Q_L^k(o, o) \le \sum_{k = 0}^{\ell_1} Q^k(o, o) + \frac{C}{L^{\alpha + d}} \sum_{k = 0}^{\ell_1} k^{\alpha + d + 1} \\
        &\le G_{\ell^*} + G_{\ell_1} - G_{\ell^*} + \frac{C' \ell_1^{\, \alpha + d + 2}}{L^{\alpha + d}} 
    \end{align*}
    Since $C' \ell_1^{\, \alpha + d + 2}/L^{\alpha + d}$ goes to 0 as $L \rightarrow \infty$, it suffices to show that $(G_{\ell_1} {-} G_{\ell^*})/G_{\ell^*} < c/(2+c)$ holds for all large $L$.
    In fact, we will show that $G_{\ell_1} {-} G_{\ell^*} = o(G_{\ell^*})$ as $L \rightarrow \infty$.
    We divide the argument into three cases:
    When the random walk driven by $Q$ is transient, i.e. (1) $\alpha < 2$ and $d = 2$ or (2) $d \ge 3$, we have $G_{\ell_1} {-} G_{\ell^*} \le \sum_{k \ge \ell^* {+} 1} Q^k(o, o) = o(1)$ as $L \rightarrow \infty$. When $d = 2$ and $\alpha > 2$, by Lemma~\ref{lemma: green_function_on_Z}, we have $G_{\ell_1} {-} G_{\ell^*} \le C_0 \log(\ell_1/\ell^*) \le C_0 \log(1 {+} c)$ for some constant $C_0 > 0$ depending on $Q$. Since $G_{\ell^*}$ diverges as $L \rightarrow \infty$, we also have $(G_{\ell_1} {-} G_{\ell^*})/G_{\ell^*} \rightarrow 0$ as $L \rightarrow \infty$. When $d = 2$ and $\alpha = 2$, by Lemma~\ref{lemma: green_function_on_Z}, we have $G_{\ell_1} {-} G_{\ell^*} \le C_1 \log(\log(\ell_1)/\log(\ell^*))$ for some constant $C_1 > 0$ depending on $Q$. Since $G_{\ell^*}$ diverges as $L \rightarrow \infty$, we again have $(G_{\ell_1} {-} G_{\ell^*})/G_{\ell^*} \rightarrow 0$ as $L \rightarrow \infty$. This completes the proof of (i).
    
    We now proceed to prove (ii). Similar to \eqref{eq: 4.6}, we also have for any $\ell > 0$ that
    \begin{align*}
        \bfP^{(L)}_{\pi}(\tau_o \le \ell) = \frac{1}{L^d} \sum_{x \in \T^d_L} \bfP^{(L)}_x(\tau_o \le \ell) \le \frac{\sum_{x \in \T^d_L} G_{\ell+\epsilon\ell}^{(L)}(o, x)}{L^d \cdot G_{\epsilon \ell}^{(L)}} = \frac{(1+\epsilon)\ell + 1}{L^d \cdot G_{\epsilon \ell}^{(L)}},
    \end{align*}
    for any $\epsilon \in (0, 1)$, where the inequality follows from Lemma~\ref{lemma: cover_to_green}.
    Fix $\epsilon := c/(2c + 2)$. 
    Similar to \eqref{eq: 4.7} we have
    \begin{align*}
        \log \bE_L [|U_{\ell_2}|] &\ge d \log L - \frac{\lambda \left((1+\epsilon)\ell_2+1\right)}{G_{\epsilon \ell_2}^{(L)}} \\
        &\ge d \log L \left( 1 - \frac{G_{\ell^*-1}}{G_{\epsilon \ell_2}}\left((1+\epsilon)(1 - c) + \frac{1}{\ell^*-1}\right)\right),
    \end{align*}
    where in the second inequality, we use $\lambda(\ell^*{-}1) \le G_{\ell^*{-}1} \cdot d \log L $, which follows from the definition of $\ell^*$. We also use the fact that $G_{\ell} \le G_{\ell}^{(L)}$ for any $\ell$ and $L$.
    Take $L_2$ large enough such that for any $L \ge L_2$, we have $1/(\ell^*{-}1) \le \epsilon c$. Therefore, in order to complete the proof of (ii), it suffices to show that for all $L \ge L_2$ (with possibly larger $L_2$), we have 
    $$\frac{G_{\ell^*-1}}{G_{\epsilon \ell_2}} \le \frac{1 - \frac{c}{2}}{1 + \epsilon - c}.$$
    Note that the quantity on the right-hand side is strictly larger than 1.
    In fact, we will show that $G_{\ell^*-1}/G_{\epsilon \ell_2}$ goes to 1 as $L$ goes to infinity.
    Observe that 
    \begin{align}\label{eq: 4.8}
         G_{\ell^*-1} \le G_{\epsilon \ell_2} + \sum_{k = \lfloor \epsilon(1-c)\ell^* \rfloor}^{\ell^*} Q^k(o, o).
    \end{align}
    We argue analogously as before in three cases.
    When the random walk is transient, the sum in equation \eqref{eq: 4.8} goes to 0 as $L$ goes to infinity. 
    When $d = 2$ and $\alpha > 2$, by Lemma~\ref{lemma: green_function_on_Z}, we have $\sum_{k = \lfloor \epsilon(1-c)\ell^* \rfloor}^{\ell^*} Q^k(o, o) \le C_0 \log(1/\epsilon(1{-}c))$ for some constant $C_0 > 0$ depending on $Q$. Since $G_{\epsilon \ell_2}$ diverges as $L$ goes to infinity, we have $\sum_{k = \lfloor \epsilon(1-c)\ell^* \rfloor}^{\ell^*} Q^k(o, o)/G_{\epsilon \ell_2} \rightarrow 0$ as $L$ goes to infinity. When $d = 2$ and $\alpha = 2$, by Lemma~\ref{lemma: green_function_on_Z}, we have $\sum_{k = \lfloor \epsilon(1-c)\ell^* \rfloor}^{\ell^*} Q^k(o, o) \le C_1 \log(\log(1/\epsilon(1{-}c)))$ for some constant $C_1 > 0$ depending on $Q$. Since $G_{\epsilon \ell_2}$ diverges as $L$ goes to infinity, we again have $\sum_{k = \lfloor \epsilon(1-c)\ell^* \rfloor}^{\ell^*} Q^k(o, o)/G_{\epsilon \ell_2} \rightarrow 0$ as $L$ goes to infinity. This completes the proof of (ii).
\end{proof}

The next step is to control the variance of $|U_{\ell_2}|$, which is the content of Lemma~\ref{lemma: variance_uncovered}.
Let $\mathrm{Var}_L$ and $\mathrm{Cov}_L$ denote the variance and covariance under $\bP_L$. 

\begin{lemma} \label{lemma: variance_uncovered}
    Under the same setting as in Proposition~\ref{prop: expected_uncovered}, we have that 
    \begin{align*}
        \lim_{L \rightarrow \infty} \frac{\mathrm{Var}_L\left[\left\vert U_{\ell_2}\right\vert\right]}{\bE_L\left[\left\vert  U_{\ell_2} \right\vert\right]^2} = 0.
    \end{align*}
\end{lemma}

\begin{proof}
    Let $h_x := \1_{\{\widehat{\tau}_x > \ell_2 \}}$. Then, we have
    \begin{align} \label{ineq: decompose_variance}
        \mathrm{Var}_L\left[\left\vert U_{\ell_2}\right\vert\right] \le \bE_L\left[\left\vert U_{\ell_2}\right\vert\right] + L^d \sum_{x \neq o} \mathrm{Cov}_L(h_o, h_x).
    \end{align}
    From Proposition~\ref{prop: expected_uncovered} (ii), we have $\bE_L\left[\left\vert U_{\ell_2}\right\vert\right] \ge L^{cd/2}$ for all large $L$.
    Therefore, in order to prove Lemma~\ref{lemma: variance_uncovered}, it is sufficient to prove
    \begin{align*}
        \lim_{L \rightarrow \infty} \frac{L^d \sum_{x \neq o} \mathrm{Cov}_L(h_o, h_x)}{\bE_L\left[\left\vert U_{\ell_2}\right\vert\right]^2} = 0,
    \end{align*}
    By Poisson thinning, we have
    \begin{align*} 
        \mathrm{Cov}_L(h_o, h_x) = \ & \bP_L(\widehat{\tau}_{o} \wedge \widehat{\tau}_{x} > \ell_2) - \bP_L(\widehat{\tau}_o > \ell_2)^2 \\
        = \ & \exp(-\lambda L^d \bfP_{\pi}^{(L)}(\tau_o \wedge \tau_x \le \ell_2)) - \exp(- 2 \lambda L^d \bfP_{\pi}^{(L)}(\tau_o \le \ell_2))\\
        = \ & \exp(- 2 \lambda L^d  \bfP_{\pi}^{(L)}(\tau_o \le \ell_2)) \left( \exp( \lambda L^d \bfP_{\pi}^{(L)}(\tau_o \vee \tau_x \le \ell_2))  - 1 \right) \\
        = \ & \bE_L[h_o]^2 \left( \exp( \lambda L^d \bfP_{\pi}^{(L)}(\tau_o \vee \tau_x \le \ell_2))  - 1 \right),   
    \end{align*}
    By the strong Markov property and transitivity of $\T^d_L$, we have
    \begin{align} \label{ineq: hitting_min_to_max}
        \bfP_{\pi}^{(L)}(\tau_o \vee \tau_x \le \ell_2) \le 2 \bfP_{\pi}^{(L)} (\tau_o \le \ell_2 ) \bfP_{o}^{(L)} (\tau_x \le \ell_2) \le \frac{4 \ell_2}{L^d G_{\ell_2}^{(L)}} \bfP_{o}^{(L)} (\tau_x \le \ell_2),
    \end{align}
    where we remind the reader that $G^{(L)}_\ell := \sum_{ k = 0}^{\ell} Q^k_{L}(o, o)$. Moreover, in the last inequality, we use the fact that $\bfP_{\pi}^{(L)} (\tau_o \le \ell_2 ) = \sum_{y \in \T^d_L} \bfP_{y}^{(L)} (\tau_o \le \ell_2 ) / L^d$ and Lemma~\ref{lemma: cover_to_green}. 
    
    Since $G_{\ell_2}^{(L)} \ge G_{\ell_2}$, for all sufficiently large $L$, we have $\ell_2/G_{\ell_2}^{(L)} \le (\ell^*{-}1)/G_{\ell^*{-}1}$. Define 
    $$S := \left\{x \in \T^d_L : \bfP_{o}^{(L)} (\tau_x \le \ell_2) \le 1/ (\log L \cdot \log \log L) \right\}.$$ 
    For any $x \in S$, by possibly further enlarging $L$, we have 
    \begin{align*}
        \exp( \lambda L^d \bfP_{\pi}^{(L)}(\tau_o \vee \tau_x \le \ell_2)) - 1 \le \exp(\frac{4 d}{\log \log L}) - 1 \le \frac{8d}{\log \log L},
    \end{align*}
    where the last inequality uses $\exp(\delta) - 1 \le 2\delta$ for any $\delta \in (0, 1)$. 
    Therefore, we have 
    \begin{align*}
    L^d \sum_{x \neq o} \mathrm{Cov}_L(h_o, h_x) 
    \le \ & L^{2d} \bE_L[h_o]^2 \frac{8d}{\log \log L} + L^d \bE_L[h_o] \ |S^c| \\
    = \ & \bE_L[|U_{\ell_2}|]^2 \frac{8d}{\log \log L} + \bE_L[|U_{\ell_2}|] \ |S^c| ,
    \end{align*}
    where, for $x \in S^c$, we use the fact that $\mathrm{Cov}_L(h_o, h_x) \le \bE_L[h_o]$. Therefore, it is sufficient to show that $|S^c| / \bE_L[|U_{\ell_2}|]$ vanishes as $L$ goes to infinity.
    Since we have
    \begin{align}
        \sum_{x \in \T^d_L} \bfP_{o}^{(L)} (\tau_x \le \ell_2) \le \sum_{x \in \T^d_L} G^{(L)}_{\ell_2}(o, x) = \ell_2+1,
    \end{align}
    it follows that $|S^c| \le (\ell_2+1) \log L \log \log L$. Since $\ell_2+1 \le \ell^*$ for all sufficiently large $L$ and we know that $\bE_{L}[|U_{\ell_2}|] \ge L^{cd/2}$, it follows from the asymptotics of $\ell^*$ (see equation \eqref{eq: Susc_asymp_t*}) that $|S^c| / \bE_L[|U_{\ell_2}|]$ vanishes as $L$ goes to infinity. This completes the proof of Lemma~\ref{lemma: variance_uncovered}.
\end{proof}

\begin{proof}[Proof of Theorem~\ref{thm: covertime}]
    Since $|U_\ell|$ is an $\N$-valued random variable, for any $c > 0$, we have 
    $$\bP_L(\tau_{\mathrm{cov}} \ge \ell_1) \le \bP_L(|U_{\ell_1}| \ge 1) \le \bE_L[|U_{\ell_1}|] \rightarrow 0$$ 
    as $L$ diverges to infinity by (i) in Proposition~\ref{prop: expected_uncovered}. 
    On the other hand, for any $c \in (0, 1)$, we have
    \begin{align*}
      \bP_L(\tau_{\mathrm{cov}} \le \ell_2) = \bP_L(|U_{\ell_2}| = 0) \le \ \bP_L \Big( \Big\vert |U_{\ell_2}| - \bE_L[|U_{\ell_2}|] \Big\vert \ge \bE_L[|U_{\ell_2}|] \Big) \le \frac{\mathrm{Var}_L[|U_{\ell_2}|]}{\bE_L[|U_{\ell_2}|]^2},
    \end{align*}
    which goes to $0$ as $L$ diverges by Lemma \ref{lemma: variance_uncovered}. This completes the proof. 
\end{proof} 

\section*{Acknowledgement}
The authors thank Itai Benjamini for suggesting the study of the long-range frog model, Matthew Junge for valuable discussions during the early stages of this project, and Mathav Murugan and Tom Hutchcroft for helpful discussions.

The first author was supported in part by NSERC, an SLMath Clay senior scholarship, and Magdalen college, Oxford. The second and third authors were supported in part by NSERC. 

\appendix 

\section{Heat Kernel Estimates}
In this section, we establish the necessary heat kernel estimates for heavy-tailed random walks on $\Z^d$ with $d \ge 2$. 
Throughout, we assume that $Q \in \mathcal{H}_{\alpha, K}$ for some $\alpha > 0$ and the walks are in discrete-time.
Recall, in \eqref{def: H_alpha}, we defined
\begin{align}
\mathcal{H}_{\alpha, K} := \left\{ Q : Q \begin{array}{c} \text{ is symmetric, translation-invariant,} \\ \text{ and heavy-tailed with parameters } \alpha \text{ and } K\end{array} \right\}.
\end{align}
Let $X_n$ be a random walk on $\Z^d$ with transition matrix $Q$. 
We denote by $Q^n(x, y) := \bfP_x(X_n = y)$ the $n$-step transition kernel, where $\bfP_x$ is the law of the random walk starting from $x \in \Z^d$. 
Let $|x|$ be the $\ell^1$-norm of $x \in \Z^d$.

The following theorem is a direct consequence of Theorem 1.7 in \cite{MR3644017}. The original theorem in \cite{MR3644017} is stated in much more general settings, namely on finitely generated groups of polynomial volume growth with symmetric one-step transition kernels. For our purposes, we only require the special case of $\Z^d$ with heavy-tailed random walks. For reader's convenience, we note that the $r$ function below is the same $r$ function defined in Theorem 1.7 of \cite{MR3644017} with their specific form worked out through a straightforward computation.

\begin{theorem} \label{thm: heat_kernel_typical}
    Let $d \ge 2$ and $Q \in \mathcal{H}_{\alpha, K}$ for some $\alpha > 0$.
    Then, there exists $C = C(Q) \in (0, \infty)$ and, for any $\kappa > 0$, there exists $c = c(\kappa, Q) > 0$ such that, for all $n \ge 1$ and $|x| \le \kappa r(n)$,
    \begin{align} \label{ineq: heat_kernel_typical}
        c \frac{1}{r(n)^d} \le Q^{n}(o, x) \le C \frac{1}{r(n)^d}.
    \end{align}
    where
    \begin{align} \label{def: typical_distance}
        r(n) := \begin{cases}
            \sqrt{n} &\text{ if } \alpha > 2, \\
            \sqrt{n \log n} &\text{ if } \alpha =2, \\
            n^{1/\alpha} &\text{ if } \alpha < 2.
        \end{cases}
    \end{align}
    Moreover, for every $\epsilon > 0$ there exists $K$ such that
    \begin{align} \label{ineq: heat_kernel_intypical}
        \sum_{m = 0}^n \sum_{\substack{y \in \Z^d \\ |y| > K r(n)}} Q^m(o, y) \le n \epsilon.
    \end{align}
\end{theorem}

\begin{remark}
    Theorem 1.7 in \cite{MR3644017} implies that the $n$-step transition kernel is ``strongly controlled'' in the sense of Definition 1.3 in \cite{MR3644017}. Note that equation \eqref{ineq: heat_kernel_typical} is exactly the first item in the definition of strongly controlled. On the other hand, in the proof of Theorem 1.7, the authors first shows that the $n$-step transition kernel is ``controlled'' in the sense of Definition 1.2. Then, equation \ref{ineq: heat_kernel_intypical} follows from the second item in the definition of controlled.
\end{remark}

We now state two lemmas that are immediate consequences of Theorem~\ref{thm: heat_kernel_typical}.
Define the Green's function up to time $n \in \N \cup \{ \infty \}$ by $G_n(x, y) := \sum_{k = 0}^n Q^k(x, y)$ and $G_n := G_n(o, o)$. The following lemma gives the asymptotic behavior of $G_n$ for different values of $\alpha$ and $d$. With a slight abuse of notation, for vertex sets $A, B \subseteq \Z^d$, we define $G_n(A, B) := \sum_{x \in A} \sum_{y \in B} G_n(x, y)$.

\begin{lemma} \label{lemma: green_function_on_Z}
Let $d \ge 2$ and $Q \in \mathcal{H}_{\alpha, K}$ for some $\alpha, K > 0$. 
The Green's function on $\Z^d$ satisfies the following: \\
(i) if $\alpha > 2$ and $d = 2$, then
\begin{align*}
    G_n \asymp \log n,
\end{align*}
(ii) if $\alpha = 2$ and $d = 2$, then
\begin{align*}
    G_n \asymp \log \log n,
\end{align*}
(iii) if $\alpha < 2$ and $d = 2$ OR $d \ge 3$, then 
\begin{align*}
    G_n \asymp 1,
\end{align*}
where the implicit constants in $\asymp$ depend on $Q$.
\end{lemma}

\begin{proof}
This follows from the on-diagonal heat kernel estimates given by \eqref{ineq: heat_kernel_typical}.
\end{proof}

\begin{lemma} \label{lemma: green_function_compare_n_and_Cn}
    Let $d \ge 2$ and $Q \in \mathcal{H}_{\alpha, K}$ for some $\alpha, K > 0$. For any $C > 0$ and any $\epsilon > 0$, there exists $L_0$ such that for any $L \ge L_0$, we have $G_{Cn} \le (1+\epsilon) G_n$.
\end{lemma}
\begin{proof}
    It is sufficient to show that $\sum_{k = n+1}^{Cn} Q^k(o, o) = o(G_n)$ as $n$ goes to infinity. By using equation \eqref{ineq: heat_kernel_typical} and Lemma~\ref{lemma: green_function_on_Z}, it is then a straightforward computation. We omit the details.
\end{proof}

We now present a lemma that was used for multiple times throughout this paper. 

\begin{lemma} \label{lemma: cover_to_green}
Let $A, B \subseteq \Z^d$ and $\ell \in \N^+$, then we have 
\begin{align*}
	\sum_{y \in B} \bfP_y(\tau_{A} \le \ell) \ge \frac{G_{\ell}(A, B)}{\max_{x \in A} G_{\ell}(x, A)}.
\end{align*}
Similarly, for any $\ell_1, \ell_2 \in \N^+$ we also have 
\begin{align*}
\sum_{y \in B} \bfP_y(\tau_{A} \le \ell_1) \le \frac{G_{\ell_1 + \ell_2}(A, B)}{\max_{x \in A} G_{\ell_2}(x, A)}.
\end{align*}
Moreover, the same inequalities hold on the torus $\T^d_L$ by replacing $G_\ell$ with $G^{(L)}_\ell$ and $\bfP_y$ with $\bfP^{(L)}_y$ for any $L \in \N^+$.
\end{lemma}

\begin{proof}
Let $N_S^{(\ell)}$ be the number of visits to $S \subseteq \Z^d$ up to time $\ell \in \N$, and let $\tau_S$ be the hitting time of subset $S$. For any $\ell$, we have that
\begin{align*}
    G_\ell(A, B) 
    & = \sum_{y \in B} \bfE_y[N_A^{(\ell)}] \\
    & = \sum_{y \in B} \sum_{x \in A} \bfE_y[N_A^{(\ell)} | \tau_A \le \ell, \tau_A = \tau_x] \bfP_y(\tau_A \le \ell, \tau_A = \tau_x) \\
    & \le \sum_{y \in B} \sum_{x \in A} \bfE_x[N_A^{(\ell)}] \bfP_y(\tau_A \le \ell, \tau_A = \tau_x) \\
    & = \max_{x \in A} G_\ell(x, A) \sum_{y \in B} \bfP_y(\tau_A \le \ell).
\end{align*}
The first inequality follows from the strong Markov property, which ensures that 
$$\bfE_y[N_A^{(\ell)} | \tau_A \le \ell, \tau_A = \tau_x] \le \bfE_x[N_A^{(\ell)}]$$ 
for any $x \in A$. This completes the proof of the first inequality.

Similarly, we can prove the second inequality in Lemma \ref{lemma: cover_to_green} by observing that
$$\bfE_x[N_A^{(\ell_2)}] \le \bfE_y[N_A^{(\ell_1 + \ell_2)} | \tau_A \le \ell_1, \tau_A = \tau_x]$$
for any $x \in A$, which again follows from the strong Markov property.
The proof of the same inequalities on $\T^d_L$ is identical and thus omitted.
\end{proof}

Let $\T^d_L := (\Z / L \Z)^d$ and 
let $\pi: \Z^d \rightarrow \T^d_L$ be the canonical projection.
For $u \in \T^d_L$, 
let $\widecheck{u}$ be the vertex in the singleton set $\pi^{-1}(u) \cap [-L/2, L/2)^d$.
For a random walk $\{X_n\}_{n \in \N^+}$ on $\Z^d$ with translation-invariant transition matrix $Q$, we denote by $\{X_n^{(L)}\}_{n \in \N^+}$ the corresponding random walk on $\T^d_L$ obtained by projecting $X_n$ through $\pi$. Its transition matrix is denoted by $Q_L$, where $Q_L(x, y) := \sum_{z \in \widecheck{y} + L\Z^d} Q(\widecheck{x}, z)$ for any $x, y \in \T^d_L$. We use $o$ to denote the origin of either $\Z^d$ or $\T^d_L$.

Let $G^{(L)}_n(x, y) := \sum_{k = 0}^n Q^k_L(x, y)$ denote the Green's function on $\T^d_L$ up to time $n \in \N \cup \{\infty\}$, and let $G^{(L)}_n := G^{(L)}_n(o, o)$. 
Throughout our analysis, we need to go back and forth between Green's functions on $\Z^d$ and $\T^d_L$.
Therefore, estimates that relate $G^{(L)}_n(x, y)$ and $G_n(x, y)$ are needed.
Luckily, the following simple observation suffices for our purposes, and we state it without a proof.

The next result allows us comparing heat kernels on $\T^d_L$ and $\Z^d$. The bound below is far from optimal, but it suffices for our purposes. 
\begin{lemma} \label{lemma: heat_kernel_from_Z_to_T}
For any $x \in \T^d_L$, $n \in \N$, and $L \in \N^+$, we have 
    \begin{align*}
        0 \le Q^n_L(o, x) - Q^n(o, \widecheck{x}) \le \frac{C \ n^{d + \alpha + 1}}{L^{d + \alpha}}
    \end{align*}
for some $C = C(\alpha, L)$. 
\end{lemma}

\begin{proof}
We use the natural coupling via the projection $\pi$ between the random walks on $\Z^d$ and $\T^d_L$.
Let $\{X_n\}_{n \in \N^+}$ be a random walk on $\Z^d$ with transition matrix $Q$ starting from $o$. 
Then,
    \begin{align*}
        & \hspace{10mm} Q^n_L(o, x) - Q^n(o, \widecheck{x}) 
        \le \bfP_o \left( X_n \notin [-L/2, L/2)^d \right)  \\
        \le \ & \bfP_o \left( \exists \ i \in \{1, \ldots, n\} \text{ s.t. } |X_i - X_{i-1}| \ge L/2n \right)
        \le \ \frac{C(\alpha, d) \ n^{d + \alpha + 1}}{L^{d+\alpha}}.
    \end{align*}
where the last inequality uses a union bound. 
\end{proof}

Let $B_x(r) := \{y \in \T^d_L : |y - x| \le r\}$, where $|\cdot|$ is the $\ell^\infty$ distance on $\T^d_L$.
Recall the definition of $\ell(r)$ in \eqref{def: s(r)}, which represents the \textit{typical time} needed for a random walker to travel a distance proportional to $r$, which is asymptotically the inverse of the function $r(\cdot)$ defined in \eqref{def: typical_distance}.

\begin{lemma} \label{apx_lemma: random_walk_expected_hitting_size_lower_bound}
    Let $Q \in \mathcal{H}_{\alpha, K}$ for some $\alpha, K > 0$. Fix $K_0 \ge 1$ and $\epsilon \in (0, 1)$.
    Then, there exists positive constants $C = C(Q)$ and $c = c(Q, K_0)$ such that the following holds:
    for any $r, L \in \N^+$ and
    any subset $D \subseteq B_x(K_0 r)$ with $|D| \ge \epsilon r^d$, we have

        \vspace{2mm}

        \begin{align}
            \epsilon c f(r) \le \sum_{y \in \widecheck{D}} \bfP_y \left( \tau_{\widecheck{x}} \le \ell(r) \right) \le \sum_{y \in D} \bfP^L_y \left( \tau_x \le \ell(r) \right) \le C f(r),
        \end{align}
        where $\widecheck{D} := \{ \widecheck{y} : y \in D \}$ and $f(r)$ is defined as follows:
        \begin{align*}
            f(r) := \begin{cases}
                \ell(r) & \text{if } d \ge 3 \text{ or } (d = 2 \text{ and } \alpha < 2), \\
                \ell(r) / \log \log \ell(r) & \text{if } d = 2 \text{ and } \alpha = 2, \\
                \ell(r) / \log \ell(r) & \text{if } d = 2 \text{ and } \alpha > 2.
            \end{cases}
        \end{align*}
\end{lemma}

\begin{proof}
    We first prove the middle inequality. Let $N^{(L)}$ be the number of (distinct) vertices in $D$ visited by the random walk on $\T^d_L$ starting from $x$ up to time $\ell(r)$. Let $N$ be the corresponding random variable for the random walk on $\Z^d$ starting from $\widecheck{x}$ and visiting vertices in $\widecheck{D}$. It follows from the coupling given by the projection map that $N^{(L)}$ stochastically dominates $N$, which gives us the middle inequality.

    We now prove the lower bound. By Lemma~\ref{lemma: cover_to_green}, we have
        \begin{align*}
            \sum_{y \in \widecheck{D}} \bfP_y(\tau_{\widecheck{x}} \le \ell(r)) \ge \frac{G_{\ell(r)}(\widecheck{x}, \widecheck{D})}{G_{\ell(r)}}.
        \end{align*}
    By Theorem~\ref{thm: heat_kernel_typical}, for any $y \in B_{\widecheck{x}}(K_0 r) \subseteq \Z^d$, we have
        \begin{align*}
            G_{\ell(r)}(\widecheck{x}, y) \ge \sum_{k = \lceil \ell(r)/2 \rceil}^{\ell(r)} Q^k(\widecheck{x}, y) \ge c' \ell(r) / r^d,
        \end{align*}
    where $c'$ is a positive constant depending on $Q$ and $K_0$. We complete the proof of the lower bound by combining the above two inequalities with Lemma~\ref{lemma: green_function_on_Z} and the assumption that $|D| \ge \epsilon r^d$.
    
    For the upper bound, it follows from the second display of Lemma~\ref{lemma: cover_to_green}\footnote{Although Lemma~\ref{lemma: cover_to_green} is stated for random walks on $\Z^d$, the same argument applies to random walks on $\T^d_L$} that
        \begin{align*}
            \sum_{y \in D} \bfP^L_y(\tau_{x} \le \ell(r)) \le \frac{G^{(L)}_{2\ell(r)}(x, D)}{G^{(L)}_{\ell(r)}} \le \frac{3\ell(r)}{G_{\ell(r)}},
        \end{align*}
    where in the last inequality we use the fact that $G^{(L)}_{2\ell(r)}(x, D) \le 2 \ell(r) + 1 \le 3 \ell(r)$. We conclude the proof by plugging in the estimates for $G_{\ell(r)}$ from Lemma~\ref{lemma: green_function_on_Z}.
\end{proof}

Let $\mathrm{R}(k) := \{X_0, X_1, \ldots, X_k\}$ be the \textit{range} of the first $k$ steps of $\{X_n\}_{n \in \N}$.

\begin{lemma} \label{lemma: range_estimates}
        Let $Q \in \mathcal{H}_{\alpha, K}$ for some $\alpha, K > 0$. Fix an arbitrary $\epsilon \in (0, 1)$.
        Then, there exist positive constants $c_0 = c_0(Q, \epsilon)$ and $c_1 = c_1(Q)$ such that the following holds:
        for 
        any $r \in \N^+$ and 
        any subset $D \subseteq B_x(r)$ with $|D| \le (1-\epsilon) |B_x(r)|$, 
        we have 
        \begin{align} \label{ineq: range_est_case1}
        \bfP_x \left( \, \left\vert \widetilde{\mathrm{R}}\left(\ell(r)\right) \right\vert \ge c_1 \epsilon f(r) \, \right) &\ge  c_0,
        \end{align}
        where $\widetilde{\mathrm{R}}(k) := \mathrm{R}(k) \cap B_x(r) \setminus D$ and $f(r)$ and $\ell(r)$ are defined as above.
\end{lemma}

\begin{proof}
By Lemma~\ref{apx_lemma: random_walk_expected_hitting_size_lower_bound}, we have the following lower bound on the expected size of $\widetilde{\mathrm{R}}(\ell(r))$:
\begin{align} \label{ineq: lower_bound_expected_new_range}
    \epsilon c' f(r) \le \bfE_x\left[\left|\widetilde{\mathrm{R}}\left(\ell(r)\right)\right| \right] 
    &= \sum_{y \in B_x(r) \setminus D} \bfP_x(\tau_y \le \ell(r)) \le C f(r),
\end{align}
where $c' = c'(Q, K)$ and $C = C(Q)$ are positive constants.

Let $b := C f(r)$. For $k \in \N^+$, define $a(k) := \max_{y \in B_x(r)} \bfP_y(|\widetilde{\mathrm{R}}(\ell(r))| \ge 2kb)$. Observe that $a(1) \le 1/2$ and, by strong Markov Property, we have $a(k_1 + k_2) \le a(k_1)a(k_2)$ for any $k_1, k_2 \in \N^+$. It then follows that $\bfE_x[|\widetilde{\mathrm{R}}(\ell(r))|^2] \le C_0 b^2$, where $C_0$ is some absolute constant. Then, the result follows from the Paley-Zygmund inequality: 
\begin{align}
    \bfP_x \left( \left|\widetilde{\mathrm{R}}(\ell(r))\right| \ge \frac{\epsilon c' f(r)}{2} \right) \ge \frac{\bfE_x[|\widetilde{\mathrm{R}}(\ell(r))|]^2}{4\bfE_x[|\widetilde{\mathrm{R}}(\ell(r))|^2]} \ge \frac{(\epsilon c' f(r))^2}{4C_0 b^2} = \frac{(\epsilon c')^2}{4C_0 C^2} =: c_0.
\end{align}

\end{proof}

\begin{remark}
    A more careful estimation on the first moment of $| \widetilde{\mathrm{R}}(\ell(r)) |$ would give us $c_0 = c_0(Q)$ without the denpendency on $\epsilon$ and $c_1 = c_2(Q) (1 - \frac{|D|}{|B_x(r)|})$. However, the current form suffices for our purposes.
\end{remark}

\begin{lemma} \label{lemma: one-step_Z_to_T}
    Let $Q \in \mathcal{H}_{\alpha, K}$ for some $\alpha, K > 0$. Then there exists $C, c > 0$ depending on $(K, \alpha, d)$ such that for any $L \in \N^+$ and any $x \in \T^d_L$, we have
    \begin{align}
        \frac{c}{|o - x|^{\alpha + d}} \le Q_L(o, x) \le \frac{C}{|o - x|^{\alpha + d}}.
    \end{align}
\end{lemma}

\begin{proof}
    The lower bound follows directly from the definition of $Q_L$ and the lower bound in the definition of $\mathcal{H}_{\alpha, K}$ (see \eqref{def: H_alpha}). In particular, we set $c = 1/K$. For the upper bound, we have 
    \begin{align*}
        Q_L(o, x) = Q(o, \widecheck{x}) + \sum_{\substack{z \in \Z^d \\ z \neq 0}} Q(o, \widecheck{x} + Lz) \le \frac{K}{|o - \widecheck{x}|^{\alpha + d}} + \frac{C(K,\alpha,d)}{L^{\alpha + d}} \le \frac{K + C(K,\alpha,d)}{|o - x|^{\alpha + d}},
    \end{align*}
    where we used the fact that $|o - \widecheck{x}| = |o - x|$ and the upper bound in the definition of $\mathcal{H}_{\alpha, K}$ in the second inequality.
    And the bound for the second term in the first ineqaulity can be derived by comparing the sum with a integral. We omit the details here.
    By setting $C := K + C(K,\alpha,d)$, we complete the proof.
\end{proof}

\begin{lemma} \label{lemma: hitting_probability_upper_bound_on_Z}
For every $\ell > 0$, there exists $C = C(d, \alpha, K, \ell) > 0$ such that
\begin{align}
    \frac{C^{-1}}{|x - y|^{\alpha + d}} \le \bfP_x(\tau_y \le \ell) \le \frac{C}{|x - y|^{\alpha + d}},
\end{align}
for any $x, y \in \Z^d$ with $x \neq y$.
\end{lemma}

\begin{proof}
    The lower bound is a direct consequence of $Q(x, y) \le \bfP_x(\tau_y \le \ell)$ and the assumption that $Q \in \mathcal{H}_{\alpha, K}$. 
    For the upper bound, we first prove the following claim inductively: for any $n \ge 1$, there exists a constant $C_n$ depending on $(\alpha, d, K, n)$ such that 
    \begin{align*}
        Q^n(x, y) \le \frac{C_n}{|x - y|^{\alpha + d}}.
    \end{align*}
    By the assumption that $Q \in \mathcal{H}_\alpha$, the upper bound trivially holds for $n = 1$ with $C_1 = K$. Assume the upper bound holds for $n = k-1$, then 
    \begin{align*}
        Q^k(x, y) 
        &= \sum_{z \in \Z^d} Q(x, z) Q^{k-1}(z, y) \\
        &\le \sum_{\substack{z \in \Z^d \text{ s.t. }\\ |x - z| \ge |x - y|/2 }} Q(x, z) Q^{k-1}(z, y) + \sum_{\substack{z \in \Z^d \text{ s.t. }\\ |y - z| \ge |x - y|/2 }} Q(x, z) Q^{k-1}(z, y) \\
        &\le \frac{2^{\alpha + d}K}{|x - y|^{\alpha + d}} + \frac{2^{\alpha + d}C_{k-1}}{|x - y|^{\alpha + d}},
    \end{align*}
    where to get the last line, we also used the fact that $Q$ is symmetric.
    Setting $C_k := 2^{\alpha + d}K + 2^{\alpha + d}C_{k-1}$ proves the upper bound.
    Therefore, we conlude the proof by observing that $\bfP_x(\tau_y \le \ell) \le \sum_{n = 1}^{\ell} Q^n(x, y)$.
\end{proof}

\section{Large Deviation Estimates for sums of Bernoulli Random Variables}

\begin{lemma} \label{lemma: large_deviation_Bern}
Let $\xi_1, ..., \xi_n$ be a sequence of i.i.d. $\Bern(p)$ random variable with $p \in (0, 1)$. Let $S_n := \sum_{i = 1}^n \xi_i$, then we have
\begin{align}
    &\text{ for } \forall \delta > 0, \hspace{2mm} \mathbf{P}\left(\frac{S_n}{n} \ge p(1 + \delta)\right) \le \exp\left(- \frac{np\delta \log(1+\delta)}{4}\right), \\
    &\text{ for } \forall \delta \in (0, 1), \hspace{2mm} \mathbf{P}\left(\frac{S_n}{n} \le p(1 - \delta)\right) \le \exp\left(- \frac{np\delta^2}{4}\right), \label{ineq: large_deviation_Bin_2} \\
    &\text{ for } \forall \delta > 0 \text{ and } k \in \N, \hspace{2mm} \mathbf{P}\left(\sup_{n \ge k} \frac{S_n}{n} \ge p(1 + \delta)\right) \le \frac{\exp\left(- \frac{kp\delta \log(1 + \delta)}{4}\right)}{1 - \exp\left(- \frac{p\delta \log(1 + \delta)}{4}\right)} \label{ineq: large_deviation_Bin_3} \\
    &\text{ for } \forall \delta \in (0, 1) \text{ and } k \in \N, \hspace{2mm} \mathbf{P}\left(\inf_{n \ge k} \frac{S_n}{n} \le p(1 - \delta)\right) \le 8 \delta^{-2} \exp\left(- \frac{kp\delta^2}{4}\right) \label{ineq: large_deviation_Bin_4}
\end{align}
\end{lemma}
We refer the reader to Fact B.1 in \cite{MR4068310} for a proof of the above lemma, which is an application of the Chernoff bound for Bernoulli random variables.

\section{Percolation}

We need the following result from Bernoulli percolation on $\Z^d$ (see Theorem 1.1 in \cite{MR1404543}).

\begin{theorem} \label{Thm: Perc_Distance}
Let $d \ge 2$ and $p > p_c^{\mathrm{site}}(d)$, where $p_c^{\mathrm{site}}(d)$ is the critical parameter of Bernoulli site percolation on $\Z^d$. Then, there exists a constant $\rho = \rho(p, d) \in [1, \infty)$ such that 
\begin{align*}
	&\limsup_{\|y\|_1 \rightarrow \infty} \frac{1}{\|y\|_1} \log \bP_p( o \leftrightarrow y, d(o, y) > \rho \|y\|_1) < 0, \\
    &\; \text{ and } \; \; \limsup_{\ell \rightarrow \infty} \frac{1}{\ell} \log \bP_p( o \leftrightarrow y, d(o, y) > \ell ) < 0,
\end{align*}
where $d(x, y)$ is the graph distance between $x$ and $y$ on the percolation clusters. 
\end{theorem}

The next collection of results is about the ``giant component'' of supercritical percolation clusters on $d$-dimensional tori with $d \ge 2$. 
% The constants below satisfy that $C(d, p) \rightarrow 0$, $\beta(d, p) \rightarrow \infty$, and $c(p) \rightarrow 1$ as $p \rightarrow 1$.

\begin{theorem} \label{thm: percolation_reference}
    Let $d\ge2$. For $m\ge1$ and $p\in(0,1)$, consider Bernoulli site percolation with parameter $p$ on $\T_m^d$, and let $\mathrm{GC}=\mathrm{GC}(m,p)$ denote its largest open cluster, with ties broken arbitrarily. 
    \begin{itemize} 
    \item[(i)] For every $p > p^{\mathrm{site}}_c(d)$, there exist constants $C=C(d,p)>0$ and $c=c(p)>0$ such that the following holds for all large $m$. Set
    \[
    n(m):=\left\lceil C(\log m)^{1/(d-1)}\right\rceil .
    \]
    Then, with probability at least $1-m^{-1}$, every box of side length $n(m)$ contains at least $c \, n(m)^d$ vertices of $\mathrm{GC}$. Moreover, we can choose $c(p)$ such that $c(p) \rightarrow 1$ as $p \rightarrow 1$.
    \item[(ii)] There exists $p_o(d) \in (0,1)$ such that, for every $p\in(p_o(d),1)$, there exists $\beta=\beta(d)>0$ such that, for every $m\ge1$ and every $U\subseteq\T_m^d$ 
    \[
    \mathbb P_p\bigl(U\cap\mathrm{GC}=\varnothing\bigr)
    \le
    \exp\left(-\beta |U|^{(d-1)/d}\right) + e^{-\beta m} \, .
    \]
    \item[(iii)] For every $p>p_c^{\mathrm{site}}(d)$ and every $\gamma>1$, there exists $C'=C'(d,p)>0$ such that, upon setting
    \[
    r(m):=\left\lceil(\log m)^\gamma\right\rceil ,
    \]
    we have
    \[
    \mathbb P_p\left(
    \mathrm{diam}(\mathrm{GC}\cap B)
    \le C'r(m)
    \ \text{for every box $B$ of side length $r(m)$}
    \right)
    \longrightarrow 1
    \]
    as $m\to\infty$. Here $\mathrm{diam}$ is measured in the intrinsic graph metric of $\mathrm{GC}$ on the full torus $\T_m^d$; thus, connecting paths may leave $B$.
    \end{itemize}
\end{theorem}

\begin{remark}
  For Part~(ii), we could only find relevant references for the corresponding result (without the term $e^{-\beta m}$ on the right-hand side) about the infinite cluster of Bernoulli bond percolation on $\mathbb{Z}^d$. Hence we derive Part~(ii) from this corresponding result.  The term $e^{-\beta m}$ in Part~(ii) is an artifact of our proof.    We expect that the proof of the result about bond percolation on $\mathbb{Z}^d$ can be adapted to apply directly to the giant component of site percolation on $\mathbb{T}_m^d$ for all $p>p_c^{\mathrm{site}}(d)$.
Since our application allows $p$ to be arbitrarily close to 1, we state Part~(ii) in this weaker form.
\end{remark}

\begin{proof}[Proof of Theorem~\ref{thm: percolation_reference}]
Part (i) is a simple consequence of Theorem~4 of \cite{MR1372330}. 

To prove part~(ii), we first observe from \cite{MR1428500} that for every $q \in (0, 1)$, there exists $p = p(q) \ge q$ such that Bernoulli site percolation on $\mathbb Z^d$
(respectively, on $\mathbb T_m^d$) with parameter $p$ stochastically dominates Bernoulli bond percolation on $\mathbb Z^d$
(respectively, on $\mathbb T_m^d$) with parameter $q$. Hence, it suffices to prove part~(ii) for Bernoulli bond percolation.
Theorem~1.2 of \cite{diskin2026supercriticalsharpnesspercolation}, which concerns Bernoulli bond percolation on $\Z^d$, implies that for every $p > p_c^{\mathrm{bond}}(d)$, there exists $c > 0$ such that for every finite set of vertices $S$,
\begin{align*}
    \mathbb P_p^{\, \Z^d, \mathrm{bond}} \left( \left\{S \leftrightarrow \infty \right\}^c \right) \le \exp\left(-c |S|^{\frac{d-1}{d}}\right).
\end{align*}
In the notation $p(q)$ (from the above comment about domination), we choose $p_o(d):=p(q_o(d))$ where $q_o(d)$ is so that
$$\theta^{\mathrm{bond}}(q_o(d)) \ge 3/4.$$

Consider the standard coupling of bond percolation on $\Z^d$ and $\T^d_m$, identifying the torus with the box $B_m = [0,m)^d \cap \Z^d$. Let $a = \lceil (3/4)^{1/d} \rceil$, and let $B_{am}$ denote the box of side length $am$ centered at $(m/2, \ldots, m/2)$. 

By the Pigeonhole Principle, there exists a box $B$ of side length $am$ for which $|U \cap B| \ge |U|/2$. By translating $U$, if necessary, we may assume that $B = B_{am}$. We may also assume that $U \subseteq B_{am}$; otherwise, we replace $U$ with $U \cap B_{am}$. 

Define the events
\begin{align*}
    A_1 &:= \Big\{ U \cap C_\infty \neq \varnothing \Big\}, \quad A_2 := \left\{ |C_\infty \cap B_{am}| > \frac34 |B_{am}| \right\}, \\
    A_3 &:= 
    \left\{ 
    \begin{array}{c}
        \text{for every pair } x,y \in B_{am} \text{ such that} \\
        x \xleftrightarrow{\Z^d} y, \text{ we have } x \xleftrightarrow{B_m} y
    \end{array}    
    \right\}.
\end{align*}
We claim that
\[
    A_1 \cap A_2 \cap A_3 \subseteq \{ U \cap \mathrm{GC} \neq \varnothing \}.
\]
Indeed, because
\[
    \frac34|B_{am}| \ge \frac{9}{16}|B_m| > \frac12 |B_m|,
\]
on the event $A_2 \cap A_3$, the set $C_\infty \cap B_{am}$ is contained in a single connected component of the torus with more than $|B_m|/2$ vertices. This component is necessarily the giant component $\mathrm{GC}$. Since $U \subseteq B_{am}$, the event $A_1$ then implies that $U \cap \mathrm{GC} \neq \varnothing$.

Thus, by a union bound, for all $p > p_o(d)$,
\begin{align*}
    \mathbb P_p \left( \mathrm{GC} \cap U = \varnothing \right) \le \, &\bP_p \left( U \not\leftrightarrow\infty \right) + \bP_p \left( |C_\infty \cap B_{am}| \le \frac34 |B_{am}| \right) + \bP_p \left( \, A_3^c  \,\right) \\
    \le \, & \exp(-c_1 |U|^{\frac{d-1}{d}}) + \exp(-c_2 m^{d-1}) + (am)^{2d} \exp(-c_3 (1{-}a) m), \\
    \le \, & \exp\left(-\beta|U|^{\frac{d-1}{d}}\right) + C(d) \exp(-\frac{c_3}{2} (1{-}a) m)
    % \le \, & \exp(-c_1 |U|^{\frac{d-1}{d}}) + \exp(-c_2 |U|^{\frac{d-1}{d}}) + C(d) \exp(-\frac{c_3}{2} (1{-}a) |U|^{\frac{d-1}{d}}) \\
    % \le  \, & \exp\left(-\beta|U|^{\frac{d-1}{d}}\right)
\end{align*}
where $c_1, c_2, c_3, \beta > 0$ depend only on $d$ and $p$.
In the second inequality, we use the supercritical sharpness result discussed above for the first term, Theorem~6 of \cite{MR1372330} for the second term, and Theorem~\ref{Thm: Perc_Distance}, together with a union bound over all pairs of vertices in $B_{am}$, for the third term. In the third inequality, we use only the fact that $U \subset \T^d_m$ and hence that $|U| \le m^d$.
% In the fourth inequality, the assumption that $U$ is contained in a box of side length at most $m^{1/(d-1)}$ implies that $|U| \le m^{d/(d-1)}$.

We now prove part (iii). 
For $x\in\mathbb Z^d$ and $r\geq 1$, write
\[
    Q_r(x):=\{y\in\mathbb Z^d:\|x-y\|_\infty\leq r\},
\]
and let $d_{\Lambda}(x,y)$ denote the chemical distance between $x$ and $y$
using only open paths contained in $\Lambda$.
We use the standard supercritical one-arm estimate (see, for example, Theorem~8.21 in \cite{Grimmett1999})
\begin{equation}\label{eq:truncated-one-arm}
    \mathbb P_p^{\, \Z^d, \mathrm{site}}\bigl(
        x\leftrightarrow \partial Q_r(x),\ |C_x|<\infty
    \bigr)
    \leq C_1 e^{-c_1r},
\end{equation}
We also use the uniqueness of the infinite cluster on $\mathbb Z^d$.
By Theorem~\ref{Thm: Perc_Distance}, we have
\begin{equation}\label{eq:chemical-distance}
    \mathbb P_p^{\, \Z^d, \mathrm{site}}\bigl(
        x\leftrightarrow y,\,
        d(x,y)>\rho\|x-y\|_1
    \bigr)
    \leq C_2 e^{-c_2\|x-y\|_1}.
\end{equation}

Let $r = r(m) =\left\lceil(\log m)^\gamma\right\rceil$,
choose $L>1+4d\rho$, and set $C':=7d\rho$. We first claim that,
uniformly over all $x,y\in\mathbb Z^d$ satisfying
$\|x-y\|_\infty\leq r$,
\begin{align}
    \mathbb P_p^{\, \Z^d, \mathrm{site}}\Bigl(&
        x\leftrightarrow \partial Q_{Lr}(x),\,
        y\leftrightarrow \partial Q_{Lr}(x),\,
        d_{Q_{Lr}(x)}(x,y)>C'r
    \Bigr)
    \leq C e^{-cr}.
    \label{eq:local-distance}
\end{align}
Indeed, by \eqref{eq:truncated-one-arm}, except on an event of probability at most
$2C_1e^{-c_1r}$, both $x$ and $y$ belong to the infinite cluster. Choose a
vertex $z$ in this cluster such that $\|z-x\|_\infty=3r$.
By uniqueness of the infinite cluster, $y\leftrightarrow z$. Moreover, we have $3r \leq \|z-x\|_1\leq 3dr$ and $2r \leq \|z-y\|_1\leq 4dr$. Applying \eqref{eq:chemical-distance} and taking a union bound over the
$O(r^{d-1})$ possible choices of $z$, we find that, except on an event of
probability at most $C_3e^{-c_3r}$, the inequalities $d(x,z)\leq 3d\rho r$ and $d(y,z)\leq 4d\rho r$ hold.
Since $L>1+4d\rho$, both paths are contained in $Q_{Lr}(x)$. Hence
\[
    d_{Q_{Lr}(x)}(x,y)
    \leq d(x,z)+d(z,y)
    \leq 7d\rho r=C'r,
\]
which proves \eqref{eq:local-distance}.
We use Theorem~4 of \cite{MR1372330} again: for every $p > p_c^{\mathrm{site}}(d)$, there exists
$a=a(p)>0$ such that $\mathbb P_p \bigl(|\mathrm{GC}|\geq a m^d\bigr) \rightarrow 1$.
Since $r=o(m)$, on this event and for all sufficiently large $m$, we have $|\mathrm{GC}|>|Q_{Lr}(x)|$ for every $x\in\mathbb T_m^d$. Consequently, if
$x,y\in\mathrm{GC}$ and $\|x-y\|_\infty\leq r$, then both $x$ and $y$
are connected to the boundary of $Q_{Lr}(x)$ by paths contained in $Q_{Lr}(x)$.
Therefore, by \eqref{eq:local-distance},
\[
    \mathbb P_p\left(
        \exists\,x,y\in\mathrm{GC}:\,
        \|x-y\|_\infty\leq r,\,
        d(x,y)>C'r
    \right)
    \leq \mathbb P_p\bigl(|\mathrm{GC}| < a m^d\bigr)+C m^d r^d e^{-cr}.
\]
Since $\alpha>1$ and $r=\lceil(\log m)^\gamma\rceil$, the right-hand
side tends to zero.
\end{proof}

\section{Proof of Lemma~\ref{lemma: Delta_g_1/d-1}} \label{apx_section: proof_of_lemma_Delta_g_1/d-1}

In this section we prove Lemma~\ref{lemma: Delta_g_1/d-1}, which is restated as follows:

\begin{lemma}
Let $d\ge 2$ and $0<\alpha<d$.
Then
\[
    \Delta>
    \begin{cases}
        \dfrac{2}{d-1}, & \alpha\ge 2,\\[1ex]
        \dfrac{\alpha}{d-1}, & \alpha<2.
    \end{cases}
\]
\end{lemma}

\begin{proof}
Set $\delta:=\Delta^{-1} =\log_2({2d}/{(d+\alpha)})$.
Since $0<\alpha<d$, we have $0<\delta<1$.

Suppose first that $\alpha\ge 2$. Since $\alpha<d$, we have $d\ge 3$,
and hence $\delta<1\le \frac{d-1}{2}$.
Therefore, $\Delta>2/(d-1)$.

Now suppose that $\alpha<2$. If $\alpha\le d-1$, then $\delta<1\le {(d-1)}/{\alpha}$,
and hence $\Delta>\alpha/(d-1)$. It remains to consider
$\alpha>d-1$. Since $d\ge2$ and $\alpha<2$, this forces $d=2$ and
$1<\alpha<2$. In this case,
\[
    \delta
    =\log_2\left(\frac{4}{2+\alpha}\right)
    <\log_2\left(\frac43\right)
    <\frac12
    <\frac1\alpha
    =\frac{d-1}{\alpha},
\]
where $\log_2(4/3)<1/2$ follows from $4/3<\sqrt2$. Thus,
$\Delta>\alpha/(d-1)$ also in this case.
\end{proof}

\bibliographystyle{plain} 
\bibliography{ref.bib} 

\end{document}